\documentclass[nosumlimits,twoside]{amsart}
\usepackage{amsfonts, amsmath, amssymb}
\usepackage{graphicx}
\usepackage{float}
\usepackage{srcltx}
\usepackage[all]{xy}
\usepackage{version}
\usepackage[T1]{fontenc}
\usepackage{pifont}
\usepackage[dvipsnames,table]{xcolor}
\usepackage[final]{showlabels}% put "final" to turns off; "inline" otherwise

\usepackage{enumerate}
\usepackage[normalem]{ulem}
\usepackage{bm}
\usepackage{latexsym}
\usepackage[2emode]{psfrag}
\usepackage{yhmath}
\usepackage{array}
\usepackage{dsfont}
\usepackage{mathrsfs}% for \mathscr
\usepackage{tikz-cd}% diagrams
\usepackage{comment}
\usepackage[colorinlistoftodos,prependcaption,textsize=tiny]{todonotes}
\usepackage{mathtools}

\newtheorem{theorem}{\sc Theorem}[section]
\newtheorem{proposition}[theorem]{\sc Proposition}

\newtheorem{lemma}[theorem]{\sc Lemma}
\newtheorem{corollary}[theorem]{\sc Corollary}
\newtheorem{question}[theorem]{\sc Question}
\theoremstyle{definition}
\newtheorem{definition}[theorem]{\sc Definition}

\newtheorem{example}[theorem]{\sc Example}

\theoremstyle{remark}
\newtheorem{remark}[theorem]{\sc Remark}

\newenvironment{invisible}{{\noindent\sc \colorbox{yellow}{Invisible:}\;}\color{gray}}{\medskip}
\excludeversion{invisible}% comment this to see the invisible part.

\allowdisplaybreaks

\excludeversion{proof?}

\newcommand{\Cc}{\mathcal{C}}

\newcommand{\Mm}{\mathcal{M}}
\newcommand{\mm}{\mathfrak{M}}

\newcommand{\Ff}{\mathcal{F}}

\newcommand{\Rr}{\mathcal{R}}

\newcommand{\ot}{\otimes}

\def\Vec{{\sf Vec}}

\newcommand{\rd}[1]{{\color{blue}{#1}}}

\newcommand{\unit}{\mathds{I}}

\newcommand{\op}{\mathrm{op}}

\begin{document}
\title[Infinitesimal braidings and pre-Cartier bialgebras]{Infinitesimal braidings and pre-Cartier bialgebras}


%\thanks{???}

\begin{abstract}
We propose an infinitesimal counterpart to the notion of braided category. The corresponding infinitesimal braidings are natural transformations which are compatible with an underlying braided monoidal structure in the sense that they constitute a first-order deformation of the braiding. This extends previously considered infinitesimal symmetric or Cartier categories, where involutivity of the braiding and an additional commutativity of the infinitesimal braiding with the symmetry are required. The generalized pre-Cartier framework is then elaborated in detail for the categories of (co)quasitriangular bialgebra (co)modules and we characterize the resulting infinitesimal $\mathcal{R}$-matrices (resp. $\mathcal{R}$-forms) on the bialgebra. It is proven that the latter are Hochschild 2-cocycles and that they satisfy an infinitesimal quantum Yang-Baxter equation, while they are Hochschild 2-coboundaries under the Cartier (co)triangular assumption in the presence of an antipode. We provide explicit examples of infinitesimal braidings, particularly on quantum $2\times 2$-matrices, $\mathrm{GL}_q(2)$, Sweedler’s Hopf algebra and via Drinfel’d twist deformation. As conceptual tools to produce examples of infinitesimal braidings we prove an infinitesimal version of the FRT construction and we provide a Tannaka-Krein reconstruction theorem for pre-Cartier coquasitriangular bialgebras. We comment on the deformation of infinitesimal braidings and construct a quasitriangular structure on formal power series of Sweedler’s Hopf algebra.
\end{abstract}

\author{Alessandro Ardizzoni}
\email{alessandro.ardizzoni@unito.it}
\urladdr{\url{https://sites.google.com/site/aleardizzonihome}}
\author{Lucrezia Bottegoni}
\email{lucrezia.bottegoni@edu.unito.it}
\author{Andrea Sciandra}
\email{andrea.sciandra@unito.it}
\author{Thomas Weber}
\email{thomas.weber@unito.it}
\address{\parbox[b]{\linewidth}{University of Turin, Department of Mathematics ``G. Peano'', via
Carlo Alberto 10, I-10123 Torino, Italy}}
\date{\today}

\maketitle

\tableofcontents

\section{Introduction}

Braided monoidal categories were first introduced in \cite{JoyalStreet} as tensor categories $(\mathcal{C},\otimes,\unit)$ endowed with a categorification of the flip operator, the braiding $\sigma$. On objects $X,Y$ of $\mathcal{C}$ the braiding is a morphism $\sigma_{X,Y}\colon X\otimes Y\to Y\otimes X$ which is compatible with the tensor product in the obvious way and, as a result, $\sigma_{X,X}$ satisfies the braid relations. For this reason braided monoidal categories have found numerous applications, particularly in knot theory, leading to topological invariants of $3$-manifolds \cite{ResTur}. Most relevant for our purpose, braided monoidal categories appear as the categorical counterpart of (co)quasitriangular bialgebras. The latter are bialgebras $H$ together with a solution of the quantum Yang-Baxter equation, the universal $\Rr$-matrix (or universal $\Rr$-form), originating in the seminal work \cite{Dr87} of Drinfel'd and expanding on earlier ideas \cite{FadTak} on the quantum inverse scattering method.
Given a (co)quasitriangular bialgebra $(H,\Rr)$, the monoidal category $(\mathcal{C},\otimes,\unit)$ of bialgebra (co)modules becomes braided with $\sigma$ given by the (co)action of $\Rr$ composed with the tensor flip \cite{Dr90,Majid-notes}. On the other hand, a braided monoidal category $(\mathcal{C},\otimes,\unit)$ satisfying certain representability conditions can be understood as the (co)representation category of a ``reconstructed'' (co)quasitriangular bialgebra \cite{Majid91}. This is referred to as Tannaka-Krein reconstruction of (co)quasitriangular bialgebras. Thus, morally we can think of braided monoidal categories and (co)quasitriangular bialgebras as two sides of the same coin.

The scope of this paper is to provide an infinitesimal version of braided monoidal category and to study the corresponding algebraic counterparts. In fact, the motivating idea is to start in the algebraic framework and consider a bialgebra $H$ with its corresponding trivial topological bialgebra $\tilde{H}=H[[\hbar]]$ of formal power series in a formal parameter $\hbar$ given by naive $\hbar$-linear extension of the bialgebra structure and replacing the tensor product with its topological completion. Then, given a (co)quasitriangular structure 
\begin{equation}\label{RTilde}
    \tilde{\Rr}=\Rr(1\otimes 1+\hbar\chi+\mathcal{O}(\hbar^2))
\end{equation}
on $\tilde{H}$ we obtain a (co)quasitriangular bialgebra $(H,\Rr)$ and ``read'' or ``declare'' the axioms of our infinitesimal $\Rr$-matrix $\chi$ as the axioms of $\tilde{\Rr}$ in first order of $\hbar$. This definition of infinitesimal $\Rr$-matrix can then be applied to any (co)quasitriangular bialgebra $(H,\Rr)$ independently of an imagined (co)quasitriangular structure \eqref{RTilde} on $\tilde{H}$. Reversing back to category theory we obtain a braided monoidal category $(\mathcal{C},\otimes,\unit,\sigma)$ with a natural transformation $t\colon\otimes\to\otimes$ now corresponding to $\chi$. The axioms of $\chi$ manifest in a certain compatibility of $t$ with $(\otimes,\sigma)$. It turns out that the resulting categories generalize Cartier categories, which already appeared in \cite{HV} and trace back to \cite{Ca93}. However, our definition is not bound to symmetric braidings (i.e., $\sigma^2=\mathrm{Id}$) and our proposed compatibilities of $t$ and $(\otimes,\sigma)$ are less restrictive. For this reason we call the categories of this paper pre-Cartier braided, while we refer to the ones studied in \cite{HV} as Cartier symmetric categories. The corresponding bialgebras are called pre-Cartier (co)(quasi)triangular. 
The idea to consider infinitesimal versions of (co)quasitriangular structures goes back to \cite{Dr83}, however exclusively in the framework of quasitriangular Lie bialgebras $(\mathfrak{g},[\cdot,\cdot],r)$ which correspond to quasitriangular bialgebras $(U\mathfrak{g},\Rr)$. Here, the classical $r$-matrix $r\in\mathfrak{g}\otimes\mathfrak{g}$ (or quasitriangular Lie bialgebra structure in general), a solution of the classical Yang-Baxter equation, is the infinitesimal counterpart of a quasitriangular structure $\tilde{\Rr}=1\otimes 1+\hbar r+\mathcal{O}(\hbar)$ on $U\mathfrak{g}[[\hbar]]$. Afterwards, in \cite{Majid00}, the symmetrization $\frac{1}{2}(r+r^{\mathrm{op}})$ of a quasitriangular Lie bialgebra structure has been used to obtain an ``infinitesimal braiding'', a tool to construct braided Lie bialgebras.
To our knowledge the concept of infinitesimal braiding has so far only been addressed in the previously mentioned context of Lie bialgebras $(\mathfrak{g},[\cdot,\cdot],r)$ with corresponding bialgebra $(U\mathfrak{g},\Rr=1\otimes 1)$ and in the context of Cartier triangular bialgebras \cite{HV} (in the dual case it was considered also for comodules over an abelian monoid algebra, still in \cite{HV}). Our approach goes genuinely beyond the established theory. The main results include a Tannaka-Krein reconstruction theorem for pre-Cartier coquasitriangular bialgebras and an FRT construction of pre-Cartier categories. The latter extends to well-known construction \cite{FRT} of coquasitriangular bialgebras from a solution of the braid relation to pre-Cartier coquasitriangular bialgebras, where additional infinitesimal braid relations, involving a second endomorphism, have to be satisfied. We further prove that an infinitesimal $\Rr$-matrix is a Hochschild $2$-cocycle and satisfies the ``infinitesimal quantum Yang-Baxter equation'', while in the Cartier symmetric Hopf algebra case it corresponds to a Hochschild $2$-coboundary. Accompanying the definitions are explicit examples of infinitesimal braidings, including on quantum $2\times 2$-matrices, $\mathrm{GL}_q(2)$ and Sweedler's Hopf algebra. As infinitesimal objects, pre-Cartier bialgebras naturally pose a quantization problem. We comment on the deformation of infinitesimal braidings, giving an explicit solution for Sweedler's Hopf algebra and further relating to Etingof-Kazhdan quantization of Lie bialgebras.\medskip

The paper is organized as follows. After the introduction we proceed by setting up the categorical framework of pre-Cartier braided categories. It includes a result on how to transfer infinitesimal braidings via a strong monoidal functor. The rest of the article will follow the outlined categorical setup, but the reader who is mainly interested in the results on Hopf algebras and their representations can skip the categorical build-up and understand the following sections independently. Section~\ref{Sec:2} treats quasitriangular bialgebras $(H,\Rr)$ endowed with an infinitesimal $\Rr$-matrix. We show that these are characterized by braided monoidal category of $H$-modules being pre-Cartier. 
 Several examples are provided, including infinitesimal $\Rr$-matrices arising from quasitriangular Lie bialgebras and a classification of infinitesimal $\Rr$-matrices on Sweedler's Hopf algebra. We further discuss the previously mentioned example of quasitriangular topological bialgebras in detail, which will motivate us to formulate a quantization problem. We devote Section~\ref{Sec:induced} to induce pre-Cartier quasitriangular structures and Section~\ref{secTwist} for Drinfel'd twist deformation of them. In Section~\ref{secCohom} we recall Hochschild cohomology theory for  coalgebras and prove the previously mentioned results about the infinitesimal $\Rr$-matrix being a $2$-cocycle. Section~\ref{sec:dual} is the dual counterpart of Section~\ref{Sec:2}. It mirrors the concepts and constructions for the corepresentation category of a coquasitriangular bialgebra. We combine both pictures in Section~\ref{Sec:FinDual} about the finite dual. The final Section~\ref{Sec:constructions} comprises the aforementioned FRT construction (Section~\ref{Sec:FRT}) and Tannaka-Krein reconstruction (Section~\ref{Sec:Tannaka}) of pre-Cartier coquasitriangular bialgebras.

\subsection{Preliminaries and notations}
Given an object $X$ in a category $\Cc$, the identity morphism on $X$ will be denoted either by $\mathrm{Id}_{X}$ or $X$ for short. By $(\Mm, \otimes, \unit)$ we denote a monoidal category, where $\otimes\colon\Mm\times\Mm\rightarrow\Mm$ is the tensor product functor and $\unit$ is the unit object. The associativity constraint $a$ and the left and right unit constraints $l,r$ will be usually omitted when clear from the context. If $\Mm$ is braided, we denote the braiding by $\sigma$. A \textit{strong monoidal} functor $(F,\phi^0,\phi^2):(\Mm,\otimes, \unit)\to(\Mm', \otimes',\unit')$ consists of a functor $F:\Mm\to\Mm'$, an isomorphism $\phi^2_{X,Y}:F(X)\otimes'F(Y)\to F(X\otimes Y)$ in $\Mm'$ which is natural in $X,Y$, for every $X$ and $Y$ in $\Mm$, and an isomorphism $\phi^0 : \unit'\to F(\unit)$ in $\Mm'$ such that the following identities hold true
\begin{gather*}
F(a_{X,Y,Z})\circ\phi^2_{X\otimes Y,Z}\circ (\phi^2_{X,Y}\otimes' F(Z))=\phi^2_{X,Y\otimes Z}\circ (F(X)\otimes' \phi^2_{Y,Z})\circ a'_{F(X),F(Y),F(Z)}\\  
F(l_X)\circ \phi^2_{\unit,X}\circ (\phi^0\otimes' F(X))=l'_{F(X)}, \quad F(r_X)\circ \phi^2_{X,\unit}\circ (F(X)\otimes' \phi^0)=r'_{F(X)} 
\end{gather*}
where $a$, $a'$ are the associativity constraints, and $l, r, l', r'$ are the unit constraints. If $\Mm$ and $\Mm'$ are braided, a \textit{braided strong monoidal} functor $F:\Mm\to\Mm'$ is a strong monoidal functor such that
$F(\sigma_{X,Y})\circ\phi^2_{X,Y}=\phi^2_{Y,X}\circ\sigma'_{F(X),F(Y)}$
where $\sigma$ and $\sigma'$ are the respective braidings.

All vector spaces are understood to be $\Bbbk$-vector spaces, where $\Bbbk$ is an arbitrary field. By a linear map we mean a $\Bbbk$-linear map and the unadorned tensor product $\otimes$ stands for $\otimes_{\Bbbk}$. We denote the canonical flip by $\tau_{M,N}:M\otimes N\to N\otimes M$, $m\otimes n\mapsto n\otimes m$, for vector spaces $M,N$, or simply by $\tau$ for sake of briefness. By ring we mean an associative ring with identity.  For an algebra we denote the multiplication and the unit by $m$ and $u$, respectively, while for a coalgebra the comultiplication and the counit are denoted by $\Delta$ and $\varepsilon$, respectively. We write $H$ for a bialgebra over $\Bbbk$. We will use the classical Sweedler’s notation for calculations with the comultiplication and we shall write $\Delta(h)= h_{1}\otimes h_{2}$ for any $h\in H$, where we omit the summation symbol. In case $H$ admits an antipode we denote it by $S:H\to H$. The center of an algebra $A$ is denoted by $\mathscr{Z}(A)$. For any right (left) $H$-module we denote without distinction the action by $\cdot$, even when there are different actions involved. For a right (left) $H$-comodule $M$ we write the coaction as $\rho^r: M\to M\otimes H$, $\rho^r(m)=m_0\otimes m_1$, (resp. $\rho^l:M\to H\otimes M$, $\rho^l(m)= m_{-1}\otimes m_0$), for any $m\in M$.

\subsection{Pre-Cartier categories}

In this section we introduce the notion of \textit{pre-Cartier} category. For this aim, recall that a category is \textit{pre-additive} when the set of morphisms between any two objects is an abelian group and the composition of morphisms satisfies the distributive law. A pre-additive braided monoidal category is a braided monoidal category which is also pre-additive and the tensor functor is additive in each entry, i.e., the tensor product and the addition of morphisms satisfy the distributive laws.

\begin{definition}\label{def:qC}
A  \textbf{pre-Cartier} category is a  pre-additive braided monoidal category $(\Mm,\otimes,\sigma)$ together with a natural transformation $t\colon\otimes\to\otimes$ such that the identities 
\begin{align}
    t_{X,Y\otimes Z}=t_{X,Y}\otimes\mathrm{Id}_{Z}+(\sigma^{-1}_{X,Y}\otimes\mathrm{Id}_{Z})\circ(\mathrm{Id}_{Y}\otimes t_{X,Z})\circ(\sigma_{X,Y}\otimes\mathrm{Id}_{Z}), \label{qC-I}\\
    t_{X\otimes Y,Z}=\mathrm{Id}_{X}\otimes t_{Y,Z}+(\mathrm{Id}_{X}\otimes\sigma^{-1}_{Y,Z})\circ(t_{X,Z}\otimes\mathrm{Id}_{Y})\circ(\mathrm{Id}_{X}\otimes\sigma_{Y,Z}), \label{qC-II}
\end{align}
 hold true for all objects $X,Y,Z$ in $\Mm$. In this case we say that $t$ is an \textbf{infinitesimal braiding} of $(\Mm,\otimes,\sigma)$.
A pre-Cartier category $(\mathcal{M},\otimes,\sigma,t)$ is called \textbf{Cartier} if in addition
\begin{equation}\label{cart}
    \sigma_{X,Y}\circ t_{X,Y}=t_{Y,X}\circ\sigma_{X,Y}
\end{equation}
holds for all objects $X,Y$ in $\mathcal{M}$.
In case the braiding is symmetric we refer to the above as \textbf{symmetric (pre-)Cartier} categories.
We summarise these definitions in the following table (where \ding{55} indicates that the corresponding condition is not required):\medskip
\begin{center}
   \begin{tabular}{|c|c|c|}
    \hline\cellcolor{gray!10} Terminology  &\cellcolor{gray!10} $\sigma_{X,Y}^{-1}=\sigma_{Y,X}$ &\cellcolor{gray!10}$\sigma_{X,Y}t_{X,Y}=t_{Y,X}\sigma_{X,Y}$ \\\hline\hline
  braided pre-Cartier category  & \ding{55} & \ding{55} \\\hline
   symmetric pre-Cartier category  & \ding{51}  & \ding{55}\\\hline
  braided Cartier category & \ding{55} & \ding{51} \\\hline
  symmetric Cartier category & \ding{51} & \ding{51} \\\hline
   \end{tabular}
 \end{center}
 \medskip
 We will sometimes use the term \textbf{Cartier type category} to refer to any of the above notions.
\end{definition}
\begin{remark} We collect here immediate consequences of the definition. The first two are  straightforward extensions of \cite[Remark 2.2]{HV} from symmetric Cartier to braided pre-Cartier categories.
\begin{enumerate}[$i)$]
\item Every pre-additive braided monoidal category is Cartier with respect to the trivial infinitesimal braiding $t_{X,Y}=0$ for all objects $X,Y$.

\item Setting $Y=\unit$ in \eqref{qC-I} and \eqref{qC-II}
implies $t_{X,\unit}=0$ and $t_{\unit,X}=0$, respectively.

\item For a symmetric monoidal category with a natural transformation $t\colon\otimes\to\otimes$, in presence of \eqref{cart}, the conditions \eqref{qC-I} and \eqref{qC-II} are equivalent as shown in \cite[XX.4]{Kassel}. Consequently, the notion of symmetric Cartier category in Definition~\ref{def:qC} coincides with that of Cartier category as introduced in \cite{HV}, which goes back to \cite{Ca93}. Note that in \cite[XX.4]{Kassel} such a category is called an \emph{infinitesimal symmetric category} in case of a strict tensor category.
\end{enumerate}

\end{remark}
%Consider a quasi-Cartier category $(\Mm,\otimes,\sigma,t)$ and let $(\Mm',\otimes')$ be a monoidal category equivalent to $(\Mm,\otimes)$ with monoidal functor $F\colon\Mm\to\Mm'$. We denote the induced braiding of $(\Mm',\otimes')$ by $\sigma'$. It is determined by the commutative diagram
%\begin{equation}
%\begin{tikzcd}
%F(M)\otimes' F(N) \arrow{rr}{\sigma'_{F(M),F(N)}}
%\arrow{d}[swap]{\cong}
%& & F(N)\otimes'F(M)\\
%F(M\otimes N) \arrow{rr}{F(\sigma_{M,N})}
%& & F(N\otimes M) \arrow{u}[swap]{\cong}
%\end{tikzcd}
%\end{equation}
%for all objects $M,N$ in $\Mm$. Similarly, we induce an infinitesimal braiding $t'$ of the braided monoidal category $(\Mm',\otimes',\sigma')$ by
%\begin{equation}
%\begin{tikzcd}
%F(M)\otimes' F(N) \arrow{rr}{t'_{F(M),F(N)}}
%\arrow{d}[swap]{\cong}
%& & F(M)\otimes'F(N)\\
%F(M\otimes N) \arrow{rr}{F(t_{M,N})}
%& & F(M\otimes N) \arrow{u}[swap]{\cong}
%\end{tikzcd}
%\end{equation}
%and it is straightforward to verify that $(\Mm',\otimes',\sigma',t')$ is quasi-Cartier.

The next result deals with transferring the Cartier type structures from one category to another. 

\begin{proposition}\label{prop:ff}
Let $\left( F,\phi ^{0},\phi ^{2}\right) \colon \Mm
\rightarrow \Mm^{\prime }$ be a strong monoidal functor between pre-additive categories $\Mm$ and $\Mm'$. Assume $F$ is additive and fully faithful (e.g., $F$ is an equivalence and the categories $\Mm$ and $\Mm'$ are additive).
\begin{enumerate}
  \item [i)] If $\Mm^{\prime }$ has a braiding (symmetry), then so does $\Mm$ and $F$ is braided strong monoidal.
  \item [ii)] If $\Mm^{\prime }$ has a natural transformation fulfilling \eqref{qC-I} and \eqref{qC-II}, then so does $\mathcal{M}$.
  \item [iii)]If $\Mm^{\prime }$ has a natural transformation  fulfilling \eqref{cart}, then so does $\mathcal{M}$.
\end{enumerate}
\end{proposition}

\begin{proof}
The proof is straightforward and, as far as  the braiding is concerned, traces back to \cite[4.4.3]{Saavedra-Rivano72}.
We just point out that $\sigma$ and $t$ for $\Mm$ are uniquely defined in terms of $\sigma'$ and $t'$ for $\Mm'$ by the commutativity of the following diagrams, for all objects $X,Y$ in $\Mm$.
\begin{equation*}
\begin{tikzcd} F(X)\otimes' F(Y) \arrow{rr}{\sigma'_{F(X),F(Y)}}
\arrow{d}[swap]{\phi_{X,Y} ^{2}} & & F(Y)\otimes'F(X)\arrow{d}[swap]{\phi_{Y,X} ^{2}}\\ F(X\otimes Y)
\arrow{rr}{F(\sigma_{X,Y})} & & F(Y\otimes X) 
\end{tikzcd}
\qquad
\begin{tikzcd} F(X)\otimes' F(Y) \arrow{rr}{t'_{F(X),F(Y)}}
\arrow{d}[swap]{\phi_{X,Y} ^{2}} & & F(X)\otimes'F(Y)\arrow{d}[swap]{\phi_{X,Y} ^{2}}\\ F(X\otimes Y)
\arrow{rr}{F(t_{X,Y})} & & F(X\otimes Y)  \end{tikzcd}
\end{equation*}%
Note also that if $F$ is an equivalence, then it is automatically fully faithful and, being part of an adjunction, it is also additive if $\Mm$ and $\Mm'$ are additive, see e.g. \cite[Corollary 1.3]{Popescu73}.
\begin{invisible}
Denote by $\sigma ^{\prime }$ the braiding of $\mathcal{M}^{\prime }$. Since
$F$ is fully faithful, there is a unique morphism $\sigma _{X,Y}:X\otimes
Y\rightarrow Y\otimes X$ such that $F\left( \sigma _{X,Y}\right) =\phi
_{Y,X}^{2}\sigma _{F\left( X\right) ,F\left( Y\right) }^{\prime }\left( \phi
_{X,Y}^{2}\right) ^{-1}.$ Set
\begin{equation*}
\phi _{X,Y,Z}^{3}:=\phi _{X,Y\otimes Z}^{2}\left( F\left( X\right) \otimes
\phi _{Y,Z}^{2}\right) =\phi _{X\otimes Y,Z}^{2}\left( \phi
_{X,Y}^{2}\otimes F\left( Z\right) \right) .
\end{equation*}%
We have%
\begin{eqnarray*}
F\left( \sigma _{X,Y}\otimes Z\right) &=&\phi _{Y\otimes X,Z}^{2}\left(
F\left( \sigma _{X,Y}\right) \otimes F\left( Z\right) \right) \left( \phi
_{X\otimes Y,Z}^{2}\right) ^{-1} \\
&=&\phi _{Y\otimes X,Z}^{2}\left( \phi _{Y,X}^{2}\otimes F\left( Z\right)
\right) \left( \sigma _{F\left( X\right) ,F\left( Y\right) }^{\prime
}\otimes F\left( Z\right) \right) \left( \left( \phi _{X,Y}^{2}\right)
^{-1}\otimes F\left( Z\right) \right) \left( \phi _{X\otimes Y,Z}^{2}\right)
^{-1} \\
&=&\phi _{Y,X,Z}^{3}\left( \sigma _{F\left( X\right) ,F\left( Y\right)
}^{\prime }\otimes F\left( Z\right) \right) \left( \phi _{X,Y,Z}^{3}\right)
^{-1}
\end{eqnarray*}%
and%
\begin{eqnarray*}
F\left( X\otimes \sigma _{Y,Z}\right) &=&\phi _{X,Z\otimes Y}^{2}\left(
F\left( X\right) \otimes F\left( \sigma _{Y,Z}\right) \right) \left( \phi
_{X,Y\otimes Z}^{2}\right) ^{-1} \\
&=&\phi _{X,Z\otimes Y}^{2}\left( F\left( X\right) \otimes \phi
_{Z,Y}^{2}\right) \left( F\left( X\right) \otimes \sigma _{F\left( Y\right)
,F\left( Z\right) }^{\prime }\right) \left( F\left( X\right) \otimes \left(
\phi _{Y,Z}^{2}\right) ^{-1}\right) \left( \phi _{X,Y\otimes Z}^{2}\right)
^{-1} \\
&=&\phi _{X,Z,Y}^{3}\left( F\left( X\right) \otimes \sigma _{F\left(
Y\right) ,F\left( Z\right) }^{\prime }\right) \left( \phi
_{X,Y,Z}^{3}\right) ^{-1}.
\end{eqnarray*}%
Thus%
\begin{eqnarray*}
F\left( \sigma _{X,Y}\otimes Z\right) &=&\phi _{Y,X,Z}^{3}\left( \sigma
_{F\left( X\right) ,F\left( Y\right) }^{\prime }\otimes F\left( Z\right)
\right) \left( \phi _{X,Y,Z}^{3}\right) ^{-1} \\
F\left( X\otimes \sigma _{Y,Z}\right) &=&\phi _{X,Z,Y}^{3}\left( F\left(
X\right) \otimes \sigma _{F\left( Y\right) ,F\left( Z\right) }^{\prime
}\right) \left( \phi _{X,Y,Z}^{3}\right) ^{-1}.
\end{eqnarray*}%
Hence%
\begin{eqnarray*}
&&F\left( \sigma _{X,Z}\otimes Y\right) F\left( X\otimes \sigma _{Y,Z}\right)
\\
&=&\phi _{Z,X,Y}^{3}\left( \sigma _{F\left( X\right) ,F\left( Z\right)
}^{\prime }\otimes F\left( Y\right) \right) \left( \phi _{X,Z,Y}^{3}\right)
^{-1}\phi _{X,Z,Y}^{3}\left( F\left( X\right) \otimes \sigma _{F\left(
Y\right) ,F\left( Z\right) }^{\prime }\right) \left( \phi
_{X,Y,Z}^{3}\right) ^{-1} \\
&=&\phi _{Z,X,Y}^{3}\left( \sigma _{F\left( X\right) ,F\left( Z\right)
}^{\prime }\otimes F\left( Y\right) \right) \left( F\left( X\right) \otimes
\sigma _{F\left( Y\right) ,F\left( Z\right) }^{\prime }\right) \left( \phi
_{X,Y,Z}^{3}\right) ^{-1} \\
&=&\phi _{Z,X,Y}^{3}\sigma _{F\left( X\right) \otimes F\left( Y\right)
,F\left( Z\right) }^{\prime }\left( \phi _{X,Y,Z}^{3}\right) ^{-1} \\
&=&\phi _{Z,X,Y}^{3}\sigma _{F\left( X\right) \otimes F\left( Y\right)
,F\left( Z\right) }^{\prime }\left( \left( \phi _{X,Y}^{2}\right)
^{-1}\otimes F\left( Z\right) \right) \left( \phi _{X\otimes Y,Z}^{2}\right)
^{-1} \\
&=&\phi _{Z,X,Y}^{3}\left( F\left( Z\right) \otimes \left( \phi
_{X,Y}^{2}\right) ^{-1}\right) \sigma _{F\left( X\otimes Y\right) ,F\left(
Z\right) }^{\prime }\left( \phi _{X\otimes Y,Z}^{2}\right) ^{-1} \\
&=&\phi _{Z,X\otimes Y}^{2}\sigma _{F\left( X\otimes Y\right) ,F\left(
Z\right) }^{\prime }\left( \phi _{X\otimes Y,Z}^{2}\right) ^{-1}=F\left(
\sigma _{X\otimes Y,Z}\right)
\end{eqnarray*}%
so that $F\left( \sigma _{X,Z}\otimes Y\right) F\left( X\otimes \sigma
_{Y,Z}\right) =F\left( \sigma _{X\otimes Y,Z}\right) $ and hence $\left(
\sigma _{X,Z}\otimes Y\right) \left( X\otimes \sigma _{Y,Z}\right) =\sigma
_{X\otimes Y,Z}.$

Similarly
\begin{eqnarray*}
&&F\left( Y\otimes \sigma _{X,Z}\right) F\left( \sigma _{X,Y}\otimes Z\right)
\\
&=&\phi _{Y,Z,X}^{3}\left( F\left( Y\right) \otimes \sigma _{F\left(
X\right) ,F\left( Z\right) }^{\prime }\right) \left( \phi
_{Y,X,Z}^{3}\right) ^{-1}\phi _{Y,X,Z}^{3}\left( \sigma _{F\left( X\right)
,F\left( Y\right) }^{\prime }\otimes F\left( Z\right) \right) \left( \phi
_{X,Y,Z}^{3}\right) ^{-1} \\
&=&\phi _{Y,Z,X}^{3}\left( F\left( Y\right) \otimes \sigma _{F\left(
X\right) ,F\left( Z\right) }^{\prime }\right) \left( \sigma _{F\left(
X\right) ,F\left( Y\right) }^{\prime }\otimes F\left( Z\right) \right)
\left( \phi _{X,Y,Z}^{3}\right) ^{-1} \\
&=&\phi _{Y,Z,X}^{3}\sigma _{F\left( X\right) ,F\left( Y\right) \otimes
F\left( Z\right) }^{\prime }\left( \phi _{X,Y,Z}^{3}\right) ^{-1} \\
&=&\phi _{Y,Z,X}^{3}\sigma _{F\left( X\right) ,F\left( Y\right) \otimes
F\left( Z\right) }^{\prime }\left( F\left( X\right) \otimes \left( \phi
_{Y,Z}^{2}\right) ^{-1}\right) \left( \phi _{X,Y\otimes Z}^{2}\right) ^{-1}
\\
&=&\phi _{Y,Z,X}^{3}\left( \left( \phi _{Y,Z}^{2}\right) ^{-1}\otimes
F\left( X\right) \right) \sigma _{F\left( X\right) ,F\left( Y\otimes
Z\right) }^{\prime }\left( \phi _{X,Y\otimes Z}^{2}\right) ^{-1} \\
&=&\phi _{Y\otimes Z,X}^{2}\sigma _{F\left( X\right) ,F\left( Y\otimes
Z\right) }^{\prime }\left( \phi _{X,Y\otimes Z}^{2}\right) ^{-1}=F\left(
\sigma _{X,Y\otimes Z}\right)
\end{eqnarray*}%
so that $\left( Y\otimes \sigma _{X,Z}\right) \left( \sigma _{X,Y}\otimes
Z\right) =\sigma _{X,Y\otimes Z}.$

If $\sigma _{X,Y}^{\prime }$ is a symmetry, i.e., $\left( \sigma
_{X,Y}^{\prime }\right) ^{-1}=\sigma _{Y,X}^{\prime },$ then%
\begin{eqnarray*}
F\left( \sigma _{X,Y}^{-1}\right)  &=&F\left( \sigma _{X,Y}\right)
^{-1}=\left( \phi _{Y,X}^{2}\sigma _{F\left( X\right) ,F\left( Y\right)
}^{\prime }\left( \phi _{X,Y}^{2}\right) ^{-1}\right) ^{-1} \\
&=&\phi _{X,Y}^{2}\left( \sigma _{F\left( X\right) ,F\left( Y\right)
}^{\prime }\right) ^{-1}\left( \phi _{Y,X}^{2}\right) ^{-1} \\
&=&\phi _{X,Y}^{2}\sigma _{F\left( Y\right) ,F\left( X\right) }^{\prime
}\left( \phi _{Y,X}^{2}\right) ^{-1}=F\left( \sigma _{Y,X}\right)
\end{eqnarray*}%
and hence $\sigma _{X,Y}^{-1}=\sigma _{Y,X}$. Define now $t$ by setting $%
F\left( t_{X,Y}\right) =\phi _{X,Y}^{2}t_{F\left( X\right) ,F\left( Y\right)
}^{\prime }\left( \phi _{X,Y}^{2}\right) ^{-1}.$Then%
\begin{eqnarray*}
F\left( t_{X,Y}\otimes Z\right)  &=&\phi _{X\otimes Y,Z}^{2}\left( F\left(
t_{X,Y}\right) \otimes F\left( Z\right) \right) \left( \phi _{X\otimes
Y,Z}^{2}\right) ^{-1} \\
&&\phi _{X\otimes Y,Z}^{2}\left( \phi _{X,Y}^{2}\otimes F\left( Z\right)
\right) \left( t_{F\left( X\right) ,F\left( Y\right) }^{\prime }\otimes
F\left( Z\right) \right) \left( \left( \phi _{X,Y}^{2}\right) ^{-1}\otimes
F\left( Z\right) \right) \left( \phi _{X\otimes Y,Z}^{2}\right) ^{-1} \\
&=&\phi _{X,Y,Z}^{3}\left( t_{F\left( X\right) ,F\left( Y\right) }^{\prime
}\otimes F\left( Z\right) \right) \left( \phi _{X,Y,Z}^{3}\right) ^{-1}
\end{eqnarray*}%
and
\begin{eqnarray*}
F\left( X\otimes t_{Y,Z}\right)  &=&\phi _{X,Y\otimes Z}^{2}\left( F\left(
X\right) \otimes F\left( t_{Y,Z}\right) \right) \left( \phi _{X,Y\otimes
Z}^{2}\right) ^{-1} \\
&=&\phi _{X,Y\otimes Z}^{2}\left( F\left( X\right) \otimes \phi
_{Y,Z}^{2}\right) \left( F\left( X\right) \otimes t_{F\left( Y\right)
,F\left( Z\right) }^{\prime }\right) \left( F\left( X\right) \otimes \left(
\phi _{Y,Z}^{2}\right) ^{-1}\right) \left( \phi _{X,Y\otimes Z}^{2}\right)
^{-1} \\
&=&\phi _{X,Y,Z}^{3}\left( F\left( X\right) \otimes t_{F\left( Y\right)
,F\left( Z\right) }^{\prime }\right) \left( \phi _{X,Y,Z}^{3}\right) ^{-1}
\end{eqnarray*}%
so that
\begin{eqnarray*}
F\left( t_{X,Y}\otimes Z\right)  &=&\phi _{X,Y,Z}^{3}\left( t_{F\left(
X\right) ,F\left( Y\right) }^{\prime }\otimes F\left( Z\right) \right)
\left( \phi _{X,Y,Z}^{3}\right) ^{-1}, \\
F\left( X\otimes t_{Y,Z}\right)  &=&\phi _{X,Y,Z}^{3}\left( F\left( X\right)
\otimes t_{F\left( Y\right) ,F\left( Z\right) }^{\prime }\right) \left( \phi
_{X,Y,Z}^{3}\right) ^{-1}.
\end{eqnarray*}

If $t^{\prime }$ fulfils \eqref{qC-I}, then%
\begin{eqnarray*}
F\left( t_{X,Y\otimes Z}\right) &=&\phi _{X,Y\otimes Z}^{2}t_{F\left(
X\right) ,F\left( Y\otimes Z\right) }^{\prime }\left( \phi _{X,Y\otimes
Z}^{2}\right) ^{-1} \\
&=&\phi _{X,Y,Z}^{3}\left( F\left( X\right) \otimes \phi _{Y,Z}^{2}\right)
t_{F\left( X\right) ,F\left( Y\otimes Z\right) }^{\prime }\left( \phi
_{X,Y\otimes Z}^{2}\right) ^{-1} \\
&=&\phi _{X,Y,Z}^{3}t_{F\left( X\right) ,F\left( Y\right) \otimes F\left(
Z\right) }^{\prime }\left( F\left( X\right) \otimes \phi _{Y,Z}^{2}\right)
\left( \phi _{X,Y\otimes Z}^{2}\right) ^{-1} \\
&=&\phi _{X,Y,Z}^{3}t_{F\left( X\right) ,F\left( Y\right) \otimes F\left(
Z\right) }^{\prime }\left( \phi _{X,Y,Z}^{3}\right) ^{-1}
\end{eqnarray*}%
and
\begin{eqnarray*}
&&F(\sigma _{X,Y}^{-1}\otimes Z)F(Y\otimes t_{X,Z})F(\sigma _{X,Y}\otimes Z)
\\
&=&\left[ F(\sigma _{X,Y}\otimes Z)\right] ^{-1}F(Y\otimes t_{X,Z})F(\sigma
_{X,Y}\otimes Z) \\
&=&\left[
\begin{array}{c}
\left[ \phi _{Y,X,Z}^{3}\left( \sigma _{F\left( X\right) ,F\left( Y\right)
}^{\prime }\otimes F\left( Z\right) \right) \left( \phi _{X,Y,Z}^{3}\right)
^{-1}\right] ^{-1} \\
\phi _{Y,X,Z}^{3}\left( F\left( Y\right) \otimes t_{F\left( X\right)
,F\left( Z\right) }^{\prime }\right) \left( \phi _{Y,X,Z}^{3}\right) ^{-1}
\\
\phi _{Y,X,Z}^{3}\left( \sigma _{F\left( X\right) ,F\left( Y\right)
}^{\prime }\otimes F\left( Z\right) \right) \left( \phi _{X,Y,Z}^{3}\right)
^{-1}%
\end{array}%
\right] \\
&=&\phi _{X,Y,Z}^{3}\left( \sigma _{F\left( X\right) ,F\left( Y\right)
}^{\prime }\otimes F\left( Z\right) \right) \left( F\left( Y\right) \otimes
t_{F\left( X\right) ,F\left( Z\right) }^{\prime }\right) \left( \sigma
_{F\left( X\right) ,F\left( Y\right) }^{\prime }\otimes F\left( Z\right)
\right) \left( \phi _{X,Y,Z}^{3}\right) ^{-1}
\end{eqnarray*}%
so that
\begin{equation*}
F\left( t_{X,Y\otimes Z}\right) =F\left( t_{X,Y}\otimes Z\right) +F(\sigma
_{X,Y}^{-1}\otimes Z)F(Y\otimes t_{X,Z})F(\sigma _{X,Y}\otimes Z)
\end{equation*}%
and hence $t_{X,Y\otimes Z}=t_{X,Y}\otimes Z+(\sigma _{X,Y}^{-1}\otimes
Z)\circ (Y\otimes t_{X,Z})\circ (\sigma _{X,Y}\otimes Z).$

Similarly, If $t^{\prime }$ fulfils \eqref{qC-II}, then
\begin{eqnarray*}
F\left( t_{X\otimes Y,Z}\right) &=&\phi _{X\otimes Y,Z}^{2}t_{F\left(
X\otimes Y\right) ,F\left( Z\right) }^{\prime }\left( \phi _{X\otimes
Y,Z}^{2}\right) ^{-1} \\
&=&\phi _{X,Y,Z}^{3}\left( \phi _{X,Y}^{2}\otimes F\left( Z\right) \right)
t_{F\left( X\otimes Y\right) ,F\left( Z\right) }^{\prime }\left( \phi
_{X\otimes Y,Z}^{2}\right) ^{-1} \\
&=&\phi _{X,Y,Z}^{3}t_{F\left( X\right) \otimes F\left( Y\right) ,F\left(
Z\right) }^{\prime }\left( \phi _{X,Y}^{2}\otimes F\left( Z\right) \right)
\left( \phi _{X\otimes Y,Z}^{2}\right) ^{-1} \\
&=&\phi _{X,Y,Z}^{3}t_{F\left( X\right) \otimes F\left( Y\right) ,F\left(
Z\right) }^{\prime }\left( \phi _{X,Y,Z}^{3}\right) ^{-1}
\end{eqnarray*}%
and%
\begin{eqnarray*}
&&F(X\otimes \sigma _{Y,Z}^{-1})F(t_{X,Z}\otimes Y)F(X\otimes \sigma _{Y,Z})
\\
&=&F(X\otimes \sigma _{Y,Z})^{-1}F(t_{X,Z}\otimes Y)F(X\otimes \sigma _{Y,Z})
\\
&=&\left[
\begin{array}{c}
\left( \phi _{X,Z,Y}^{3}\left( F\left( X\right) \otimes \sigma _{F\left(
Y\right) ,F\left( Z\right) }^{\prime }\right) \left( \phi
_{X,Y,Z}^{3}\right) ^{-1}\right) ^{-1} \\
\phi _{X,Z,Y}^{3}\left( t_{F\left( X\right) ,F\left( Z\right) }^{\prime
}\otimes F\left( Y\right) \right) \left( \phi _{X,Z,Y}^{3}\right) ^{-1} \\
\phi _{X,Z,Y}^{3}\left( F\left( X\right) \otimes \sigma _{F\left( Y\right)
,F\left( Z\right) }^{\prime }\right) \left( \phi _{X,Y,Z}^{3}\right) ^{-1}%
\end{array}%
\right] \\
&=&\left( \phi _{X,Y,Z}^{3}\right) \left( F\left( X\right) \otimes \left(
\sigma _{F\left( Y\right) ,F\left( Z\right) }^{\prime }\right) ^{-1}\right)
\left( t_{F\left( X\right) ,F\left( Z\right) }^{\prime }\otimes F\left(
Y\right) \right) \left( F\left( X\right) \otimes \sigma _{F\left( Y\right)
,F\left( Z\right) }^{\prime }\right) \left( \phi _{X,Y,Z}^{3}\right) ^{-1}
\end{eqnarray*}%
so that $F\left( t_{X\otimes Y,Z}\right) =F\left( X\otimes t_{Y,Z}\right)
+F(X\otimes \sigma _{Y,Z}^{-1})F(t_{X,Z}\otimes Y)F(X\otimes \sigma _{Y,Z})$
and hence $t_{X\otimes Y,Z}=X\otimes t_{Y,Z}+(X\otimes \sigma
_{Y,Z}^{-1})\circ (t_{X,Z}\otimes Y)\circ (X\otimes \sigma _{Y,Z}).$

If $t^{\prime }$ fulfils \eqref{cart}, then%
\begin{eqnarray*}
F\left( \sigma _{X,Y}\right) F\left( t_{X,Y}\right) &=&\phi _{Y,X}^{2}\sigma
_{F\left( X\right) ,F\left( Y\right) }^{\prime }\left( \phi
_{X,Y}^{2}\right) ^{-1}\phi _{X,Y}^{2}t_{F\left( X\right) ,F\left( Y\right)
}^{\prime }\left( \phi _{X,Y}^{2}\right) ^{-1} \\
&=&\phi _{Y,X}^{2}\sigma _{F\left( X\right) ,F\left( Y\right) }^{\prime
}t_{F\left( X\right) ,F\left( Y\right) }^{\prime }\left( \phi
_{X,Y}^{2}\right) ^{-1} \\
&=&\phi _{Y,X}^{2}t_{F\left( Y\right) ,F\left( X\right) }^{\prime }\sigma
_{F\left( X\right) ,F\left( Y\right) }^{\prime }\left( \phi
_{X,Y}^{2}\right) ^{-1} \\
&=&\phi _{Y,X}^{2}t_{F\left( Y\right) ,F\left( X\right) }^{\prime }\left(
\phi _{Y,X}^{2}\right) ^{-1}\phi _{Y,X}^{2}\sigma _{F\left( X\right)
,F\left( Y\right) }^{\prime }\left( \phi _{X,Y}^{2}\right) ^{-1} \\
&=&F\left( t_{Y,X}\right) F\left( \sigma _{X,Y}\right)
\end{eqnarray*}%
and hence $F\left( \sigma _{X,Y}\right) F\left( t_{X,Y}\right) =F\left(
t_{Y,X}\right) F\left( \sigma _{X,Y}\right) $ so that $\sigma
_{X,Y}t_{X,Y}=t_{Y,X}\sigma _{X,Y}$.
\end{invisible}
\end{proof}

\begin{example}
Let $f:A\to B$ be a morphism of bialgebras which is an epimorphism in the category of rings. By \cite[Proposition XI.1.2]{Stenstrom75}, the restriction of scalars functor $f_*:{}_B\Mm\to{}_A\Mm$, which is always faithful, is also full. Since ${}_B\Mm$ and ${}_A\Mm$ are additive categories and $f_*$ is part of an adjunction, we get that  $f_*$ is an additive functor. Clearly $f_*$ is strict monoidal. Hence the previous result applies. As a consequence, if  $_A\Mm$ is a Cartier type category, so is $_B\Mm$.

In particular, since $f$ induces the surjective bialgebra map $f':A\to f(A),a\mapsto f(a)$, we get that, if $_A\Mm$ is a Cartier type category, so is $_{f (A)}\Mm$. A direct proof of this result will be provided in Proposition \ref{prop:image}. This can be obtained by looking at the notion of pre-Cartier quasitriangular bialgebra, which is object of the next section. 
\end{example}

\begin{example} Let $A$ and $B$ be bialgebras such that there is a strong monoidal equivalence between the categories $_A\Mm $ and $_B\Mm$.
By Proposition \ref{prop:ff}, $_A\Mm$ is a Cartier type category if and only if so is $_B\Mm$. This applies in particular if $A$ is obtained from $B$ by twisting with a Drinfel'd twist (we will return to this notion in more detail in the following section). A direct proof of this fact will be provided in Theorem \ref{thm:twist}.
\end{example}

\begin{example} Let $A$ and $B$ be two bialgebras which are \emph{monoidally co-Morita equivalent}, i.e., such that there is a strong monoidal $\Bbbk$-linear equivalence between the categories $^A\Mm $ and $^B\Mm$, see \cite[Definition 5.6]{Schauenburg96}. By Proposition \ref{prop:ff}, $^A\Mm$ is a Cartier type category if and only if so is $^B\Mm$. Note that, by \cite[Corollary 5.9]{Schauenburg96}, if $H$ is a finite-dimensional Hopf algebra, then every Hopf algebra monoidally co-Morita equivalent to $H$ is obtained from $H$ by twisting with a $2$-cocycle. 
\end{example}

\section{Pre-Cartier quasitriangular bialgebras}\label{Sec:2}

We are now going to discuss a class of examples of pre-Cartier categories modelled on the braided monoidal category $({}_H\Mm,\otimes,\sigma^\Rr)$ of left $H$-modules, where $\sigma^\Rr$ is the braiding attached to a quasitriangular bialgebra $(H,\Rr)$. Let us recall from \cite{Kassel} and \cite{Majid-book} some definitions and preliminary results beforehand. \\

%\subsubsection*{\textbf{Preliminary results}}

Fix a bialgebra $(H,m,u,\Delta,\varepsilon)$ in the following. For an element $T\in H\otimes H$ we adopt the short notation $T=\sum_{i}{T^{i}\otimes T_{i}}=T^{i}\otimes T_{i
}$, omitting the summation symbol. 
Then, for example, $\tau_{H,H}(T)=T^{\mathrm{op}}=T_{i}\otimes T^{i}$, $(\mathrm{Id}_{H}\otimes\tau_{H,H})(T\otimes1_{H})=T^{i}\otimes1_{H}\otimes T_{i}$, etc. If $T$ is invertible we write $T^{-1}=\overline{T}^{i}\otimes\overline{T}_{i}$. We further employ \textit{leg notation} $T_{12}=T\otimes1_{H}$, $T_{23}=1_{H}\otimes T$, $T_{13}=(\mathrm{Id}_{H}\otimes\tau_{H,H})(T_{12})$, $T_{21}=T^{\mathrm{op}}\otimes1_{H}$, etc. \\

 %\rd{[I would move this part on the twist (including Drinfeld functor below) to the related Subsection \ref{secTwist}. I know that the idea was to collect the notions we will use at the beginning, but the final effect is to interrupt the narration.]}

A bialgebra $H$ is said to be \textbf{quasitriangular} if there is an invertible element $\Rr\in H\otimes H$, the \textit{universal $\Rr$-matrix} or \textit{quasitriangular structure}, such that $H$ is quasi-cocommutative, i.e.,
\begin{equation}\label{qtr1}
    \Delta^\mathrm{op}(\cdot)=\Rr\Delta(\cdot)\Rr^{-1},
\end{equation}
and the hexagon equations
\begin{align}
    (\mathrm{Id}_H\otimes\Delta)(\Rr)&=\Rr_{13}\Rr_{12},\label{qtr2}\\
    (\Delta\otimes\mathrm{Id}_H)(\Rr)&=\Rr_{13}\Rr_{23},\label{qtr3}
\end{align}
are satisfied. If in addition $\Rr^{-1}=\Rr^\op$, then $(H,\Rr)$ is called \textbf{triangular}. Recall that, by \cite[Theorem VIII.2.4]{Kassel}, a quasitriangular bialgebra $(H,\Rr)$ satisfies the \textit{quantum Yang-Baxter} equations 
\begin{equation}\label{eq:QYB}
\Rr_{12}\Rr_{13}\Rr_{23}=\Rr_{23}\Rr_{13}\Rr_{12},\qquad\Rr^{-1}_{12}\Rr^{-1}_{13}\Rr^{-1}_{23}=\Rr^{-1}_{23}\Rr^{-1}_{13}\Rr^{-1}_{12}
\end{equation}
and $(\varepsilon\otimes\mathrm{Id}_{H})(\Rr^{\pm 1})=1_{H}=(\mathrm{Id}_{H}\otimes\varepsilon)(\Rr^{\pm 1})$. 
Moreover, we have 
\begin{equation}\label{eq:SotimesId(R)}
(S\otimes\mathrm{Id})(\Rr)=\Rr^{-1},\qquad (\mathrm{Id}\otimes S)(\Rr^{-1})=\Rr.
\end{equation}
Thus, we have that
\begin{equation}\label{eq:SotimesS(R)}
(S\otimes S)(\Rr)=\Rr,\qquad (S\otimes S)(\Rr^{-1})=\Rr^{-1}.
\end{equation}
One can prove (see \cite[Theorem 9.2.4 and paragraph thereafter]{Majid-book}) that a bialgebra $H$ is quasitriangular if and only if the monoidal category $({}_H\Mm,\otimes)$ of left $H$-modules is braided with braiding determined on objects $M,N$ by
\begin{equation*}
    \sigma^\Rr_{M,N}\colon M\otimes N\rightarrow N\otimes M,\qquad
    m\otimes n\mapsto\Rr^\op\cdot(n\otimes m)=(\Rr_i\cdot n)\otimes(\Rr^i\cdot m).
\end{equation*}
Note that $(\sigma^\Rr_{M,N})^{-1}:N\otimes M\to M\otimes N, n\otimes m\mapsto \Rr^{-1}\cdot (m\otimes n)=(\overline{\Rr}^i\cdot m)\otimes (\overline{\Rr}_{i}\cdot n)$ in this case.\par

\subsubsection*{\textbf{Definition and examples}}
%For our purpose, we introduce the following notion.
\begin{definition}\label{defn:curved(quasi)triang}
We call a triple $(H,\Rr, \chi)$ a \textbf{pre-Cartier} (quasi)triangular bialgebra if $(H,\Rr)$ is a (quasi)triangular bialgebra and $\chi\in H\otimes H$ is such that
\begin{align}
    \chi\Delta(\cdot)&=\Delta(\cdot)\chi,\label{cqtr1}\\
%    \Rr\chi&=\chi^\op\Rr\label{ctr2}\\
    (\mathrm{Id}_H\otimes\Delta)(\chi)&=\chi_{12}+\Rr^{-1}_{12}\chi_{13}\Rr_{12},\label{cqtr2}\\
    (\Delta\otimes\mathrm{Id}_H)
    (\chi)&=\chi_{23}+\Rr^{-1}_{23}\chi_{13}\Rr_{23},\label{cqtr3}
\end{align}
hold. We call such a $\chi$  an \textbf{infinitesimal $\Rr$-matrix}. \\Moreover, the triple $(H,\Rr,\chi)$ is called \textbf{Cartier} if it is pre-Cartier and $\chi$ satisfies
\begin{equation}\label{eq:ctr2}
    \mathcal{R}\chi=\chi^\mathrm{op}\mathcal{R},
\end{equation}
in addition. 
\end{definition}
We summarise these notions in the following table (again, \ding{55} indicates that the corresponding condition is not required).  
\begin{center}
   \begin{tabular}{|c|c|c|}
    \hline\cellcolor{gray!10} Terminology &\cellcolor{gray!10} $\Rr^{-1}=\Rr^\op$ &\cellcolor{gray!10}$\Rr\chi=\chi^\op\Rr$ \\\hline\hline
   pre-Cartier quasitriangular & \ding{55} & \ding{55} \\\hline
   pre-Cartier triangular & \ding{51}  & \ding{55}\\\hline
   Cartier quasitriangular & \ding{55} & \ding{51} \\\hline
   Cartier triangular & \ding{51} & \ding{51} \\\hline
   \end{tabular}
 \end{center}
\medskip

\begin{remark}\label{rmk:chivs} The following statements obviously hold.
\begin{enumerate}[i)]
    \item Every (quasi)triangular bialgebra is Cartier with trivial $\chi=0$.
    \item Given a (quasi)triangular bialgebra $(H,\Rr)$, the set of its infinitesimal $\Rr$-matrices forms a vector space.
    \item By applying $\mathrm{Id}\otimes\varepsilon\otimes\mathrm{Id}$ to \eqref{cqtr2} and \eqref{cqtr3}, and multiplying, we get
    \begin{align}
(\mathrm{Id}\otimes \varepsilon)(\chi)&=0,\label{eq:eps-chi1}\\
%$\chi^i\otimes  \chi_i=\chi^i\otimes\varepsilon(\chi_i) 1_H+\chi^i\otimes \chi_i$, thus 
(\varepsilon\otimes\mathrm{Id})(\chi)&=0,  \label{eq:eps-chi2}
\end{align}respectively.
\end{enumerate}
\end{remark}

\begin{example}\label{exa:chicntr} Let $(H,\Rr)$ be a (quasi)triangular bialgebra and let $\mathscr{Z}(H)$ be its center. If $\Rr\in\mathscr{Z}(H)\otimes\mathscr{Z}(H)$, then by \eqref{qtr1} $H$ is cocommutative. Moreover, \eqref{cqtr2} rewrites as $ (\mathrm{Id}_H\otimes\Delta)(\chi)=\chi_{12}+\chi_{13}$ which means that $\chi\in H\otimes P(H)$. Similarly, \eqref{cqtr3} means that $\chi\in P(H)\otimes H$, so that the two equalities together are equivalent to $\chi\in  P(H)\otimes P(H)$. Furthermore, \eqref{eq:ctr2} holds if and only if $\chi=\chi^\op$, i.e., $\chi$  is symmetric.
\begin{enumerate}[1)]
    \item As an example, since $H$ is cocommutative  a possible choice for the quasitriangular structure is $\Rr=1\otimes 1$ and in this case \eqref{cqtr2} and \eqref{cqtr3} are equivalent to $\chi\in  P(H)\otimes P(H)$, cf. \cite[Proposition XX.4.2]{Kassel}.
    \item Another example is the case when $H$ is commutative and then one trivially has $\Rr\in\mathscr{Z}(H)\otimes\mathscr{Z}(H)$ for every quasitriangular structure $\Rr$. Since \eqref{cqtr1} is trivially satisfied in this case, then, by the foregoing, $\chi\in H\otimes H$ is an infinitesimal $\Rr$-matrix if and only if $\chi\in P(H)\otimes P(H)$.    In particular, let $G$ be an abelian group and consider the group algebra $H=\Bbbk G$. Note that $\Rr$ needs not to be $1\otimes 1$ in this case, see e.g. \cite[Example 2.1.17]{Majid-book}. Since $H$ is commutative a possible infinitesimal $\Rr$-matrix $\chi$ must live in $P(H)\otimes P(H)$ but $P(\Bbbk G)=0$, as a direct computation shows. Thus $\chi=0$.
    \item Another example is the universal enveloping algebra $U\mathfrak{g}$ of a (complex) Lie algebra $(\mathfrak{g},[\cdot,\cdot])$. It is a cocommutative Hopf algebra and thus, if we consider $\Rr=1\otimes 1$, the above implies that infinitesimal $\Rr$-matrices $\chi$ are characterized as elements of $P(U\mathfrak{g})\otimes P(U\mathfrak{g})=\mathfrak{g}\otimes\mathfrak{g}$, such that $\chi\Delta(\cdot)=\Delta(\cdot)\chi$ holds. The latter condition is equivalent to $\chi$ being $\mathfrak{g}$-invariant, i.e., $(\mathrm{ad}_x\otimes\mathrm{Id}+\mathrm{Id}\otimes\mathrm{ad}_x)(\chi)=0$ for all $x\in\mathfrak{g}$, where $\mathrm{ad}_x(y):=[x,y]$ for $x,y\in\mathfrak{g}$.
    In summary, pre-Cartier structures for $(U\mathfrak{g},1\otimes 1)$ are precisely the elements of $(\mathfrak{g}\otimes\mathfrak{g})^\mathfrak{g}$ and the symmetric ones among those, i.e.,  $\chi\in(\mathfrak{g}\otimes\mathfrak{g})^\mathfrak{g}$ satisfying $\chi^{\mathrm{op}}=\chi$, correspond to the Cartier structures.
    Such elements arise in the study of quasitriangular Lie bialgebras.
    We briefly recall that a \textbf{Lie bialgebra} $(\mathfrak{g},[\cdot,\cdot],\delta)$ is a Lie algebra $(\mathfrak{g},[\cdot,\cdot])$ and a Lie coalgebra $(\mathfrak{g},\delta)$ such that $[\cdot,\cdot]\colon\mathfrak{g}\otimes\mathfrak{g}\to\mathfrak{g}$ and $\delta\colon\mathfrak{g}\to\mathfrak{g}\otimes\mathfrak{g}$ are compatible. A Lie bialgebra $(\mathfrak{g},[\cdot,\cdot],\delta)$ is called \textbf{quasitriangular} if there is an element $r\in\mathfrak{g}\otimes\mathfrak{g}$ such that $\delta(x)=(\mathrm{ad}_x\otimes\mathrm{Id}+\mathrm{Id}\otimes\mathrm{ad}_x)(r)$ for all $x\in\mathfrak{g}$, satisfying the \textbf{classical Yang-Baxter equation}
    $$
    [r_{12},r_{13}]+[r_{12},r_{23}]+[r_{13},r_{23}]=0
    $$
    and such that $r+r^{\mathrm{op}}\in(\mathfrak{g}\otimes\mathfrak{g})^{\mathfrak{g}}$. If $r$ is skew-symmetric in addition it is called a \textbf{classical $r$-matrix}.
    Thus, the symmetrization of any quasitriangular structure $r\in\mathfrak{g}\otimes\mathfrak{g}$ of a Lie bialgebra $(\mathfrak{g},[\cdot,\cdot],\delta)$ gives a Cartier structure $\chi=r+r^{\mathrm{op}}$ and classical $r$-matrices lead to $\chi=0$. We refer to \cite[Sections 2 and 3]{ES2010} for details on (quasitriangular) Lie bialgebras and $r$-matrices. Note that the classical Yang-Baxter equation was not used to verify the axioms of Definition~\ref{defn:curved(quasi)triang}. 
\end{enumerate} 
\end{example}

\begin{remark}
The element $\chi=r+r^{\mathrm{op}}$ associated to a quasitriangular Lie bialgebra $(\mathfrak{g},[\cdot,\cdot],r)$ we discussed in Example~\ref{exa:chicntr} already appeared in \cite{Majid00}. There it was associated with the ``infinitesimal braiding operator''
$$
\psi\colon V\otimes V\to V\otimes V,\qquad
\psi(v\otimes w):=\chi\cdot(v\otimes w-w\otimes v)
$$
for a left $\mathfrak{g}$-module $V$. It was shown in \cite[Lemma 2.1]{Majid00} that $\psi$ is a $2$-cocycle, a statement which we are going to generalize in Theorem~\ref{thm:cQYB}. The operator is further used to define ``braided Lie-bialgebras'' in \cite[Definition 2.2]{Majid00}, a notion which was extended in \cite[Section 5]{HV}. It would be interesting to adapt the concept of braided Lie-bialgebra to our notion of pre-Cartier quasitriangular bialgebra. However, this goes beyond the scope of the present paper.
\end{remark}

\begin{example}\label{exa:chisym}
Let $(H,\Rr)$ be a (quasi)triangular bialgebra and let $\chi \in \mathscr{Z}(H\otimes H)$. Then \eqref{cqtr1} is trivially satisfied, \eqref{cqtr2} and \eqref{cqtr3} together means that $\chi \in P(H)\otimes P(H)$ while \eqref{eq:ctr2} means that $\chi$ is symmetric. As a consequence, $(H,\Rr,\chi)$ is pre-Cartier if $\chi \in P(H)\otimes P(H)$ is central and it is Cartier when it is further symmetric.
\end{example}

In the following main theorem of this section we prove that pre-Cartier quasitriangular bialgebras $(H,\Rr,\chi)$ correspond to pre-Cartier braided monoidal structures on $({}_H\Mm,\otimes,\sigma^\Rr)$ in the same spirit in which $\Rr$ corresponds to $\sigma^\Rr$.

\begin{theorem}\label{thm}
Let $(H,\Rr)$ be a (quasi)triangular bialgebra. Then, there is a bijection between pre-Cartier structures of $(H,\Rr)$ and pre-Cartier structures of $(_{H}\mathcal{M},\otimes,\sigma^{\Rr})$. 
%It is pre-Cartier if and only if
%the corresponding representation category $({}_H\Mm,\otimes,\sigma^\Rr)$ is symmetric (resp. braided) pre-Cartier. 
The corresponding infinitesimal braiding on ${}_H\Mm$ is defined for all objects $M,N$ in ${}_H\Mm$ by
\begin{equation}\label{t-chi}
    t_{M,N}\colon M\otimes N\rightarrow M\otimes N,\qquad
    m\otimes n\mapsto\chi\cdot(m\otimes n)=(\chi^i\cdot m)\otimes(\chi_i\cdot n),
\end{equation}
where $\chi=\chi^i\otimes\chi_i\in H\otimes H$ is the infinitesimal $\Rr$-matrix for $H$. Moreover, there is a bijective correspondence between Cartier structures of $(H,\Rr)$ and Cartier structures of $(_{H}\mathcal{M},\otimes,\sigma^{\Rr})$.
%A (quasi)triangular bialgebra $(H,\Rr)$ is Cartier if and only if $({}_H\Mm,\otimes,\sigma^\Rr)$ is symmetric (resp. braided)  Cartier}.
\end{theorem}

\begin{proof}
Suppose there is an infinitesimal braiding $t$ on ${}_H\Mm$ and set $\chi:=t_{H,H}(1_{H}\otimes1_{H})$. Note that for every $M$ in ${}_H\Mm$ and $m\in M$ we can consider a map $l_m:H\to M$ in  ${}_H\Mm$ by setting $l_{m}(h)=h\cdot m$. %Similarly, given $N$ in ${}_H\Mm$ and $n\in N$, we can consider $l_{n}:H\to N$ in ${}_H\Mm$. 
Since $t$ is natural, for every $M,N$ in $_H\Mm$ and $m\in M,n\in N$, we have that $t_{M,N}(l_{m}\otimes l_{n})(1_{H}\otimes1_{H})=(l_{m}\otimes l_{n})t_{H,H}(1_{H}\otimes 1_{H})$, i.e., $t_{M,N}(m\otimes n)=(l_{m}\otimes l_{n})(\chi)$. By employing the notation $\chi=\chi^{i}\otimes\chi_{i}$, we arrive at $t_{M,N}(m\otimes n)=(\chi^{i}\cdot m)\otimes(\chi_{i}\cdot n)$. Now, the fact that $t_{M,N}$ is a morphism in ${}_H\Mm$ for all $M,N$ in ${}_H\Mm$ is equivalent to the fact that $t_{H,H}$ is a morphism in ${}_H\Mm$, by using the naturality of $t$. Indeed, if this is true for $H$, given $M$ and $N$ in ${}_H\Mm$ we obtain that
\[
\begin{split}
t_{M,N}&(h\cdot(m\otimes n))=t_{M,N}(h_{1}\cdot m\otimes h_{2}\cdot n)
=t_{M,N}(l_{m}(h_{1})\otimes l_{n}(h_{2}))
=t_{M,N}(l_{m}\otimes l_{n})(h_{1}\otimes h_{2})
\\&=(l_{m}\otimes l_{n})t_{H,H}(h\cdot(1_{H}\otimes 1_{H}))
=h\cdot(l_{m}\otimes l_{n})t_{H,H}(1_{H}\otimes 1_{H})
=h\cdot t_{M,N}(m\otimes n)
\end{split}
\]
for all $m\in M$ and $n\in N$. By using again naturality of $t$, the fact that $t_{H,H}$ is in ${}_H\Mm$ is equivalent to $h\cdot(t_{H,H}(1_{H}\otimes 1_{H}))=t_{H,H}(h\cdot(1_{H}\otimes 1_{H}))$ for every $h\in H$. This is exactly $h_{1}\chi^{i}\otimes h_{2}\chi_{i}=\chi^{i}h_{1}\otimes\chi_{i}h_{2}$ for every $h\in H$, i.e., \eqref{cqtr1}. 
%%%%%%%%%%%%%%%%%%%%%%%%%%%%%%%%%%%%%%%%%%%%%%%%%%%%%%%%%%%%%%%%%%%%%%%%%%%%%%%%%
\begin{comment}
Now we have that $\sigma^{\Rr}_{X,Y}\eta_{X,Y}=\eta_{Y,X}\sigma^{\Rr}_{X,Y}$ for every $X,Y$ in ${}_H\mm$ if and only if $\sigma^{\Rr}_{H,H}\eta_{H,H}=\eta_{H,H}\sigma^{\Rr}_{H,H}$, by using the naturality of $\sigma^{\Rr}$ and of $\eta$. In fact if we have $\sigma^{\Rr}_{H,H}\eta_{H,H}=\eta_{H,H}\sigma^{\Rr}_{H,H}$ then, given $X$ and $Y$ in ${}_H\mm$ we have that
\[
\begin{split}
    \sigma^{\Rr}_{X,Y}\eta_{X,Y}(x\otimes y)&=\sigma^{\Rr}_{X,Y}\eta_{X,Y}(l_{x}\otimes l_{y})(1_{H}\otimes1_{H})\\&=\sigma^{\Rr}_{X,Y}(l_{x}\otimes l_{y})\eta_{H,H}(1_{H}\otimes1_{H})\\&=(l_{y}\otimes l_{x})\sigma^{\Rr}_{H,H}\eta_{H,H}(1_{H}\otimes1_{H})\\&=(l_{y}\otimes l_{x})\eta_{H,H}\sigma^{\Rr}_{H,H}(1_{H}\otimes1_{H})\\&=\eta_{Y,X}(l_{y}\otimes l_{x})\sigma^{\Rr}_{H,H}(1_{H}\otimes1_{H})\\&=\eta_{Y,X}\sigma^{\Rr}_{X,Y}(l_{x}\otimes l_{y})(1_{H}\otimes1_{H})=\eta_{Y,X}\sigma^{\Rr}_{X,Y}(x\otimes y)
\end{split}
\]
By using again naturality of $\sigma^{\Rr}$ and $\eta$ we have that \eqref{eq:Cartier1} is equivalent to $\sigma^{\Rr}_{H,H}(\eta_{H,H}(1_{H}\otimes1_{H}))=\eta_{H,H}(\sigma^{\Rr}_{H,H}(1_{H}\otimes1_{H}))$ and then to $R_{i}\chi_{i}\otimes R^{i}\chi^{i}=\chi^{i}R_{i}\otimes\chi_{i}R^{i}$, which is $R^{\mathrm{op}}\chi^{\mathrm{op}}=\chi R^{\mathrm{op}}$. 
This last one is clearly equivalent, by applying $\tau$, to \eqref{ctr2}.
\end{comment}
%%%%%%%%%%%%%%%%%%%%%%%%%%%%%%%%%%%%%%%%%%%%%%%%%%%%%%%%%%%%%%%%%%%%%%%%%%%%%%%%%
Next, we prove that \eqref{qC-I} for $X, Y, Z$ in ${}_H\Mm$ is equivalent to \eqref{cqtr2}. In fact, from naturality of $t$, we reduce to the case $X=Y=Z=H$ and we have that
\[
t_{H,H\otimes H}(1_{H}\otimes1_{H}\otimes1_{H})=\chi^{i}\otimes(\chi_{i}\cdot(1_{H}\otimes1_{H}))=\chi^{i}\otimes\chi_{i1}\otimes\chi_{i2}=(\mathrm{Id}_{H}\otimes\Delta)(\chi),
\]
while
\[
(t_{H,H}\otimes\mathrm{Id}_{H})(1_{H}\otimes1_{H}\otimes1_{H})=\chi^{i}\otimes\chi_{i}\otimes1_{H}=\chi_{12}
\]
and
\begin{align*}
((\sigma^{\Rr}_{H,H})^{-1}\otimes\mathrm{Id}_{H})&(\mathrm{Id}_{H}\otimes t_{H,H})(\sigma^{\Rr}_{H,H}\otimes\mathrm{Id}_{H})(1_{H}\otimes1_{H}\otimes1_{H})\\
&=((\sigma^{\Rr}_{H,H})^{-1}\otimes\mathrm{Id}_{H})(\mathrm{Id}_{H}\otimes t_{H,H})(\Rr_{i}\otimes\Rr^{i}\otimes1_{H})\\
&=((\sigma^{\Rr}_{H,H})^{-1}\otimes\mathrm{Id}_{H})(\Rr_{i}\otimes\chi^{j}\Rr^{i}\otimes\chi_{j})
=\overline{\Rr}^{k}\chi^{j}\Rr^{i}\otimes \overline{\Rr}_{k}\Rr_{i}\otimes\chi_{j}\\
&=(\overline{\Rr}^{k}\otimes\overline{\Rr}_{k}\otimes1_{H})(\chi^{j}\otimes1_{H}\otimes\chi_{j})(\Rr^{i}\otimes\Rr_{i}\otimes1_{H})
=\Rr^{-1}_{12}\chi_{13}\Rr_{12}.
\end{align*}
Similarly, one proves that \eqref{qC-II} is equivalent to \eqref{cqtr3}.

Finally, we prove that \eqref{cart} is equivalent to \eqref{eq:ctr2}. Assume that \eqref{eq:ctr2} holds, then we have
\begin{align*}
    \sigma^\Rr_{M,N}(t_{M,N}(m\otimes n))
    &=\sigma^\Rr_{M,N}(\chi^i\cdot m\otimes\chi_i\cdot n)
    =(\Rr_j\chi_i)\cdot n\otimes(\Rr^j\chi^i)\cdot m\\
    &=(\Rr\chi)^\mathrm{op}\cdot(n\otimes m)
    \overset{\eqref{eq:ctr2}}=(\chi^\mathrm{op}\Rr)^\mathrm{op}\cdot(n\otimes m)
   =t_{N,M}(\sigma^\Rr_{M,N}(m\otimes n))
\end{align*}
for all left $H$-modules $M,N$ and $m\in M$, $n\in N$, hence \eqref{cart} holds. On the other hand, assume that \eqref{cart} is satisfied, then
$$
(\chi^\mathrm{op}\Rr)^\mathrm{op}=\chi\Rr^{\mathrm{op}}
=t_{H,H}(\sigma^\Rr_{H,H}(1_H\otimes 1_H))
=\sigma^\Rr_{H,H}(t_{H,H}(1_H\otimes 1_H))
=(\Rr\chi)^\mathrm{op},
$$
so \eqref{eq:ctr2} holds and the thesis follows.
\end{proof}

Now we test the notion of (pre-)Cartier bialgebra on some examples. The first example we consider is Sweedler's Hopf algebra. Note that, as it is shown in \cite{Ge92}, Taft Hopf algebras $H_{n^{2}}$ do not admit a quasitriangular structure for $n>2$, while for $n=2$, i.e., Sweedler's case, there is a $1$-parameter family of triangular structures which is considered below. \\

\textbf{Sweedler's Hopf Algebra.}
%\begin{example}[Sweedler's Hopf Algebra]\label{ex:Sweedler}
	Let $\Bbbk$ be a field of characteristic not $2$. Consider the free $\Bbbk$-algebra $H$ generated by two elements $g$ and $x$ modulo the relations $g^2=1$, $x^2=0$ and $xg=-gx$, where $1$ denotes the unity of $H$. It becomes a Hopf algebra with comultiplication, counit and antipode determined by $\Delta(g)=g\otimes g$, $\Delta(x)=x\otimes 1+g\otimes x$, $\varepsilon(g)=1$, $\varepsilon(x)=0$, $S(g)=g$ and $S(x)=-gx$, respectively. In particular, $H$ is neither commutative nor cocommutative. It is the smallest Hopf algebra with these properties and it is known as  \emph{Sweedler's 4-dimensional Hopf algebra}.
	There is an exhaustive $1$-parameter family of triangular structures on $H$ given by $\Rr$-matrices
	\begin{equation}\label{trngSweed}
	\Rr_\lambda:=1\otimes 1-2\nu\otimes \nu
	+\lambda(x\otimes x+2x\nu\otimes x\nu-2x\otimes x\nu)
	\end{equation}
	for $\lambda\in\Bbbk$, where $\nu:=\frac{1-g}{2}$, see \cite[Exercise 2.1.7]{Majid-book}. Expanding $\nu$ gives 
 \[\textstyle \Rr_\lambda=\frac{1}{2}(1\otimes 1+g\otimes 1+1\otimes g-g\otimes g)+\frac{\lambda}{2}(x\otimes x-xg\otimes x+x\otimes xg+xg\otimes xg).\] 
If $(H,\Rr)$ is quasitriangular then $\Rr=\Rr_{\lambda}$ for some $\lambda\in\Bbbk$, see \cite[Exercise 12.2.11]{Radford}.
In the next result, whose current more conceptual proof was suggested by the referee, we classify the infinitesimal $\Rr$-matrices for Sweedler's Hopf Algebra. 
%\rd{\sout{We thank the referee for suggesting the current  proof.} } 

\begin{proposition}\label{prop:Sweedler}
   Let $(H,\Rr_\lambda)$ be Sweedler's Hopf algebra with triangular structure $\Rr_\lambda$ as in \eqref{trngSweed}. Then, there is an exhaustive 1-parameter family of infinitesimal $\Rr$-matrices $\chi_\alpha$ on $H$, for $\alpha\in\Bbbk$, given by $\chi_\alpha=\alpha xg\otimes x$.
%As a consequence, all infinitesimal braidings on ${}_H\Mm$  are given for objects $M,N$ in ${}_H\Mm$ by
%\begin{equation*}
%t_{M,N}\colon M\otimes N\rightarrow M\otimes N,\quad
%m\otimes n\mapsto\chi_\alpha\cdot(m\otimes n)=(\alpha xg\cdot m)\otimes(x\cdot n).
%\end{equation*}   
\end{proposition}

 \begin{proof} If $H$ is pre-Cartier, by definition there exists an element $\chi\in H\otimes H$ satisfying the equations \eqref{cqtr1}, \eqref{cqtr2} and \eqref{cqtr3}. 
 %\textcolor{gray}{ 
  %A basis of $H$ as a $\Bbbk$-vector space is $\{1,g,x,xg\}$, hence a basis of $H\otimes H$ as $\Bbbk$-vector space is given by $\{1\otimes 1,1\otimes g,1\otimes x,1\otimes xg,g\otimes1, g\otimes g, g\otimes x,g\otimes xg,x\otimes1,x\otimes g,x\otimes x, x\otimes xg,xg\otimes1,xg\otimes g,xg\otimes x, xg\otimes xg\}$. We can write $\chi\in H\otimes H$ as 
  %\begin{align*}
  %\chi&=a(1\otimes 1)+b(1\otimes g)+c(1\otimes x)+d(1\otimes xg)+e(g\otimes 1)+f(g\otimes g)+h(g\otimes x)+i(g\otimes xg)\\&+l(x\otimes 1)+m(x\otimes g)+n(x\otimes x)+o(x\otimes xg)+p(xg\otimes 1)+q(xg\otimes g)+r(xg\otimes x)+s(xg\otimes xg),
  %\end{align*}
  %for some $a,b,c,d,e,f,h,i,l,m,n,o,p,q,r,s\in \Bbbk$.} 
  We look at \eqref{cqtr1} on generators. Clearly \eqref{cqtr1} on $1$ does not give information.
  %From \eqref{cqtr1} we have trivially that $\chi\Delta(1)=\Delta(1)\chi$ as $\Delta(1)=1\otimes 1$. 
  The condition $\chi\Delta(g)=\Delta(g)\chi$ is equivalent to $\chi=g\chi^ig^{-1}\otimes g\chi_i g^{-1}$. Let $H_{0}:=\Bbbk\langle g\rangle$ be the coradical of $H$. Then, $H=H_{0}\oplus xH_{0}$. The conjugation by $g$ is trivial on $H_{0}$ and it is the multiplication by $-1$ on $xH_{0}$. Then, since $\mathrm{char}(\Bbbk)\neq2$, we obtain $\chi\in H_{0}\otimes H_{0}+xH_{0}\otimes xH_{0}$.
  %i.e.,\[
	%\chi=a(1\otimes1)+b(1\otimes g)+e(g\otimes1)+f(g\otimes g)+n(x\otimes x)+o(x\otimes xg)+r(xg\otimes x)+s(xg\otimes xg).
	%\]
From $(\mathrm{Id}\otimes\varepsilon)(\chi)=0$ and $(\varepsilon\otimes\mathrm{Id})(\chi)=0$, we obtain $\chi\in\Bbbk(1-g)\otimes\Bbbk(1-g)+xH_{0}\otimes xH_{0}$. %i.e., $a=f$ and $b=e=-a$.
%\textcolor{gray}{From $\chi\Delta(g)=\Delta(g)\chi$, we have that $\chi\Delta(g)=a(g\otimes g)+b(g\otimes1)+c(g\otimes xg)+d(g\otimes x)+e(1\otimes g)+f(1\otimes1)+h(1\otimes xg)+i(1\otimes x)+l(xg\otimes g)+m(xg\otimes1)+n(xg\otimes xg)+o(xg\otimes x)+p(x\otimes g)+q(x\otimes1)+r(x\otimes xg)+s(x\otimes x)$
	%is equal to
	%$\Delta(g)\chi=a(g\otimes g)+b(g\otimes1)-c(g\otimes xg)-d(g\otimes x)+e(1\otimes g)+f(1\otimes1)-h(1\otimes xg)-i(1\otimes x)-l(xg\otimes g)-m(xg\otimes1)+n(xg\otimes xg)+o(xg\otimes x)-p(x\otimes g)-q(x\otimes1)+r(x\otimes xg)+s(x\otimes x)$.
	%Then, by comparing the respective terms of the basis of $H\otimes H$, since $\mathrm{char}(\Bbbk)\neq2$, we obtain that $c=d=h=i=l=m=p=q=0$, and
	%\[
	%\chi=a(1\otimes1)+b(1\otimes g)+e(g\otimes1)+f(g\otimes g)+n(x\otimes x)+o(x\otimes xg)+r(xg\otimes x)+s(xg\otimes xg).
	%\]}
We observe that $(x)=xH_{0}$ is a Hopf ideal of $H$ with $(x)^{2}=0$. Therefore, equation $0=\chi\Delta(x)-\Delta(x)\chi$ mod $(x)\otimes(x)$ implies that $\chi\in xH_{0}\otimes xH_{0}$, since $(1-g)x=x(1+g)$. 
%i.e.,
%\[
	%\chi=n(x\otimes x)+o(x\otimes xg)+r(xg\otimes x)+s(xg\otimes xg).
	%\]
%\textcolor{gray}{Now, we look at $\chi\Delta(x)=\Delta(x)\chi$. We have that the first member is $\chi\Delta(x)=
	%a(x\otimes1)+a(g\otimes x)+b(x\otimes g)-b(g\otimes xg)-e(xg\otimes1)+e(1\otimes x)-f(xg\otimes g)-f(1\otimes xg)$,	while the second is
%$\Delta(x)\chi=a(x\otimes1)+a(g\otimes x)+b(x\otimes g)+b(g\otimes xg)+e(xg\otimes1)+e(1\otimes x)+f(xg\otimes g)+f(1\otimes xg)$. By comparing again the respective terms of the basis we obtain that $b=e=f=0$. Hence we have that 
	%\[
	%\chi=a(1\otimes 1)+n(x\otimes x)+o(x\otimes xg)+r(xg\otimes x)+s(xg\otimes xg).
	%\]}
	The condition \eqref{cqtr1} for $xg$ does not give further information. %i.e., $\chi\Delta(xg)=\Delta(xg)\chi$, %about the parameters in $\chi$. 
We write now $\chi=xh\otimes xh'$, for $h,h'\in H_0$. Then, from \eqref{cqtr2} and $\Rr_{\lambda}^{-1}=\Rr_{\lambda}^\op$, we get \[xh\otimes (x\otimes 1)\Delta(h')+xh\otimes (g\otimes x)\Delta(h')=\chi\otimes 1+(\Rr_0)_{12}(xh\otimes 1\otimes xh')(\Rr_0)_{12}=\chi\otimes 1+xh\otimes 1\otimes xh',
\]
so that $xh\otimes (x\otimes 1)\Delta (h')=\chi\otimes 1$, and hence we can assume $h'=1$. %\textcolor{gray}{(It follows writing $h'=a1+bg$ in the latter equality. Now take $ha$ as a new $h$.)} 
Similarly, from \eqref{cqtr3} we obtain
\[
(x\otimes1)\Delta(h)\otimes x+(g\otimes x)\Delta(h)\otimes x=1\otimes xh\otimes x+(\Rr_{0})_{23}(xh\otimes1\otimes x)(\Rr_{0})_{23}=1\otimes xh\otimes x+xh\otimes1\otimes x,
\]
so that $(g\otimes x)\Delta(h)\otimes x=1\otimes xh\otimes x$, hence $h=\alpha g$, $\alpha\in\Bbbk$. 
It is immediate to verify that $\chi=\alpha xg\otimes x$, $\alpha\in \Bbbk$,  satisfies \eqref{cqtr2} and \eqref{cqtr3}. \qedhere
\begin{invisible}
Indeed, we have $(\mathrm{Id}\otimes\Delta)(\chi)=r(xg\otimes x\otimes 1+xg\otimes g\otimes x)$ and, by noting that $(\Rr_\lambda^{-1})_{12}\chi_{13}{\Rr_\lambda}_{12}=(\Rr_0^{-1})_{12}\chi_{13}{\Rr_0}_{12}$ as $x^2=0$, we get  $\chi_{12}+{\Rr_\lambda}^{-1}_{12}\chi_{13}{\Rr_\lambda}_{12}=rxg\otimes x\otimes 1+\frac{1}{4}r(xg\otimes 1\otimes x+x\otimes 1\otimes x+xg\otimes g\otimes x-x\otimes g\otimes x-x\otimes 1\otimes x-xg\otimes 1\otimes x-x\otimes g\otimes x+xg\otimes g\otimes x+xg\otimes g\otimes x+x\otimes g\otimes x+xg\otimes 1\otimes x-x\otimes 1\otimes x+x\otimes g\otimes x+xg\otimes g\otimes x+x\otimes 1\otimes x -xg\otimes 1\otimes x)=r(xg\otimes x\otimes 1+xg\otimes g\otimes x)$, hence \eqref{cqtr2} holds. Similarly, we have $(\Delta\otimes\mathrm{Id})(\chi)=r(xg\otimes g\otimes x+1\otimes xg\otimes x)$, and observing that $(\Rr_\lambda^{-1})_{23}\chi_{13}{\Rr_\lambda}_{23}=(\Rr_0^{-1})_{23}\chi_{13}{\Rr_0}_{23}$, we get $\chi_{23}+(\Rr_\lambda^{-1})_{23}\chi_{13}{\Rr_\lambda}_{23}=r(1\otimes xg\otimes x)+\frac{1}{4}r(xg\otimes 1\otimes x+xg\otimes g\otimes x+xg\otimes 1\otimes xg-xg\otimes g\otimes xg+xg\otimes g\otimes x+xg\otimes 1\otimes x+xg\otimes g\otimes xg-xg\otimes 1\otimes xg-xg\otimes 1\otimes xg-xg\otimes g\otimes xg-xg\otimes 1\otimes x+xg\otimes g\otimes x+xg\otimes g\otimes xg+xg\otimes 1\otimes xg+xg\otimes g\otimes x-xg\otimes 1\otimes x)=r(xg\otimes g\otimes x+1\otimes xg\otimes x)$, thus \eqref{cqtr3} holds. So, by denoting $\alpha =r\in \Bbbk$, for Sweedler's Hopf algebra $H$ with triangular structure $\Rr_{\lambda}$ ($\lambda\in \Bbbk$), the infinitesimal $\Rr$-matrix is of the form $\chi_\alpha=\alpha(xg\otimes x)$.
\end{invisible}
\end{proof}

\begin{remark}\label{rmk:Sw}
Note that $\Rr_\lambda\chi_\alpha=\frac{1}{2}\alpha(xg\otimes x-xg\otimes xg-x\otimes x-x\otimes xg)=-\chi_\alpha^{\mathrm{op}}\Rr_\lambda$, i.e., \begin{equation}\label{eq:ctr2-1}
 \Rr_\lambda\chi_\alpha=-\chi_\alpha^{\mathrm{op}}\Rr_\lambda .  
\end{equation}
If we further assume that $\chi_\alpha$ satisfies \eqref{eq:ctr2}, i.e., $\Rr_\lambda\chi_\alpha=\chi_\alpha^{\mathrm{op}}\Rr_\lambda$, then we get $\chi_\alpha^{\mathrm{op}}\Rr_\lambda=0$, hence $\chi_\alpha=0$ as $\Rr_\lambda$ is invertible. Thus, Sweedler's Hopf algebra is pre-Cartier through a $1$-parameter family of infinitesimal $\Rr$-matrices, while it is Cartier only with the trivial infinitesimal $\Rr$-matrix. %This shows that the pre-Cartier setting is more general than the Cartier one. 
\end{remark}

Note that \eqref{eq:ctr2-1} rewrites as $\Rr\chi=q\chi^\op\Rr$ for $q=-1$ which could be seen as  $q$-analogue of \eqref{eq:ctr2}. For this reason, we were tempted to introduce a notion of $q$-Cartier category by replacing \eqref{cart} by $\sigma_{X,Y}\circ t_{X,Y}=qt_{Y,X}\circ \sigma_{X,Y}$ for some $q\in\Bbbk$. However, we could not go too far beyond the example of Sweedler's $4$-dimensional Hopf algebra to justify this new notion. \medskip

\noindent \textbf{Topological bialgebra and quantization.}
Let $H$ be a bialgebra and consider the corresponding trivial topological bialgebra $\tilde{H}=H[[\hbar]]$ of formal power series with formal parameter $\hbar$, see for example \cite[Section XVI.4, Example 3]{Kassel}. The bialgebra structure of $\tilde{H}$ is obtained by $\hbar$-linearly extending the bialgebra structure of $H$, where one has to replace the tensor product with its topological completion.
If $\tilde{\Rr}=\Rr+\mathcal{O}(\hbar)\in(H\otimes H)[[\hbar]]\cong\tilde{H}\tilde{\otimes}\tilde{H}$, where $\tilde{\otimes}$ denotes the topological tensor product, is a quasitriangular structure on $\tilde{H}$, it follows that $\Rr$ is a quasitriangular structure on $H$. In particular, $\Rr\in H\otimes H$ is invertible and we can write
$$
\tilde{\Rr}=\Rr(1\otimes 1+\hbar\chi+\mathcal{O}(\hbar^2))
$$
for an element $\chi\in H\otimes H$. We have the following result.
\begin{proposition}\label{prop:h-bar}
Given $(\tilde{H},\tilde{\Rr})$ as above, then $(H,\Rr,\chi)$ is pre-Cartier.
\end{proposition}
\begin{proof}
The axioms of $\chi$ are precisely the defining axioms of $\tilde{\Rr}$ in first order of $\hbar$. Explicitly, \eqref{qtr1} in first order of $\hbar$ reads
$\Rr\chi\Delta(\cdot)=\Delta^\mathrm{op}(\cdot)\Rr\chi$.
Using that $\Delta$ is quasi-cocommutative with respect to $\Rr$ implies
$\Rr\chi\Delta(\cdot)=\Rr\Delta(\cdot)\chi$.
Then, multiplying the latter by $\Rr^{-1}$ from the left gives \eqref{cqtr1}.
Next, \eqref{qtr2} in order one of $\hbar$ gives $(\mathrm{Id}_H\otimes\Delta)(\Rr)(\mathrm{Id}_H\otimes\Delta)(\chi)=\Rr_{13}\chi_{13}\Rr_{12}+\Rr_{13}\Rr_{12}\chi_{12}$. Using that $\Rr$ satisfies \eqref{qtr2} and multiplying the former with $(\Rr_{13}\Rr_{12})^{-1}$ from the left implies \eqref{cqtr2}. Similarly, \eqref{qtr3} reads 
$(\Delta\otimes\mathrm{Id}_H)(\Rr)(\Delta\otimes\mathrm{Id}_H)(\chi)=\Rr_{13}\chi_{13}\Rr_{23}+\Rr_{13}\Rr_{23}\chi_{23}$ in order one of $\hbar$. Since $\Rr$ satisfies \eqref{qtr3} this gives \eqref{cqtr3} if we multiply the former with $(\Rr_{13}\Rr_{23})^{-1}$ from the left.
\end{proof}
%\end{example}

The above result suggests the following quantization problem.
\begin{question}\label{Q1} Given a pre-Cartier quasitriangular bialgebra $(H,\Rr,\chi)$ is there a quasitriangular structure $\tilde{\Rr}$ on the trivial topological bialgebra $\tilde{H}=H[[\hbar]]$ such that $\tilde{\Rr}=\Rr(1\otimes 1+\hbar\chi+\mathcal{O}(\hbar^2))$?
\end{question}
This scenario bares similarity to the deformation quantization of Poisson manifolds. In the following lines we want to give some intuition why this is the case and refer to \cite{Esposito} for more information on symplectic/Poisson manifolds and their quantization. A star product is an associative unital product $\star$ on the formal power series of smooth functions $\mathscr{C}^\infty(M)[[\hbar]]$ on a manifold $M$, which is also a smooth deformation of the pointwise product of functions. The skew-symmetrization of its first order structures $M$ as a Poisson manifold and it was shown by Kontsevich in \cite{Kont} that every Poisson manifold admits a star product quantization. An infinitesimal $\Rr$-matrix can be understood as a ``noncommutative bivector field'', as it satisfies the axioms \eqref{cqtr2} and \eqref{cqtr3} and every quasitriangular structure on formal power series leads to an infinitesimal $\Rr$-matrix , as observed in Proposition~\ref{prop:h-bar}. The question about existence of a quantization, as formulated in Question~\ref{Q1}, is less obvious.

In the symplectic case it is known (see for example \cite{Fed}) that equivalence classes of star product deformations are in bijection with formal power series of the second de Rham cohomology $\mathrm{H}_\mathrm{dR}^2(M,\mathbb{C})[[\hbar]]$. This last observation motivates the following question about pre-Cartier bialgebras.

\begin{question}\label{Q2} Are the ``quantizations'' $\tilde{\Rr}$ of $\chi$ controlled by a certain cohomology class?
\end{question}

In the following we give an affirmative answer to the Question \ref{Q1} in the case of Sweedler's Hopf algebra.

\begin{proposition}
Let $\chi$ be an arbitrary infinitesimal $\Rr$-matrix on Sweedler's
Hopf algebra $(H,\Rr_\lambda)$ over $\mathbb{C}$ and let $\lambda\in\mathbb{C}$. Then
    $\tilde{\Rr}:=\Rr_\lambda\exp(\hbar\chi)
    =\Rr_\lambda(1\otimes 1+\hbar\chi)
    \in(H\otimes H)[[\hbar]]$
is a quasitriangular structure on the trivial topological bialgebra $\tilde{H}
=H[[\hbar]]$. 
\end{proposition}
\begin{proof}
Recall from Proposition~\ref{prop:Sweedler} that 
$
	\Rr_\lambda:=1\otimes 1-2\nu\otimes\nu
	+\lambda(x\otimes x+2x\nu\otimes x\nu-2x\otimes x\nu)
$,
with $\lambda\in\mathbb{C}$ and $\nu:=\frac{1-g}{2}$,
is a family of triangular structures for $H$ and that 
all infinitesimal $\Rr$-matrices on $(H,\Rr_\lambda)$ are of the form 
$\chi=\alpha xg\otimes x\in H\otimes H$, with
$\alpha\in\mathbb{C}$. Since $x^2=0$ and $xg=-gx$ we have
$\exp(\hbar\chi)=\sum_{n=0}^\infty\frac{\hbar^n}{n!}\chi^n
=1\otimes 1+\hbar\chi$ and thus $\tilde{\Rr}:=\Rr_\lambda\exp(\hbar\chi)
=\Rr_\lambda(1\otimes 1+\hbar\chi)\in(H\otimes H)[[\hbar]]$.
We prove that $\tilde{\Rr}$ is a quasitriangular structure for 
$\tilde{H}$. First of all, $\tilde{\Rr}$ is invertible, since its zeroth order
in $\hbar$, namely $\Rr_\lambda$, is. Explicitly, its inverse is $\tilde{\Rr}^{-1}=(1\otimes 1-\hbar\chi)\Rr_\lambda^{-1}$.
Next, from the quasi-cocommutativity of $\Rr$ and \eqref{cqtr1} we obtain
$\Rr_\lambda(1\otimes 1+\hbar\chi)\Delta(\cdot)
=\Rr_\lambda\Delta(\cdot)(1\otimes 1+\hbar\chi)
=\Delta^\mathrm{op}(\cdot)\Rr_\lambda(1\otimes 1+\hbar\chi),$
i.e., quasi-cocommutativity of $\tilde{\Rr}$.
The hexagon equations for $\tilde{\Rr}$ follow from the ones of $\Rr$,
the relations \eqref{cqtr2},\eqref{cqtr3}
and the fact that $\chi_{13}\chi_{12}=0=\chi_{13}\chi_{23}$
as a consequence of $x^2=0$. Explicitly,
\begin{align*}
    (\mathrm{Id}\otimes\Delta)(\tilde{\Rr})
    &=(\mathrm{Id}\otimes\Delta)(\Rr)(1\otimes 1\otimes 1+\hbar(\mathrm{Id}\otimes\Delta)(\chi))\\
    &\overset{\eqref{cqtr2}}=\Rr_{13}\Rr_{12}(1\otimes 1\otimes 1+\hbar\chi_{12}
    +\hbar\Rr_{12}^{-1}\chi_{13}\Rr_{12})
\end{align*}
coincides with
\begin{align*}
    \tilde{\Rr}_{13}\tilde{\Rr}_{12}
    &=\Rr_{13}(1\otimes 1\otimes 1+\hbar\chi_{13})
    \Rr_{12}(1\otimes 1\otimes 1+\hbar\chi_{12})\\
    &=\Rr_{13}\Rr_{12}
    +\hbar\Rr_{13}\Rr_{12}\chi_{12}
    +\hbar\Rr_{13}\chi_{13}\Rr_{12}
    +\hbar^2\underbrace{\Rr_{13}\chi_{13}\Rr_{12}\chi_{12}}_{=0}
\end{align*}
and similarly $(\Delta\otimes\mathrm{Id})(\tilde{\Rr})$ coincides with $ \tilde{\Rr}_{13}\tilde{\Rr}_{23}$.
\begin{comment}and similarly
\begin{align*}
    (\Delta\otimes\mathrm{Id})(\tilde{\Rr})
    &=(\Delta\otimes\mathrm{Id})(\Rr)(1\otimes 1\otimes 1+\hbar(\Delta\otimes\mathrm{Id})(\chi))\\
    &=\Rr_{13}\Rr_{23}(1\otimes 1\otimes 1+\hbar\chi\otimes 1
    +\hbar\Rr_{23}^{-1}\chi_{13}\Rr_{23})
\end{align*}
coincides with
\begin{align*}
    \tilde{\Rr}_{13}\tilde{\Rr}_{23}
    &=\Rr_{13}(1\otimes 1\otimes 1+\hbar\chi_{13})
    \Rr_{23}(1\otimes 1\otimes 1+\hbar\chi_{23})\\
    &=\Rr_{13}\Rr_{23}
    +\hbar\Rr_{13}\Rr_{23}\chi_{23}
    +\hbar\Rr_{13}\chi_{13}\Rr_{23}
    +\hbar^2\underbrace{\Rr_{13}\chi_{13}\Rr_{23}\chi_{23}}_{=0}.
\end{align*}    
\end{comment}
This completes the proof.
\end{proof}
One could try to attack the previous quantization problem in greater generality. Here, we just comment on the relation of Question \ref{Q1} to the quantization of quasitriangular Lie bialgebras.
Recall from Example \ref{exa:chicntr} 3) that given a quasitriangular Lie bialgebra $(\mathfrak{g},[\cdot,\cdot],r)$ we obtain a Cartier triangular bialgebra $(U\mathfrak{g},\Rr,\chi)$ with $\Rr=1\otimes 1$ and $\chi=r+r^{\mathrm{op}}$.
According to \cite{EtingofKazhdan} there is a quasitriangular topological bialgebra $(\tilde{H},\tilde{\Delta},\tilde{\Rr})$ such that $\tilde{H}/\hbar\tilde{H}\cong U\mathfrak{g}$, $\tilde{\Delta}=\Delta+\mathcal{O}(\hbar)$ and
$$
\tilde{\Rr}=1\otimes 1+\hbar r+\mathcal{O}(\hbar^2)\in\tilde{H}\otimes\tilde{H}.
$$
 Note that the bialgebra structure of $\tilde{H}$ might differ from the $\hbar$-linear extension of $\Delta$, it is just required that in zeroth order of $\hbar$ the bialgebra $(H,\Delta)$ is recovered. This indicates that Question \ref{Q1} might be too restrictive and should be formulated in terms of arbitrary topological bialgebras rather than the trivial one.

\subsection{Induced structures}\label{Sec:induced}
In this subsection we study the behaviour of (pre-)Cartier (quasi)triangular bialgebras with respect to induced structures, i.e., images, quotients and tensor products.
\begin{proposition}\label{prop:image}
Let $f:H\to H'$ be a bialgebra map. If $(H,\Rr,\chi)$ is a (pre-)Cartier (quasi)triangular bialgebra, then so is $(f(H), (f\otimes f)(\Rr), (f\otimes f)(\chi))$.
\end{proposition}
\begin{proof}
Assume that $(H,\Rr)$ is a quasitriangular bialgebra, i.e., \eqref{qtr1}, \eqref{qtr2}, \eqref{qtr3} hold for the invertible element $\Rr=\Rr^i\otimes \Rr_i\in H\otimes H$. 
 It is easy to check that $(f(H), \Rr')$ is quasitriangular, where $\Rr':=(f\otimes f)(\Rr)$ cf. \cite[Exercise 12.2.2]{Radford}.
\begin{invisible}
We put in invisible the proof of Radford's exercise, (see also \cite[Definition 12.2.7 and Exercise 12.2.2]{Radford}).
We have that $f(H)$ is a sub-bialgebra of $H'$ and $(f\otimes f)(\Rr)\in H'\otimes H'$ is invertible with inverse $(f\otimes f)(\Rr^{-1})$, where $\Rr^{-1}=\overline{\Rr}^j\otimes\overline{\Rr}_j\in H\otimes H$ is the inverse of $\Rr$, since $f\otimes f$ is a morphism of algebras. 
%Indeed, using that $f$ and $f\otimes f$ are morphisms of algebras, $(f\otimes f)(\Rr)(f\otimes f)(\Rr^{-1})=(f\otimes f)(\Rr\Rr^{-1})=(f\otimes f)(1_H\otimes 1_H)=f(1_{H})\otimes f(1_{H})=1_{H'}\otimes 1_{H'}$ and $(f\otimes f)(\Rr^{-1})(f\otimes f)(\Rr)=(f\otimes f)(\Rr^{-1}\Rr)=(f\otimes f)(1_H\otimes 1_H)=f(1_{H})\otimes f(1_{H})=1_{H'}\otimes 1_{H'}$. 
Now we prove that equations \eqref{qtr1}, \eqref{qtr2}, \eqref{qtr3} hold for $(f\otimes f)(\Rr)$, simply using that $f$ is a morphism of algebras and of coalgebras and the respective equations for $\Rr$. In fact, since \eqref{qtr1} holds for $\Rr$, we have that for any $x\in H$ 
\[
\begin{split}
\Delta^{\mathrm{op}}_{H'}(f(x))&=(f\otimes f)(\Delta_H^{\mathrm{op}}(x))=(f\otimes f)(\Rr\Delta_H(x)\Rr^{-1})\\&=(f\otimes f)(\Rr)(f\otimes f)(\Delta_H(x))(f\otimes f)(\Rr^{-1})\\&=(f\otimes f)(\Rr)\Delta_{H'}(f(x))(f\otimes f)(\Rr)^{-1},
\end{split}
\]
hence $(f\otimes f)(\Rr)$ satisfies \eqref{qtr1}. Next, since \eqref{qtr2} holds for $\Rr$, i.e., $\Rr^i\otimes \Rr_{i_1}\otimes\Rr_{i_2}=\Rr^i\Rr^j\otimes\Rr_j\otimes \Rr_i$, we have that 
\[
\begin{split}
(\mathrm{Id}_{H'}\otimes\Delta_{H'})((f\otimes f)(\Rr))&=f(\Rr^i)\otimes f(\Rr_{i_1})\otimes f(\Rr_{i_2})=(f\otimes f\otimes f)(\Rr^i\otimes \Rr_{i_1}\otimes\Rr_{i_2})\\&=(f\otimes f\otimes f)(\Rr^i\Rr^j\otimes\Rr_j\otimes \Rr_i)=f(\Rr^i)f(\Rr^j)\otimes f(\Rr_j)\otimes f(\Rr_i)\\&=(f(\Rr^i)\otimes 1\otimes f(\Rr_i))(f(\Rr^j)\otimes f(\Rr_j)\otimes 1)=(f\otimes f)(\Rr)_{13}(f\otimes f)(\Rr)_{12},
\end{split}
\]
hence $(f\otimes f)(\Rr)$ satisfies \eqref{qtr2}. Similarly, since \eqref{qtr3} holds for $\Rr$, then \eqref{qtr3} holds for $(f\otimes f)(\Rr)$, too.
 \end{invisible}
%since \eqref{qtr3} holds for $\Rr$, i.e., $\Rr^i_1\otimes\Rr^i_2\otimes\Rr_i=\Rr^i\otimes\Rr^j\otimes\Rr_i\Rr_j$, we obtain that 
%\[
%\begin{split}
%(\Delta_{H'}\otimes\mathrm{Id}_{H'})((f\otimes f)(\Rr))&=f(\Rr^i_1)\otimes f(\Rr^i_2)\otimes f(\Rr_i)=(f\otimes f\otimes f)(\Rr^i_1\otimes\Rr^i_2\otimes \Rr_i)\\&=(f\otimes f\otimes f)(\Rr^i\otimes\Rr^j\otimes\Rr_i\Rr_j)=f(\Rr^i)\otimes f(\Rr^j)\otimes f(\Rr_i)f(\Rr_j)\\&=(f(\Rr^i)\otimes 1\otimes f(\Rr_i))(1\otimes f(\Rr^j)\otimes f(\Rr_j))=(f\otimes f)(\Rr)_{13}(f\otimes f)(\Rr)_{23},
%\end{split} 
%\]
%thus \eqref{qtr3} holds for $(f\otimes f)(\Rr)$. 
Moreover, if $H$ is triangular, i.e., $\Rr^{-1}=\Rr^\mathrm{op}$, we get
\[
(\Rr')^{-1}=(f\otimes f)(\Rr)^{-1}=(f\otimes f)(\Rr^{-1})=(f\otimes f)(\Rr^{\mathrm{op}})=f(\Rr_i)\otimes f(\Rr^i)=(f\otimes f)(\Rr)^\op=(\Rr')^\op, 
\]
hence also $(f(H), \Rr')$ is triangular. 
Now, assume that $(H,\Rr)$ is a pre-Cartier quasitriangular bialgebra, i.e., there exists $\chi=\chi^i\otimes\chi_i\in H\otimes H$ which satisfies \eqref{cqtr1}, \eqref{cqtr2}, \eqref{cqtr3}. We show that $(f(H), \Rr')$ is pre-Cartier through $\chi':=(f\otimes f)(\chi)\in H'\otimes H'$. Indeed, since \eqref{cqtr1} holds for $\chi$, we have that 
\[
\begin{split}
\chi'&\Delta_{H'}(f(x))=(f\otimes f)(\chi)(f\otimes f)(\Delta_H(x))=(f\otimes f)(\chi\Delta_H(x))\\&=(f\otimes f)(\Delta_H(x)\chi)=(f\otimes f)(\Delta_H(x))(f\otimes f)(\chi)=\Delta_{H'}(f(x))\chi'
\end{split}
\]
for any $x\in H$, thus \eqref{cqtr1} holds for $\chi'$. Next, since \eqref{cqtr2} holds for $\chi$, i.e., $\chi^i\otimes\chi_{i_1}\otimes\chi_{i_2}=\chi^i\otimes\chi_i\otimes 1_H+\overline{\Rr}^j\chi^i\Rr^k\otimes\overline{\Rr}_j\Rr_k\otimes\chi_i$, we have that 
\[
\begin{split}
(\mathrm{Id}_{H'}\otimes\Delta_{H'})&((f\otimes f)(\chi))=f(\chi^i)\otimes f(\chi_{i_1})\otimes f(\chi_{i_2})=(f\otimes f\otimes f)(\chi^i\otimes\chi_{i_1}\otimes\chi_{i_2})\\
&=(f\otimes f\otimes f)(\chi^i\otimes\chi_i\otimes 1_H+\overline{\Rr}^j\chi^i\Rr^k\otimes\overline{\Rr}_j\Rr_k\otimes\chi_i)\\&=f(\chi^i)\otimes f(\chi_i)\otimes1_{H'}+f(\overline{\Rr}^j)f(\chi^i)f(\Rr^k)\otimes f(\overline{\Rr}_j)f(\Rr_k)\otimes f(\chi_i)\\
&=\chi'_{12}+(f(\overline{\Rr}^j)\otimes f(\overline{\Rr}_j)\otimes 1_{H'})(f(\chi^i)\otimes 1_{H'}\otimes f(\chi_i))(f(\Rr^k)\otimes f(\Rr_k)\otimes 1_{H'})\\
&=\chi'_{12}+(\Rr')^{-1}_{12}\chi'_{13}\Rr'_{12},
\end{split}
\]
hence $\chi'$ satisfies \eqref{cqtr2}.
Similarly, since \eqref{cqtr3} holds for $\chi$, then $\chi'$ satisfies \eqref{cqtr3}.
%Finally, since \eqref{cqtr3} holds for $\chi$, i.e, $\chi^i_1\otimes\chi^i_2\otimes\chi_i=1_H\otimes \chi^i\otimes\chi_i+ \chi^i\otimes\overline{\Rr}^j\Rr^k\otimes \overline{\Rr}_j\chi_i\Rr_k$, we have that 
%\[
%\begin{split}
%(\Delta_{H'}\otimes \mathrm{Id}_{H'})((f\otimes f)(\chi))&=f(\chi^i_1)\otimes f(\chi^i_2)\otimes f(\chi_i)=(f\otimes f\otimes f)(\chi^i_1\otimes\chi^i_2\otimes\chi_i)\\&=(f\otimes f\otimes f)(1_H\otimes \chi^i\otimes\chi_i+ \chi^i\otimes\overline{\Rr}^j\Rr^k\otimes \overline{\Rr}_j\chi_i\Rr_k)\\&=1_{H'}\otimes f(\chi^i)\otimes f(\chi_i)+ f(\chi^i)\otimes f(\overline{\Rr}^j)f(\Rr^k)\otimes f(\overline{\Rr}_j)f(\chi_i)f(\Rr_k)\\&=(f\otimes f)(\chi)_{23}+(1_{H'}\otimes f(\overline{\Rr}^j)\otimes f(\overline{\Rr}_j))(f(\chi^i)\otimes1_{H'}\otimes f(\chi_i))(1_{H'}\otimes f(\Rr^k)\otimes f(\Rr_k))\\&=(f\otimes f)(\chi)_{23}+((f\otimes f)(\Rr)^{-1})_{23}(f\otimes f)(\chi)_{13}(f\otimes f)(\Rr)_{23}, 
%\end{split}
%\]
%hence $(f\otimes f)(\chi)$ satisfies \eqref{cqtr3}. 
Thus, $(f(H), \Rr')$ is pre-Cartier through $\chi'$. Furthermore, if the quasitriangular bialgebra $(H, \Rr)$ is assumed to be Cartier, i.e., $\chi$ satisfies \eqref{eq:ctr2} in addition, then $(f(H), \Rr')$ is Cartier. In fact, 
\[
\begin{split}
\Rr'\chi'&=(f\otimes f)(\Rr)(f\otimes f)(\chi)=(f\otimes f)(\Rr\chi)=(f\otimes f)(\chi^{\mathrm{op}}\Rr)=(f\otimes f)(\chi^{\mathrm{op}})(f\otimes f)(\Rr)\\&=(f\otimes f)(\chi)^{\mathrm{op}}(f\otimes f)(\Rr)=(\chi')^\op\Rr'. 
\end{split}
\]
The Cartier triangular case follows from the previous ones.
\end{proof}
%\rd{Rewrite in terms of bi-ideals and remove proof.}
As a consequence, we obtain the following result.
\begin{corollary}\label{cor:quotient}
Let $H$ be a bialgebra and $I\subseteq H$ a bi-ideal.
If $(H,\Rr,\chi)$ is a (pre-)Cartier (quasi)triangular bialgebra, then so is the quotient bialgebra $H/I$.
\end{corollary}	

%\begin{proof}
%Since $I$ is a two-sided ideal and a two-sided coideal then $H/I$ is a quotient bialgebra and the canonical projection $\pi:H\to H/I$ is a surjective morphism of bialgebras. Thus we can apply Proposition \ref{prop:image} and conclude. 
%\end{proof}

Next, we show that the tensor product of (quasi)triangular bialgebras is (pre-)Cartier if at least one of them is (pre-)Cartier.

\begin{proposition}\label{prop:tensor}
Let $(H, \Rr)$, $(H',\Rr')$ be (quasi)triangular bialgebras. Then, so is $H\otimes H'$ through $\tilde{\Rr}=(\mathrm{Id}_{H}\otimes\tau_{H,H'}\otimes\mathrm{Id}_{H'})(\Rr\otimes \Rr')$, where $\tau$ is the canonical flip. Furthermore, if at least one of $(H,\Rr)$ and $(H',\Rr')$ is (pre-)Cartier, then  so is $(H\otimes H',\tilde{\Rr})$.
\end{proposition}

\begin{proof}
The first assertion can be obtained as a particular case of \cite[Theorem 2.2]{Ch98} once observed that $H\bowtie^{1\otimes 1} H'=H\otimes H'$ and $[\Rr,\Rr']=\tilde{\Rr}$, using the notations therein.
\begin{invisible}
Assume that $(H, \Rr)$, $(H',\Rr')$ are quasitriangular bialgebras and observe that $\tilde{\Rr}$ is an invertible element of $H\otimes H'\otimes H\otimes H'$, with inverse $\tilde{\Rr}^{-1}=(\mathrm{Id}_{H}\otimes\tau_{H,H'}\otimes\mathrm{Id}_{H'})(\Rr^{-1}\otimes\Rr'^{-1})$. Since $(H,\Rr)$ and $(H',\Rr')$ satisfy \eqref{qtr1}, i.e., for every $c\in H$ and $h\in H'$
\[
c_{2}\otimes c_{1}=\Rr^{i}c_{1}\overline{\Rr}^{j}\otimes\Rr_{i}c_{2}\overline{\Rr}_{j}\ \text{ and }\ h_{2}\otimes h_{1}=\Rr'^{k}h_{1}\overline{\Rr'}^{l}\otimes\Rr'_{k}h_{2}\overline{\Rr'}_{l},
\]
then \eqref{qtr1} holds for $(H\otimes H',\tilde{\Rr})$. Indeed we have that
\[
\begin{split}
\tilde{\Rr}&\Delta_{H\otimes H'}(c\otimes h)\tilde{\Rr}^{-1}=(\Rr^{i}\otimes\Rr'^{k}\otimes \Rr_{i}\otimes\Rr'_k)(c_{1}\otimes h_{1}\otimes c_{2}\otimes h_{2})(\overline{\Rr}^{j}\otimes\overline{\Rr'}^{l}\otimes\overline{\Rr}_{j}\otimes\overline{\Rr'}_{l})\\&=\Rr^{i}c_{1}\overline{\Rr}^{j}\otimes\Rr'^{k}h_{1}\overline{\Rr'}^{l}\otimes\Rr_{i}c_{2}\overline{\Rr}_{j}\otimes\Rr'_{k}h_{2}\overline{\Rr'}_{l}=c_{2}\otimes h_{2}\otimes c_{1}\otimes h_{1}=\Delta_{H\otimes H'}^{\mathrm{op}}(c\otimes h).
\end{split}
\]
Since $(H, \Rr)$, $(H',\Rr')$ satisfy \eqref{qtr2}, we have $\Rr^i\otimes\Rr_{i_1}\otimes\Rr_{i_2}=\Rr^i\Rr^j\otimes\Rr_j\otimes\Rr_i$ and $\Rr'^{k}\otimes\Rr'_{k_1}\otimes\Rr'_{k_2}=\Rr'^k\Rr'^l\otimes\Rr'_l\otimes\Rr'_k$. Then, we get \begin{displaymath}
\begin{split}
(\mathrm{Id}\otimes\Delta_{H\otimes H'})(\tilde{\Rr})&=(\mathrm{Id}\otimes\Delta_{H\otimes H'})(\Rr^i\otimes\Rr'^k\otimes\Rr_i\otimes\Rr'_k)\\
&=\Rr^i\otimes\Rr'^k\otimes\Rr_{i_1}\otimes\Rr'_{k_1}\otimes\Rr_{i_2}\otimes\Rr'_{k_2}\\
&=\Rr^i\Rr^j\otimes\Rr'^k\Rr'^l\otimes\Rr_j\otimes\Rr'_l\otimes\Rr_i\otimes \Rr'_k\\
&=(\Rr^i\otimes\Rr'^k\otimes 1_H\otimes 1_{H'}\otimes\Rr_i\otimes\Rr'_k)(\Rr^j\otimes\Rr'^l\otimes\Rr_j\otimes\Rr'_l \otimes 1_H\otimes 1_{H'})\\
&=(\Rr^i\otimes\Rr'^k\otimes\Rr_i\otimes\Rr'_k)_{13}(\Rr^j\otimes\Rr'^l\otimes\Rr_j\otimes\Rr'_l)_{12}
=\tilde{\Rr}_{13}\tilde{\Rr}_{12},
\end{split}
\end{displaymath}
thus \eqref{qtr2} holds for $(H\otimes H',\tilde{\Rr})$. Similarly, since $(H, \Rr)$, $(H',\Rr')$ satisfy \eqref{qtr3}, we have that \eqref{qtr3} holds for $(H\otimes H',\tilde{\Rr})$.
%Since $(H, \Rr)$, $(H',\Rr')$ satisfy \eqref{qtr3}, we have $\Rr^i_1\otimes\Rr^i_{2}\otimes\Rr_{i}=\Rr^i\otimes \Rr^j\otimes\Rr_i\Rr_j$ and $\Rr'^{k}_1\otimes\Rr'^k_{2}\otimes\Rr'_{k}=\Rr'^k\otimes\Rr'^l\otimes\Rr'_k\Rr'_l$. Then, we get \begin{displaymath}
%\begin{split}
%(\Delta\otimes\mathrm{Id})(\tilde{\Rr})&=(\Delta\otimes\mathrm{Id})(\Rr^i\otimes\Rr'^k\otimes\Rr_i\otimes\Rr'_k)\\
%&=\Rr^i_1\otimes\Rr'^k_1\otimes\Rr^i_{2}\otimes\Rr'^k_{2}\otimes\Rr_{i}\otimes\Rr'_{k}\\
%&=\Rr^i\otimes\Rr'^k\otimes\Rr^j\otimes\Rr'^l\otimes\Rr_{i}\Rr_{j}\otimes\Rr'_{k}\Rr'_{l}\\
%&=(\Rr^i\otimes\Rr'^k\otimes 1_H\otimes 1_{H'}\otimes\Rr_i\otimes\Rr'_k)(1_H\otimes 1_{H'}\otimes\Rr^j\otimes\Rr'^l\otimes\Rr_j\otimes\Rr'_l)\\
%&=(\Rr^i\otimes\Rr'^k\otimes\Rr_i\otimes\Rr'_k)_{13}(\Rr^j\otimes\Rr'^l\otimes\Rr_j\otimes\Rr'_l)_{23}\\
%&=\tilde{\Rr}_{13}\tilde{\Rr}_{23}
%\end{split}
%\end{displaymath}
%Now, since the flip map $\tau$ is a morphism of algebras, then we have
%$$\tilde{\Rr}^{-1}=((\mathrm{Id}\otimes\tau_{H,H'}\otimes\mathrm{Id})(\Rr\otimes \Rr'))^{-1}=(\mathrm{Id}\otimes\tau_{H,H'}\otimes\mathrm{Id})((\Rr\otimes \Rr')^{-1})=(\mathrm{Id}\otimes\tau_{H,H'}\otimes\mathrm{Id})(\Rr^{-1}\otimes \Rr'^{-1}).$$
Hence $(H\otimes H',\tilde{\Rr})$ is a quasitriangular bialgebra.  
\end{invisible}

We show that if $\Rr^{-1}=\Rr^{\mathrm{op}}$ and $\Rr'^{-1}=\Rr'^{\mathrm{op}}$ then $\tilde{\Rr}^{-1}=\tilde{\Rr}^{\mathrm{op}}$. Indeed,
\[
\begin{split}
\tilde{\Rr}^{-1}&=(\mathrm{Id}_{H}\otimes\tau_{H,H'}\otimes\mathrm{Id}_{H'})(\Rr^{-1}\otimes\Rr'^{-1})=(\mathrm{Id}_{H}\otimes\tau_{H,H'}\otimes\mathrm{Id}_{H'})(\Rr^{\mathrm{op}}\otimes\Rr'^{\mathrm{op}})\\&=\Rr_i\otimes\Rr'_{j}\otimes\Rr^i\otimes\Rr'^j=\tau_{H\otimes H',H\otimes H'}(\tilde{\Rr})=\tilde{\Rr}^{\mathrm{op}}.
\end{split}
\]
Now, we suppose that $(H,\Rr)$ and $(H',\Rr')$ are  quasitriangular bialgebras and that $(H,\Rr)$ is pre-Cartier and we show that $H\otimes H'$ is pre-Cartier. Since $(H,\Rr)$ is pre-Cartier, there exists $\chi=\chi^i\otimes\chi_i\in H\otimes H$ 
%and $\chi'=\chi'^{l}\otimes\chi'_{l}\in H'\otimes H'$ 
such that \eqref{cqtr1}, \eqref{cqtr2} and \eqref{cqtr3} are satisfied. We define $\tilde{\chi}:=\chi^i\otimes1_{H'}\otimes\chi_{i}\otimes1_{H'}\in H\otimes H'\otimes H\otimes H'$.
%$\hat{\chi}:=1_{H}\otimes\chi'^{l}\otimes1_{H}\otimes\chi'_{l}\in H\otimes H'\otimes H\otimes H'$. 
Since $\chi$ satisfies \eqref{cqtr1}, we have $\chi^ic_{1}\otimes\chi_{i}c_{2}=c_{1}\chi^{i}\otimes c_{2}\chi_{i}$ 
%\[
%\text{and}\ \chi'^{l}d_{1}\otimes\chi'_{l}d_{2}=d_{1}\chi'^{l}\otimes d_{2}\chi'_{l}
%\]
for every $c\in H$, thus 
\[
\begin{split}
    \tilde{\chi}&\Delta_{H\otimes H'}(c\otimes d)=(\chi^i\otimes1_{H'}\otimes\chi_{i}\otimes1_{H'})(c_{1}\otimes d_{1}\otimes c_{2}\otimes d_{2})=\chi^ic_{1}\otimes d_{1}\otimes\chi_{i}c_{2}\otimes d_{2}\\&=c_{1}\chi^i\otimes d_{1}\otimes c_{2}\chi_{i}\otimes d_{2}=(c_{1}\otimes d_{1}\otimes c_{2}\otimes d_{2})(\chi^{i}\otimes1_{H'}\otimes\chi_{i}\otimes1_{H'})=\Delta_{H\otimes H'}(c\otimes d)\tilde{\chi},
\end{split}
\]
%\[
%\begin{split}
    %\hat{\chi}\Delta(c\otimes d)&=(1_{H}\otimes\chi'^{l}\otimes1_{H}\otimes\chi'_{l})(c_{1}\otimes d_{1}\otimes c_{2}\otimes d_{2})\\&=c_{1}\otimes\chi'^{l}d_{1}\otimes c_{2}\otimes\chi'_{l}d_{2}\\&=c_{1}\otimes d_{1}\chi'^{l}\otimes c_{2}\otimes d_{2}\chi'_{l}\\&=(c_{1}\otimes d_{1}\otimes c_{2}\otimes d_{2})(1_{H}\otimes\chi'^{l}\otimes1_{H}\otimes\chi'_{l})\\&=\Delta(c\otimes d)\hat{\chi},
%\end{split}
%\]
so that $\tilde{\chi}$ satisfies \eqref{cqtr1}. Now, since $\chi$ satisfies \eqref{cqtr2}, i.e., $\chi^i\otimes\chi_{i_1}\otimes\chi_{i_2}=\chi^{i}\otimes\chi_{i}\otimes1_{H}+\overline{\Rr}^{j}\chi^{i}\Rr^{k}\otimes\overline{\Rr}_{j}\Rr_{k}\otimes\chi_{i}$, 
%\[
%\text{ and }\ \chi'^{l}\otimes\chi'_{l_1}\otimes\chi'_{l_2}=\chi'^{l}\otimes\chi'_{l}\otimes1_{H'}+\overline{\Rr'}^{m}\chi'^{l}\Rr'^{n}\otimes\overline{\Rr'}_{m}\Rr'_{n}\otimes\chi'_{l}, 
%\]
we obtain that
\[
\begin{split}
    &(\mathrm{Id}\otimes\Delta_{H\otimes H'})(\tilde{\chi})=\chi^i\otimes1_{H'}\otimes\chi_{i_1}\otimes1_{H'}\otimes\chi_{i_2}\otimes1_{H'}\\&=\chi^{i}\otimes1_{H'}\otimes\chi_{i}\otimes1_{H'}\otimes1_{H}\otimes1_{H'}+\overline{\Rr}^{j}\chi^i\Rr^{k}\otimes1_{H'}\otimes\overline{\Rr}_{j}\Rr_{k}\otimes1_{H'}\otimes\chi_{i}\otimes1_{H'}\\
    &=\tilde{\chi}_{12}+(\overline{\Rr}^{j}\otimes\overline{\Rr}'^{m}\otimes\overline{\Rr}_{j}\otimes\overline{\Rr}'_{m}\otimes1_{H}\otimes1_{H'})(\chi^{i}\otimes1_{H'}\otimes1_{H}\otimes1_{H'}\otimes\chi_{i}\otimes1_{H'})(\Rr^{k}\otimes\Rr'^{n}\otimes\Rr_{k}\otimes\Rr'_{n}\otimes1_{H}\otimes1_{H'})\\
    &=\tilde{\chi}_{12}+\tilde{\Rr}^{-1}_{12}\tilde{\chi}_{13}\tilde{\Rr}_{12},
\end{split}
\]
%\[
%\begin{split}
    %(\mathrm{Id}\otimes\Delta)(\hat{\chi})&=1_{H}\otimes\chi'^{l}\otimes1_{H}\otimes\chi'_{l_1}\otimes1_{H}\otimes\chi'_{l_2}\\&=1_{H}\otimes\chi'^{l}\otimes1_{H}\otimes\chi'_{l}\otimes1_{H}\otimes1_{H'}+1_{H}\otimes\overline{\Rr}'^{m}\chi'^{l}\Rr'^{n}\otimes1_{H}\otimes\overline{\Rr}'_{m}\Rr'_{n}\otimes1_{H}\otimes\chi'_{l}\\&=\hat{\chi}_{12}+1_{H}\otimes\overline{\Rr}'^{m}\chi'^{l}\Rr'^{n}\otimes1_{H}\otimes\overline{\Rr}'_{m}\Rr'_{n}\otimes1_{H}\otimes\chi'_{l}\\&=\hat{\chi}_{12}+(\overline{\Rr}^{j}\otimes\overline{\Rr}'^{m}\otimes\overline{\Rr}_{j}\otimes\overline{\Rr}'_{m}\otimes1_{H}\otimes1_{H'})(1_{H}\otimes\chi'^{l}\otimes1_{H}\otimes1_{H'}\otimes1_{H}\otimes\chi'_{l})(\Rr^{k}\otimes\Rr'^{n}\otimes\Rr_{k}\otimes\Rr'_{n}\otimes1_{H}\otimes1_{H'})\\&=\hat{\chi}_{12}+\tilde{\Rr}^{-1}_{12}\hat{\chi}_{13}\tilde{\Rr}_{12},
%\end{split}
%\]
where in the third equality we use that $\overline{\Rr}'^{m}\Rr'^{n}\otimes\overline{\Rr}'_{m}\Rr'_{n}=\Rr'^{-1}\Rr'=1_{H'}\otimes1_{H'}$ and so $\tilde{\chi}$ satisfies \eqref{cqtr2}. Similarly, since $\chi$ satisfies \eqref{cqtr3}, then $\tilde{\chi}$ satisfies \eqref{cqtr3} and thus $H\otimes H'$ is a pre-Cartier quasitriangular bialgebra through $\tilde{\chi}$. 
Moreover, if $\Rr\chi=\chi^{\mathrm{op}}\Rr$, 
%and $\Rr'\chi'=\chi'^{\mathrm{op}}\Rr'$, 
i.e., $\Rr^{k}\chi^i\otimes\Rr_{k}\chi_{i}=\chi_{i}\Rr^{k}\otimes\chi^{i}\Rr_{k}$, 
%and $\Rr'^{j}\chi'^{l}\otimes\Rr'_{j}\chi'_{l}=\chi'_{l}\Rr'^{j}\otimes\chi'^{l}\Rr'_{j}$, 
we obtain that
\[
\begin{split}
    \tilde{\Rr}\tilde{\chi}&=(\Rr^{k}\otimes\Rr'^{j}\otimes\Rr_{k}\otimes\Rr'_{j})(\chi^{i}\otimes1_{H'}\otimes\chi_{i}\otimes1_{H'})
    =\Rr^{k}\chi^{i}\otimes\Rr'^{j}\otimes\Rr_{k}\chi_{i}\otimes\Rr'_{j}\\
    &=\chi_{i}\Rr^{k}\otimes\Rr'^{j}\otimes\chi^{i}\Rr_{k}\otimes\Rr'_{j}
    =(\chi_{i}\otimes1_{H'}\otimes\chi^{i}\otimes1_{H'})(\Rr^{k}\otimes\Rr'^{j}\otimes\Rr_{k}\otimes\Rr'_{j})=\tilde{\chi}^{\mathrm{op}}\tilde{\Rr},
\end{split}
\]
then $H\otimes H'$ is Cartier.
%\[
%\begin{split}
    %\tilde{\Rr}\hat{\chi}&=(\Rr^{k}\otimes\Rr'^{j}\otimes\Rr_{k}\otimes\Rr'_{j})(1_{H}\otimes\chi'^{l}\otimes1_{H}\otimes\chi'_{l})\\&=\Rr^{k}\otimes\Rr'^{j}\chi'^{l}\otimes\Rr_{k}\otimes\Rr'_{j}\chi'_{l}\\&=\chi_{i}\Rr^{k}\otimes\Rr'^{j}\otimes\chi^{i}\Rr_{k}\otimes\Rr'_{j}+\Rr^{k}\otimes\chi'_{l}\Rr'^{j}\otimes\Rr_{k}\otimes\chi'^{l}\Rr'_{j}\\&=(1_{H}\otimes\chi'_{l}\otimes1_{H}\otimes\chi'^{l})(\Rr^{k}\otimes\Rr'^{j}\otimes\Rr_{k}\otimes\Rr'_{j})\\&=\hat{\chi}^{\mathrm{op}}\tilde{\Rr}.
%\end{split}
%\]
\end{proof}

Observe that, if $(H',\Rr')$ is (pre-)Cartier with $\chi'=\chi'^{m}\otimes\chi'_{m}$, then $H\otimes H'$ is (pre-)Cartier with $\hat{\chi}:=1_{H}\otimes\chi'^{m}\otimes1_{H}\otimes\chi'_{m}$ by analogue computations. 
In view of Remark \ref{rmk:chivs} we can even take a linear combination %$\lambda\chi^m\otimes 1_{H'}\otimes\chi_m\otimes 1_{H'}+\lambda'1_H\otimes\chi^{'m}\otimes 1_H\otimes\chi'_m$ 
of $\tilde{\chi}$ and $\hat{\chi}$.
%In the proof of Proposition \ref{prop:tensor} we have shown that $\tilde{\chi}+\hat{\chi}$ is an infinitesimal $\tilde{R}$-matrix for $H\otimes H'$.
%Note that both $\tilde{\chi}=\chi^i\otimes1_{H'}\otimes\chi_{i}\otimes1_{H'}$ and $\hat{\chi}=1_{H}\otimes\chi'^{l}\otimes1_{H}\otimes\chi'_{l}\in H\otimes H'\otimes H\otimes H'$ are also infinitesimal $\tilde{R}$-matrix for $H\otimes H'$. In particular, $\tilde{\chi}$ satisfies respectively \eqref{cqtr1}, \eqref{cqtr2}, \eqref{cqtr3}, \eqref{eq:ctr2}, when so does $\chi=\chi^i\otimes\chi_i\in H\otimes H$. Analogously, $\hat{\chi}$ satisfies respectively \eqref{cqtr1}, \eqref{cqtr2}, \eqref{cqtr3}, \eqref{eq:ctr2}, when so does $\chi'=\chi'^l\otimes\chi'_l\in H'\otimes H'$. Thus, more generally, given two (quasi)triangular bialgebras, if one of them is (pre-)Cartier, then their tensor product is (pre-)Cartier.

\begin{example}
    Given Sweedler's Hopf algebra $(H,\Rr_{\lambda})$ with infinitesimal $\Rr$-matrix $\chi_{\alpha}=\alpha xg\otimes x$, with $\alpha\in\Bbbk$, and an arbitrary (quasi)triangular bialgebra $(H',\Rr')$, we obtain that $H\otimes H'$ is (quasi)triangular with $\tilde{\Rr}:=(\mathrm{Id}_{H}\otimes\tau_{H,H'}\otimes\mathrm{Id}_{H'})(\Rr_{\lambda}\otimes\Rr')$ and it is pre-Cartier with infinitesimal $\Rr$-matrix $\tilde{\chi_{\alpha}}:=\alpha xg\otimes1_{H'}\otimes x\otimes1_{H'}$.  
\end{example}

%\rd{Add an example by tesoring Sweedler by himself? Maybe deduce properties on $E(2)$ by seeing it as a quotient of $E(1)\otimes E(1)$ if possible? Ask Renda!?!}\textcolor{red}{separate paper?}

\subsection{Twisting pre-Cartier quasitriangular bialgebras}\label{secTwist}

First we recall the notion of Drinfel'd twist. An invertible element $\Ff\in H\otimes H$ is said to be a \textbf{Drinfel'd twist} if the (dual) $2$-cocycle condition and the normalization properties
\begin{align}
    (\Ff\otimes 1_{H})(\Delta\otimes\mathrm{Id}_H)(\Ff)&=(1_{H}\otimes\Ff)(\mathrm{Id}_H\otimes\Delta)(\Ff),\label{2-cocy}\\
    (\varepsilon\otimes\mathrm{Id}_H)(\Ff)&=1_{H}=(\mathrm{Id}_H\otimes\varepsilon)(\Ff),\label{norm}
\end{align}
are satisfied. This notion goes back to Drinfel'd \cite{Dr87}.

Given a Drinfel'd twist $\Ff$ we consider the linear map $\Delta_\Ff\colon H\rightarrow H\otimes H$ defined via
\begin{equation*}
    \Delta_\Ff(\cdot):=\Ff\Delta(\cdot)\Ff^{-1}.
\end{equation*}
We employ the short notation $\Delta_\Ff(h)=h_{1_\Ff}\otimes h_{2_\Ff}$ for $h\in H$.
Then, $H_\Ff:=(H,m,u,\Delta_\Ff,\varepsilon)$ is a bialgebra. If $H$ is a Hopf algebra with antipode $S$, then also $H_\Ff$ is a Hopf algebra with antipode $S_\Ff$ given by $S_{\Ff}(a)=US(a)U^{-1}$, for any $a\in H$, where $U:=\Ff^i S(\Ff_i)$ is invertible.

If $(H,\Rr)$ is a (quasi)triangular bialgebra, so is $H_\Ff$ with universal $\Rr$-matrix $\Rr_\Ff:=\Ff^\op\Rr\Ff^{-1}$, c.f. \cite[Theorem 2.3.4]{Majid-book}. Observe that every quasitriangular structure $\Rr$ on a bialgebra $H$ is a Drinfel’d twist. In fact, $\Rr$ is normalized and satisfies the 2-cocycle property since
\[
\Rr_{12}(\Delta\otimes\mathrm{Id}_H)(\Rr)\overset{\eqref{qtr3}}=\Rr_{12}\Rr_{13}\Rr_{23}\overset{\eqref{eq:QYB}}{=}\Rr_{23}\Rr_{13}\Rr_{12}\overset{\eqref{qtr2}}=\Rr_{23}(\mathrm{Id}_H\otimes\Delta)(\Rr).\]

Note that every left $H$-module is also a left $H_\Ff$-module (and viceversa), since the algebra structures of $H$ and $H_\Ff$ coincide. This defines an isomorphism of categories
\begin{equation*}
    \mathrm{Drin}_\Ff\colon{}_H\Mm\rightarrow{}_{H_\Ff}\Mm,\quad
    M\mapsto M_\Ff,
\end{equation*}
where $M_\Ff$ is the vector space $M$ endowed with the left $H_\Ff$-action instead of the given $H$-action. The inverse functor is denoted by $\mathrm{Drin}_\Ff^{-1}\colon{}_{H_\Ff}\Mm\rightarrow{}_H\Mm$, $M\mapsto M_{\Ff^{-1}}$. Note that $\Ff^{-1}$ is a Drinfel'd twist on $H_\Ff$. Since the comultiplications of $H$ and $H_\Ff$ are different, the monoidal product of ${}_{H_\Ff}\Mm$ is denoted by $M\otimes_\Ff N$ for objects $M,N$ in ${}_{H_\Ff}\Mm$, where $M\otimes_\Ff N$ equals $M\otimes N$ as a vector space but is understood as a left $H_\Ff$-module via
\begin{equation*}
    h\cdot(m\otimes_\Ff n):=\Delta_\Ff(h)\cdot(m\otimes n)=(h_{1_\Ff}\cdot m)\otimes(h_{2_\Ff}\cdot n)
\end{equation*}
for all $h\in H_\Ff$, $m\in M$ and $n\in N$. One can show that the functor $\mathrm{Drin}_\Ff$ is strong monoidal with natural isomorphisms
\begin{equation*}
    M_\Ff\otimes_\Ff N_\Ff\xrightarrow{\cong}(M\otimes N)_\Ff,\quad
    m\otimes_\Ff n\mapsto\Ff^{-1}\cdot(m\otimes n)=(\overline{\Ff}^i\cdot m)\otimes(\overline{\Ff}_i\cdot n),
\end{equation*}
for all $M,N$ objects in ${}_H\Mm$.
If $(H,\Rr)$ is quasitriangular and $\Ff$ a Drinfel'd twist on $H$ we have a braided strong monoidal equivalence $({}_H\Mm,\otimes,\sigma^\Rr)\cong({}_{H_F}\Mm,\otimes_\Ff,\sigma^{\Rr_\Ff})$ given by the Drinfel'd functor $\mathrm{Drin}_\Ff$, see e.g. \cite[Lemma XV.3.7]{Kassel}, \cite[Section 5.3]{AS14}. \\

Using the Drinfel'd twist $\Ff$, we can obtain an infinitesimal $\Rr$-matrix for $H_{\Ff}$.

\begin{theorem}\label{thm:twist}
Consider a pre-Cartier (quasi)triangular bialgebra $(H,\Rr,\chi)$ and a Drinfel'd twist $\Ff$ on $H$. Then, $(H_\Ff,\Rr_\Ff,\chi_\Ff)$ is a pre-Cartier (quasi)triangular bialgebra, where $\chi_\Ff:=\Ff\chi\Ff^{-1}$. If $(H,\Rr,\chi)$ is Cartier, so is $(H_\Ff,\Rr_\Ff,\chi_\Ff)$.
\end{theorem}

\begin{proof}
We have to prove that $\chi_{\Ff}$ satisfies \eqref{cqtr1}, \eqref{cqtr2} and \eqref{cqtr3}. We have that
\[
\chi_{\Ff}\Delta_{\Ff}(h)=\Ff\chi\Ff^{-1}\Ff\Delta(h)\Ff^{-1}=\Ff\chi\Delta(h)\Ff^{-1}=\Ff\Delta(h)\chi\Ff^{-1}=\Ff\Delta(h)\Ff^{-1}\Ff\chi\Ff^{-1}=\Delta_{\Ff}(h)\chi_{\Ff}
\]
for every $h\in H$, i.e., $\chi_{\Ff}$ satisfies \eqref{cqtr1}.
Next, 
\begin{align*}
    (\mathrm{Id}&\otimes\Delta_\Ff)(\chi_\Ff)
    =\Ff_{23}(\mathrm{Id}\otimes\Delta)(\Ff)(\mathrm{Id}\otimes\Delta)(\chi)(\mathrm{Id}\otimes\Delta)(\Ff^{-1})\Ff^{-1}_{23}\\
    &\overset{\eqref{cqtr2}}{=}\Ff_{23}(\mathrm{Id}\otimes\Delta)(\Ff)\chi_{12}(\mathrm{Id}\otimes\Delta)(\Ff^{-1})\Ff^{-1}_{23}
    +\Ff_{23}(\mathrm{Id}\otimes\Delta)(\Ff)\Rr^{-1}_{12}\chi_{13}\Rr_{12}(\mathrm{Id}\otimes\Delta)(\Ff^{-1})\Ff^{-1}_{23}\\
    &\overset{\eqref{2-cocy}}{=}\Ff_{12}(\Delta\otimes\mathrm{Id})(\Ff)\chi_{12}(\Delta\otimes\mathrm{Id})(\Ff^{-1})\Ff^{-1}_{12}
    +\Ff_{12}(\Delta\otimes\mathrm{Id})(\Ff)\Rr^{-1}_{12}\chi_{13}\Rr_{12}(\Delta\otimes\mathrm{Id})(\Ff^{-1})\Ff^{-1}_{12}\\
    &\overset{\eqref{cqtr1},\eqref{qtr1}}{=}\Ff_{12}\chi_{12}(\Delta\otimes\mathrm{Id})(\Ff)(\Delta\otimes\mathrm{Id})(\Ff^{-1})\Ff^{-1}_{12}
    +\Ff_{12}\Rr^{-1}_{12}(\Delta^\mathrm{op}\otimes\mathrm{Id})(\Ff)\chi_{13}(\Delta^\mathrm{op}\otimes\mathrm{Id})(\Ff^{-1})\Rr_{12}\Ff^{-1}_{12}\\
    &=(\chi_\Ff)_{12}
    +(\Rr^{-1}_\Ff)_{12}\Ff_{21}(\Delta^\mathrm{op}\otimes\mathrm{Id})(\Ff)\chi_{13}(\Delta^\mathrm{op}\otimes\mathrm{Id})(\Ff^{-1})\Ff_{21}^{-1}(\Rr_\Ff)_{12}\\
    &\overset{(*)}{=}(\chi_\Ff)_{12}
    +(\Rr^{-1}_\Ff)_{12}\Ff_{13}((\Ff_i)_1\otimes\Ff^i\otimes(\Ff_i)_2)\chi_{13}(\Delta^\mathrm{op}\otimes\mathrm{Id})(\Ff^{-1})\Ff_{21}^{-1}(\Rr_\Ff)_{12}\\
    &\overset{\eqref{cqtr1}}{=}(\chi_\Ff)_{12}
    +(\Rr^{-1}_\Ff)_{12}\Ff_{13}\chi_{13}((\Ff_i)_1\otimes\Ff^i\otimes(\Ff_i)_2)(\Delta^\mathrm{op}\otimes\mathrm{Id})(\Ff^{-1})\Ff_{21}^{-1}(\Rr_\Ff)_{12}\\
    &\overset{(**)}{=}(\chi_\Ff)_{12}
    +(\Rr^{-1}_\Ff)_{12}\Ff_{13}\chi_{13}\Ff_{13}^{-1}(\Rr_\Ff)_{12}
    =(\chi_\Ff)_{12}
    +(\Rr^{-1}_\Ff)_{12}(\chi_\Ff)_{13}(\Rr_\Ff)_{12},
\end{align*}
where in the equations $(*)$ and $(**)$ we used the $2$-cocycle property $\Ff_{12}(\Delta\otimes\mathrm{Id})(\Ff)=\Ff_{23}(\mathrm{Id}\otimes\Delta)(\Ff)$ with $\tau_{H,H}\otimes\mathrm{Id}_H$ applied to both sides.

Similarly, $(\Delta_\Ff\otimes\mathrm{Id})(\chi_\Ff)=(\chi_\Ff)_{23}
    +(\Rr^{-1}_\Ff)_{23}(\chi_\Ff)_{13}(\Rr_\Ff)_{23}$ is proven.
\begin{invisible}
\begin{align*}
    (\Delta_\Ff&\otimes\mathrm{Id})(\chi_\Ff)
    =\Ff_{12}(\Delta\otimes\mathrm{Id})(\Ff)(\Delta\otimes\mathrm{Id})(\chi)(\Delta\otimes\mathrm{Id})(\Ff^{-1})\Ff^{-1}_{12}\\
    &\overset{\eqref{cqtr3}}{=}\Ff_{12}(\Delta\otimes\mathrm{Id})(\Ff)\chi_{23}(\Delta\otimes\mathrm{Id})(\Ff^{-1})\Ff^{-1}_{12}
    +\Ff_{12}(\Delta\otimes\mathrm{Id})(\Ff)\Rr^{-1}_{23}\chi_{13}\Rr_{23}(\Delta\otimes\mathrm{Id})(\Ff^{-1})\Ff^{-1}_{12}\\
    &\overset{\eqref{2-cocy}}{=}\Ff_{23}(\mathrm{Id}\otimes\Delta)(\Ff)\chi_{23}(\mathrm{Id}\otimes\Delta)(\Ff^{-1})\Ff^{-1}_{23}
    +\Ff_{23}(\mathrm{Id}\otimes\Delta)(\Ff)\Rr^{-1}_{23}\chi_{13}\Rr_{23}(\mathrm{Id}\otimes\Delta)(\Ff^{-1})\Ff^{-1}_{23}\\
    &\overset{\eqref{cqtr1},\eqref{qtr1}}{=}(\chi_\Ff)_{23}
    +\Ff_{23}\Rr^{-1}_{23}(\mathrm{Id}\otimes\Delta^\mathrm{op})(\Ff)\chi_{13}(\mathrm{Id}\otimes\Delta^\mathrm{op})(\Ff^{-1})\Rr_{23}\Ff^{-1}_{23}\\
    &=(\chi_\Ff)_{23}
    +(\Rr^{-1}_\Ff)_{23}\Ff_{32}(\mathrm{Id}\otimes\Delta^\mathrm{op})(\Ff)\chi_{13}(\mathrm{Id}\otimes\Delta^\mathrm{op})(\Ff^{-1})\Ff^{-1}_{32}(\Rr_\Ff)_{23}\\
    &\overset{(***)}{=}(\chi_\Ff)_{23}
    +(\Rr^{-1}_\Ff)_{23}\Ff_{13}(\Ff^i_1\otimes\Ff_i\otimes\Ff^i_2)\chi_{13}(\mathrm{Id}\otimes\Delta^\mathrm{op})(\Ff^{-1})\Ff^{-1}_{32}(\Rr_\Ff)_{23}\\
    &\overset{\eqref{cqtr1}}{=}(\chi_\Ff)_{23}
    +(\Rr^{-1}_\Ff)_{23}\Ff_{13}\chi_{13}(\Ff^i_1\otimes\Ff_i\otimes\Ff^i_2)(\mathrm{Id}\otimes\Delta^\mathrm{op})(\Ff^{-1})\Ff^{-1}_{32}(\Rr_\Ff)_{23}\\
    &\overset{(****)}{=}(\chi_\Ff)_{23}
    +(\Rr^{-1}_\Ff)_{23}\Ff_{13}\chi_{13}\Ff^{-1}_{13}(\Rr_\Ff)_{23}
    =(\chi_\Ff)_{23}
    +(\Rr^{-1}_\Ff)_{23}(\chi_\Ff)_{13}(\Rr_\Ff)_{23},
\end{align*}
where in $(***)$ and $(****)$ we used the $2$-cocycle property $\Ff_{23}(\mathrm{Id}\otimes\Delta)(\Ff)=\Ff_{12}(\Delta\otimes\mathrm{Id})(\Ff)$ with $\mathrm{Id}_H\otimes\tau_{H,H}$ applied to both sides.
\end{invisible}

If $\Rr\chi=\chi^\mathrm{op}\Rr$, it follows that
$
\Rr_\Ff\chi_\Ff
=\Ff^\mathrm{op}\Rr\Ff^{-1}\Ff\chi\Ff^{-1}
=\Ff^\mathrm{op}\chi^\mathrm{op}\Rr\Ff^{-1}
=\chi_\Ff^\mathrm{op}\Rr_\Ff
$.
We have thus proved that $(H_\Ff,\Rr_\Ff,\chi_\Ff)$ is Cartier if $(H,\Rr,\chi)$ is.
\end{proof}

\begin{remark}
  Observe that, on the other hand, given a (quasi)triangular bialgebra $(H,\Rr)$ and a Drinfel'd twist $\Ff$ on $H$, if $(H_\Ff,\Rr_\Ff,\chi)$ is (pre-)Cartier, then $(H,\Rr,\Ff^{-1}\chi\Ff)$ is (pre-)Cartier. 
\end{remark}

\begin{remark}\label{rmk:Rtwist}
    If we have a (pre-)Cartier quasitriangular bialgebra $(H,\Rr,\chi)$, by using $\Rr$ as a Drinfel'd twist, we obtain immediately that $(H_{\Rr},\Rr_{\Rr},\chi_{\Rr})$ is (pre-)Cartier quasitriangular with $\Delta_{\Rr}=\Rr\Delta(\cdot)\Rr^{-1}=\Delta^{\mathrm{op}}$, $\Rr_{\Rr}=\Rr^{\mathrm{op}}$ and $\chi_{\Rr}=\Rr\chi\Rr^{-1}$. 
\end{remark}

\begin{example}
It is known (see \cite{AEG01}) that Sweedler's Hopf algebra has a 1-parameter family of Drinfel'd twists $\Ff_{t}\in H\otimes H$ of the form
\[
\textstyle \Ff_{t}:=1\otimes1+\frac{t}{2}xg\otimes x,
\]
for $t\in\Bbbk$. 
\begin{invisible}
 If we use the notation  $\chi_{\alpha}=\alpha(xg\otimes x)$, we have that   $\Ff_{t}:=1\otimes1+\chi_{t/2}$.
 Thus 
 \[
\begin{split}
(\Rr_{\lambda})_{\Ff_{t}}&=\Ff_{t}^{\mathrm{op}}\Rr_{\lambda}\Ff_{t}^{-1}=(1\otimes1+\chi_{t/2})\Rr_{\lambda}(1\otimes1-\chi_{t/2})\\
&=(1\otimes1+\chi_{t/2})(\Rr_{\lambda}-\Rr_{\lambda}\chi_{t/2})\\
&=\Rr_{\lambda}+\chi_{t/2}\Rr_{\lambda}-\Rr_{\lambda}\chi_{t/2}-\chi_{t/2}\Rr_{\lambda}\chi_{t/2}\\
&\overset{(*)}{=}\Rr_{\lambda}+[\chi_{t/2},\Rr_{\lambda}]+\chi_{t/2}\chi_{t/2}^\op\Rr_{\lambda}
=\Rr_{\lambda}+[\chi_{t/2},\Rr_{\lambda}]
\end{split}
\]
where in $(*)$ we used Remark \ref{rmk:Sw}.
\end{invisible}
We know that $(H,\Rr_{\lambda})$ is triangular for every $\lambda\in\Bbbk$, thus we have that $(H_{\Ff_{t}},(\Rr_{\lambda})_{\Ff_{t}})$ is triangular with
\[
\begin{split}
(\Rr_{\lambda})&_{\Ff_{t}}=\Ff_{t}^{\mathrm{op}}\Rr_{\lambda}\Ff_{t}^{-1}=(1\otimes1+\frac{t}{2}x\otimes xg)\Rr_{\lambda}(1\otimes1-\frac{t}{2}xg\otimes x)\\&=\Rr_{\lambda}-\frac{t}{4}(1\otimes1+1\otimes g+g\otimes 1-g\otimes g)(xg\otimes x)+\frac{t}{4}(x\otimes xg)(1\otimes1+1\otimes g+g\otimes1-g\otimes g)\\&=\Rr_{\lambda}+\frac{t}{2}(x\otimes xg+x\otimes x+xg\otimes xg-xg\otimes x).
\end{split}
\]
Furthermore, we have shown that, for every $\lambda\in\Bbbk$, $(H,\Rr_{\lambda})$ becomes pre-Cartier with $\chi_{\alpha}=\alpha xg\otimes x$, $\alpha\in\Bbbk$. Hence we have that $(H_{\Ff_{t}},(\Rr_{\lambda})_{\Ff_{t}})$ becomes pre-Cartier with 
\[
(\chi_{\alpha})_{\Ff_{t}}=\Ff_{t}\chi_{\alpha}\Ff_{t}^{-1}=\alpha(1\otimes1+\frac{t}{2}xg\otimes x)(xg\otimes x)(1\otimes 1-\frac{t}{2}xg\otimes x)=\alpha xg\otimes x=\chi_{\alpha}.
\]
Thus, $\chi_{\alpha}$ is invariant under Drinfel'd twist deformations. Furthermore, by Remark \ref{rmk:Rtwist}, we can also take $\Ff=\Rr_{\lambda}$ and then we obtain that $(H_{\Rr_\lambda},\Rr_\lambda^{\mathrm{op}},\Rr_\lambda\chi_{\alpha}\Rr_\lambda^{\mathrm{op}})$ is pre-Cartier. But $\chi_{\alpha}=\alpha xg\otimes x$ and then we can show $\Rr_{\lambda}\chi_{\alpha}\Rr_{\lambda}^{\mathrm{op}}=
%\alpha\Rr_{0}(xg\otimes x)\Rr_{0}^{\mathrm{op}}=\frac{\alpha}{4}(1_{H}\otimes1_{H}+g\otimes1_{H}+1_{H}\otimes g-g\otimes g)(xg\otimes x)(1_{H}\otimes1_{H}+g\otimes1_{H}+1_{H}\otimes g-g\otimes g)=\frac{\alpha}{4}(xg\otimes x+x\otimes x+xg\otimes xg-x\otimes xg-x\otimes x-xg\otimes x-x\otimes xg+xg\otimes xg-xg\otimes xg-x\otimes xg-xg\otimes x+x\otimes x-x\otimes xg-xg\otimes xg-x\otimes x+xg\otimes x)=
-\frac{\alpha}{4}(x\otimes xg)$. 
\end{example}

\subsection{The role of Hochschild cohomology of coalgebras}\label{secCohom} In this subsection we show one of the most relevant properties of the infinitesimal $\Rr$-matrix $\chi$ of a pre-Cartier quasitriangular bialgebra $(H,\Rr,\chi)$, that is, it is always a 2-cocycle in Hochschild cohomology. Furthermore, if $H$ is a Cartier triangular Hopf algebra on a field $\Bbbk$ with $\mathrm{char}\left( \Bbbk \right) \neq 2$, then $\chi$ is a 2-coboundary.
\\

We use \cite[XVIII.5]{Kassel} for the notion of (Hochschild) cohomology for a coalgebra. Let $(H, m, u,\Delta,\varepsilon)$ be a bialgebra and regard $\Bbbk $ as an $H$-bicomodule via the unit $u$, i.e., with left and right coactions $k\mapsto1_{H}\otimes k$ and $k\mapsto k\otimes1_{H}$, respectively. Then, we can consider the cobar complex of $H$
$$\xymatrix{\Bbbk\ar[r]^-{b^0}& H\ar[r]^-{b^1}& H\otimes H\ar[r]^-{b^2}& H\otimes H\otimes H\ar[r]^-{b^3}&\cdots }$$
The differential $b^n:H^{\otimes n}\to H^{\otimes n+1}$ is given by $b^n=\sum_{i=0}^{n+1}(-1)^i\delta_n^i$ where $\delta^i_n : H^{\otimes n}\to H^{\otimes n+1}$ are the linear maps
\[
\delta^i_n(x_1\otimes \cdots\otimes x_n) = \begin{cases}
1\otimes x_1\otimes\cdots\otimes x_n, \text{ if } i=0\\
x_1\otimes\cdots\otimes x_{i-1}\otimes\Delta(x_i)\otimes x_{i+1}\otimes\cdots\otimes x_n, \text{ if } 1\leq i \leq n\\
x_1\otimes \cdots\otimes x_{n}\otimes 1, \text{ if } i=n+1.
\end{cases}
\]
Here, if $n=0$, we set $H^{\otimes 0}=\Bbbk$ and $\delta^0_0 (1_{\Bbbk})=\delta^1_0(1_{\Bbbk})=1_{H}$. For instance, for small values of $n$, we have $b^0(k)=0$ for $k\in\Bbbk$, $b^1(x)=1\otimes x-\Delta(x)+x\otimes 1$ for $x\in H$, $b^2(x\otimes y)=1\otimes x\otimes y-(\Delta\otimes\mathrm{Id}_H)(x\otimes y)+(\mathrm{Id}_H\otimes \Delta)(x\otimes y)-x\otimes y\otimes 1$, for $x\otimes y\in H\otimes H$, and so on. The elements in $\mathrm{Z}^n(H):=\mathrm{Ker}(b^n)$ are the \textit{Hochschild $n$-cocycles}, while the elements in $\mathrm{B}^n(H):=\mathrm{Im}(b^{n-1})$ are the \textit{Hochschild $n$-coboundaries}. The $n$-th cohomology group is $\mathrm{H}^n( H)=\frac{\mathrm{Z}^n(H)}{\mathrm{B}^n(H)}$.

\begin{theorem}\label{thm:cQYB}
Let $(H,\Rr,\chi)$ be a pre-Cartier quasitriangular bialgebra. Then the following equivalent statements hold.
\begin{enumerate}[i)]
\item $\chi$ is a Hochschild $2$-cocycle, i.e., 
$\chi_{12}+(\Delta\otimes\mathrm{Id})(\chi)=\chi_{23}+(\mathrm{Id}\otimes\Delta)(\chi)$;
\medskip
\item the \textbf{infinitesimal QYB equation} 
\begin{equation}\label{cQYB}
\begin{split}
    \Rr_{12}\chi_{12}\Rr_{13}\Rr_{23}
    &+\Rr_{12}\Rr_{13}\chi_{13}\Rr_{23}
    +\Rr_{12}\Rr_{13}\Rr_{23}\chi_{23}\\
    &=\Rr_{23}\chi_{23}\Rr_{13}\Rr_{12}
    +\Rr_{23}\Rr_{13}\chi_{13}\Rr_{12}
    +\Rr_{23}\Rr_{13}\Rr_{12}\chi_{12}
\end{split}
\end{equation}
holds true;
\medskip
\item $\Rr_{12}^{-1}\chi_{13}\Rr_{12}=\Rr_{23}^{-1}\chi_{13}\Rr_{23}$.
\footnote{As we will see afterwards, this condition is equivalent to say that $\chi$ belongs to the cotensor
product $H\square_HH=\{h^i\otimes h_i\in H\otimes H~|~\rho^l(h^i)\otimes h_i=h^i\otimes\rho^l(h_i)\}$, where
$\rho^r,\rho^l\colon H\to H\otimes H$ are coactions defined by
$\rho^r(h)=\Rr^{-1}(h\otimes 1)\Rr$ and $\rho^l(h)=\Rr^{-1}(1\otimes h)\Rr$.}
\end{enumerate}
\end{theorem}
\begin{proof}
Consider a pre-Cartier quasitriangular bialgebra $(H,\Rr,\chi)$.
The equivalence between the above statements {\it i)} and {\it iii)} follows by the identities \eqref{cqtr2} and \eqref{cqtr3}.
\begin{invisible}
We prove that {\it ii)} is also equivalent to {\it i)}.
Multiplying \eqref{cQYB} with $\Rr_{23}^{-1}\Rr_{13}^{-1}\Rr_{12}^{-1}\overset{\mathrm{QYB}}{=}\Rr_{12}^{-1}\Rr_{13}^{-1}\Rr_{23}^{-1}$ shows that {\it ii)} is equivalent to
\begin{align*}
    \Rr_{23}^{-1}\Rr_{13}^{-1}\chi_{12}\underbrace{\Rr_{13}\Rr_{23}}_{\overset{\eqref{qtr3}}{=}(\Delta\otimes\mathrm{Id})(\Rr)}
    +\underbrace{\Rr_{23}^{-1}\chi_{13}\Rr_{23}
    +\chi_{23}}_{\overset{\eqref{cqtr3}}{=}(\Delta\otimes\mathrm{Id})(\chi)}
    =\Rr_{12}^{-1}\Rr_{13}^{-1}\chi_{23}\underbrace{\Rr_{13}\Rr_{12}}_{\overset{\eqref{qtr2}}{=}(\mathrm{Id}\otimes\Delta)(\Rr)}
    +\underbrace{\Rr_{12}^{-1}\chi_{13}\Rr_{12}
    +\chi_{12}}_{\overset{\eqref{cqtr2}}{=}(\mathrm{Id}\otimes\Delta)(\chi)}.
\end{align*}
Using \eqref{cqtr1} and again \eqref{qtr2},\eqref{qtr3} it follows that
the above is equivalent to {\it i)}.    
\end{invisible}
Moreover,
\begin{align*}
   \chi_{12}\Rr_{13}\Rr_{23}
    &\overset{\eqref{qtr3}}{=}\chi_{12}(\Delta\otimes\mathrm{Id})(\Rr)
    \overset{\eqref{cqtr1}}{=}(\Delta\otimes\mathrm{Id})(\Rr)\chi_{12}
    \overset{\eqref{qtr3}}{=}\Rr_{13}\Rr_{23}\chi_{12}
\end{align*}
so that 
\begin{align*}
    \Rr_{12}\chi_{12}\Rr_{13}\Rr_{23}
    =\Rr_{12}\Rr_{13}\Rr_{23}\chi_{12}\overset{\eqref{eq:QYB}}{=}\Rr_{23}\Rr_{13}\Rr_{12}\chi_{12}
\end{align*}
and similarly one verifies that $\Rr_{23}\chi_{23}\Rr_{13}\Rr_{12}=\Rr_{12}\Rr_{13}\Rr_{23}\chi_{23}$. Thus, assumption {\it ii)} is equivalent to 
$\Rr_{12}\Rr_{13}\chi_{13}\Rr_{23}=\Rr_{23}\Rr_{13}\chi_{13}\Rr_{12}$,
which itself is equivalent (via multiplication with
$\Rr_{23}^{-1}\Rr_{13}^{-1}\Rr_{12}^{-1}=\Rr_{12}^{-1}\Rr_{13}^{-1}\Rr_{23}^{-1}$)
to {\it iii)}.

We continue to prove that {\it iii)} holds in general, which then implies that {\it i)} and {\it ii)} are satisfied, as well, by the previously shown equivalence.
On the one hand we have
\begin{eqnarray*}
\mathcal{R}_{23}\left( \chi _{12}+\mathcal{R}_{12}^{-1}\chi _{13}\mathcal{R}%
_{12}\right)  
%&=&\mathcal{R}_{23}\left( \chi _{12}+\mathcal{R}_{12}^{-1}\chi_{13}\mathcal{R}_{12}\right)  \\
&\overset{\eqref{cqtr2}}{=}&\mathcal{R}_{23}\left( \mathrm{Id}_{H}\otimes
\Delta \right) \left( \chi \right)  \\
&\overset{\eqref{qtr1}}{=}&\left( \mathrm{Id}_{H}\otimes \Delta ^{\mathrm{op}%
}\right) \left( \chi \right) \mathcal{R}_{23} \\
&=&\left( \mathrm{Id}_{H}\otimes \tau \right) \left( \left( \mathrm{Id}%
_{H}\otimes \Delta \right) \left( \chi \right) \right) \mathcal{R}_{23} \\
&\overset{\eqref{cqtr2}}{=}&\left( \mathrm{Id}_{H}\otimes \tau \right) \left(
\chi _{12}+\mathcal{R}_{12}^{-1}\chi _{13}\mathcal{R}_{12}\right) \mathcal{R}%
_{23} \\
&=&\left( \chi _{13}+\mathcal{R}_{13}^{-1}\chi _{12}\mathcal{R}_{13}\right)
\mathcal{R}_{23}
\end{eqnarray*}%
so that $\mathcal{R}_{23}\chi _{12}+\mathcal{R}_{23}\mathcal{R}%
_{12}^{-1}\chi _{13}\mathcal{R}_{12}=\chi _{13}\mathcal{R}_{23}+\mathcal{R}%
_{13}^{-1}\chi _{12}\mathcal{R}_{13}\mathcal{R}_{23}$. 
Since we already proved that $\chi _{12}\mathcal{R}_{13}\mathcal{R}_{23}=\mathcal{R}_{13}\mathcal{R}_{23}\chi _{12}$, the last summand becomes $\mathcal{R}_{23}\chi _{12}$ 
so that the equality is equivalent to $\mathcal{R}_{23}%
\mathcal{R}_{12}^{-1}\chi _{13}\mathcal{R}_{12}=\chi _{13}\mathcal{R}_{23}$,
i.e., $\mathcal{R}_{12}^{-1}\chi _{13}\mathcal{R}_{12}=\mathcal{R}%
_{23}^{-1}\chi _{13}\mathcal{R}_{23}.$
\end{proof}

The following example motivates the terminology ``infinitesimal QYB equation'' a posteriori.

\begin{example}
Let $\tilde{H}=H[[\hbar]]$ be a trivial topological bialgebra with arbitrary quasitriangular structure $\tilde{\Rr}=\Rr(1\otimes 1+\hbar\chi+\mathcal{O}(\hbar^2))\in(H\otimes H)[[\hbar]]$. From Proposition \ref{prop:h-bar} we know that $\Rr\in H\otimes H$ is a quasitriangular structure and $\chi$ is an infinitesimal $\Rr$-matrix for $(H,\Rr)$. Thus, Theorem~\ref{thm:cQYB} implies that 
\begin{equation*}
\begin{split}
    \Rr_{12}\chi_{12}\Rr_{13}\Rr_{23}
    &+\Rr_{12}\Rr_{13}\chi_{13}\Rr_{23}
    +\Rr_{12}\Rr_{13}\Rr_{23}\chi_{23}\\
    &=\Rr_{23}\chi_{23}\Rr_{13}\Rr_{12}
    +\Rr_{23}\Rr_{13}\chi_{13}\Rr_{12}
    +\Rr_{23}\Rr_{13}\Rr_{12}\chi_{12}.
\end{split}
\end{equation*}
This is precisely the quantum Yang-Baxter equation \eqref{eq:QYB} of $\tilde{\Rr}$ in order $\hbar^1$. 
\end{example}

\begin{example}
If $(\mathfrak{g},[\cdot,\cdot],r)$ is a quasitriangular Lie bialgebra we have seen in Example~\ref{exa:chicntr} 3) that $(U\mathfrak{g},\Rr,\chi)$ is a Cartier triangular bialgebra with $\Rr=1\otimes 1$ and $\chi=r+r^{\mathrm{op}}$. By Theorem~\ref{thm:cQYB} $\chi$ is a Hochschild $2$-cocycle and \eqref{cQYB} is satisfied.
\end{example}

In particular, all infinitesimal $\Rr$-matrices $\chi_\alpha$ on Sweedler's Hopf algebra
(see Proposition \ref{prop:Sweedler}) are Hochschild $2$-cocycles and
satisfy \eqref{cQYB}. 
In fact, we would like to see that Theorem \ref{thm:cQYB}  provides an alternative proof of the classification of infinitesimal $\Rr$-matrices on Sweedler's Hopf algebra. To this aim we first need the following result.

\begin{proposition}
\label{prop:projtoab}
 Let $(H,\Rr,\chi)$ be a pre-Cartier quasitriangular Hopf algebra endowed with a Hopf algebra projection $\pi:H\to L$ onto a commutative Hopf algebra $L$. Then, we have $(\pi\otimes\pi)(\chi)\in P(L)\otimes P(L)$, $(\mathrm{Id}\otimes\pi)(\chi)\in\mathscr{Z}(H)\otimes L$ and $(\pi\otimes\mathrm{Id})(\chi)\in L\otimes\mathscr{Z}(H)$.
\end{proposition}
\begin{proof}
Since $H$ is pre-Cartier, by Proposition \ref{prop:image}, so is $\pi(H)=L$ and the corresponding infinitesimal $\Rr$-matrix is  $(\pi\otimes\pi)(\chi)$. By Example \ref{exa:chicntr}\,2), we get $(\pi\otimes\pi)(\chi)\in P(L)\otimes P(L)$. Now for every $h\in H$ the equality \eqref{cqtr1} means $\chi^i h_1\otimes\chi_i h_2=h_1\chi^i\otimes h_2\chi_i$. This is equivalent to  $\chi^i h\otimes\chi_i=h_1\chi^i\otimes h_2\chi_iS(h_3)$ so that, by applying $\mathrm{Id}\otimes\pi$ on both sides and using the fact that $L$ is commutative, we get $\chi^i h\otimes\pi(\chi_i)=h\chi^i\otimes \pi(\chi_i)$ for every $h\in H$, i.e., $\chi^i\otimes \pi(\chi_i)\in\mathscr{Z}(H)\otimes L$. 
Similarly, by rewriting $\chi^i h_1\otimes\chi_i h_2=h_1\chi^i\otimes h_2\chi_i$ as $S(h_1)\chi^i h_2\otimes\chi_i h_3=\chi^i\otimes h\chi_i$ and by applying $\pi\otimes\mathrm{Id}$ we conclude that $\pi(\chi^i)\otimes \chi_i\in L\otimes\mathscr{Z}(H)$.
\end{proof}

\begin{remark}
We include here the announced alternative proof of Proposition \ref{prop:Sweedler}, where a classification of infinitesimal $\Rr$-matrices on Sweedler's Hopf algebra $H$ is given, that takes advantages of the new techniques developed so far. 
The Hopf algebra $H$ has a Hopf algebra projection $\pi:H\to H_0:=\Bbbk\langle g\rangle$ defined on the basis by setting $\pi (x^mg^a):=\delta_{m,0}g^a$. By Proposition \ref{prop:projtoab}, we have $(\pi\otimes\pi)(\chi)\in P(H_0)\otimes P(H_0)=0$. Since $H=H_{0}\oplus xH_{0}$ we get $\chi\in xH_{0}\otimes H_{0}+H_{0}\otimes xH_{0}+xH_{0}\otimes xH_{0}$.
%Now, if we write $\chi\in H\otimes H$ as
%$\chi=a(1\otimes 1)+b(1\otimes g)+c(1\otimes x)+d(1\otimes xg)+e(g\otimes 1)+f(g\otimes g)+h(g\otimes x)+i(g\otimes xg)+l(x\otimes 1)+m(x\otimes g)+n(x\otimes x)+o(x\otimes xg)+p(xg\otimes 1)+q(xg\otimes g)+r(xg\otimes x)+s(xg\otimes xg)$
%as done in the proof of Proposition %\ref{prop:Sweedler}, then
%$(\pi\otimes\pi)(\chi)=a(1\otimes 1)+b(1\otimes g)+e(g\otimes 1)+f(g\otimes g)$ so that $a=b=e=f=0$ and hence
%$\chi=c(1\otimes x)+d(1\otimes xg)+h(g\otimes x)+i(g\otimes xg)+l(x\otimes 1)+m(x\otimes g)+n(x\otimes x)+o(x\otimes xg)+p(xg\otimes 1)+q(xg\otimes g)+r(xg\otimes x)+s(xg\otimes xg)$.
%But we can say something more. 
Still by Proposition \ref{prop:projtoab}, we have $(\mathrm{Id}\otimes \pi)(\chi)\in \mathscr{Z}(H)\otimes H_0$. A direct computation shows that, since  $\mathrm{char}(\Bbbk)\neq2$, one has $\mathscr{Z}(H)=\Bbbk1$, thus $(\mathrm{Id}\otimes\pi)(\chi)\in \Bbbk1\otimes H_0$. On the other hand, $(\mathrm{Id}\otimes\pi)(\chi)\in xH_{0}\otimes H_{0}$.
%we get $\chi^i\otimes \pi(\chi_i)=l(x\otimes 1)+m(x\otimes g)+p(xg\otimes 1)+q(xg\otimes g)$. 
Thus, we get $\chi\in H_{0}\otimes xH_{0}+xH_{0}\otimes xH_{0}$.
%we get $l=m=p=q=0$ and hence
%$\chi=c(1\otimes x)+d(1\otimes xg)+h(g\otimes x)+i(g\otimes xg)+n(x\otimes x)+o(x\otimes xg)+r(xg\otimes x)+s(xg\otimes xg)$.
Once more  by Proposition \ref{prop:projtoab}, we conclude that $(\pi\otimes\mathrm{Id})(\chi)\in H_0\otimes\mathscr{Z}(H)=H_0\otimes\Bbbk1$. But $(\pi\otimes\mathrm{Id})(\chi)\in H_{0}\otimes xH_{0}$, so $\chi\in xH_{0}\otimes xH_{0}$. 
%$c=d=h=i=0$ and hence
%$\chi=n(x\otimes x)+o(x\otimes xg)+r(xg\otimes x)+s(xg\otimes xg)$.
Since $(H,\Rr)$ is pre-Cartier triangular, by Theorem \ref{thm:cQYB} we get that $\chi$ is a Hochschild $2$-cocycle, i.e.,
$\chi_{12}+(\Delta\otimes\mathrm{Id})(\chi)=\chi_{23}+(\mathrm{Id}\otimes\Delta)(\chi)$. If we write this equality explicitly on $\chi=xh\otimes xh'$, for $h,h'\in H_{0}$, we get 
%$n=o=s=0$ and hence 
$\chi=\alpha xg\otimes x$, $\alpha\in\Bbbk$. Note that in this proof we did not use the specific form of $\Rr_\lambda$.

 Note also that, e.g. by  \cite[Lemma 2.13]{Da13}, the second cohomology group $\mathrm{H}^2(H)$ of Sweedler's Hopf algebra is one-dimensional and it is generated by $xg\otimes x$ (in loc. cit. the generator is $x\otimes gx=-x\otimes xg$ because the given comultiplication is the opposite of the one we are considering on $x$). In particular, $xg\otimes x$ is not a $2$-coboundary.
 \end{remark}

Our next aim is to investigate whether the infinitesimal $\Rr$-matrix $\chi$ of a Cartier triangular Hopf algebra is a coboundary. To this 
 end we first need to obtain some relations that hold for $\chi$ and that can be deduced by employing two coactions.

\begin{definition}\label{def:coact}
Given a quasitriangular bialgebra $\left( H,\mathcal{R} \right) $ we define the maps $\rho^r\colon H\to H\otimes H$ and $\rho^l\colon H\to H\otimes H$
by setting 
\begin{align*}
    \rho^r(a)&=\Rr^{-1}(a\otimes 1_H)\Rr=\overline{\Rr}^ia\Rr^j\otimes \overline{\Rr}_i\Rr_j;\\
    \rho^l(a)&=\Rr^{-1}(1_H\otimes a)\Rr=\overline{\Rr}^i\Rr^j\otimes\overline{\Rr}_iaR_j.
\end{align*}
\end{definition}

\begin{lemma}
 Let $\left( H,\mathcal{R} \right) $ be a quasitriangular bialgebra. Then, the maps  $\rho^r$ and $\rho^l$ define a right and left coaction on $H$, respectively.
\end{lemma}

\begin{proof}
The map $\rho^r$ is a right coaction on $H$ as
\[
(\mathrm{Id}_{H}\otimes\varepsilon)\rho^r(a)=\overline{\Rr}^ia\Rr^j\otimes\varepsilon(\overline{\Rr}_i\Rr_j)=\overline{\Rr}^ia\Rr^j\otimes\varepsilon(\overline{\Rr}_i)\varepsilon(\Rr_j)=a\otimes1_{\Bbbk}
\]
and
\[
\begin{split}
    (\mathrm{Id}\otimes\Delta)\rho^r(a)
    &=\overline{\Rr}^ia\Rr^j\otimes\Delta(\overline{\Rr}_i\Rr_j)
    =\overline{\Rr}^ia\Rr^j\otimes\Delta(\overline{\Rr}_i)\Delta(\Rr_j)\\
    &=(\mathrm{Id}\otimes\Delta)(\Rr^{-1})(a\otimes 1_{H}\otimes 1_{H})(\mathrm{Id}\otimes\Delta)(\Rr)\\
    &\overset{\eqref{qtr2}}=(\Rr_{13}\Rr_{12})^{-1}(a\otimes1_{H}\otimes1_{H})\Rr_{13}\Rr_{12}\\
    &=\Rr_{12}^{-1}\Rr_{13}^{-1}(a\otimes1_{H}\otimes1_{H})\Rr_{13}\Rr_{12}\\
    &=(\overline{\Rr}^i\overline{\Rr}^j\otimes\overline{\Rr}_i\otimes\overline{\Rr}_j)(a\otimes1_{H}\otimes1_{H})(\Rr^k\Rr^\ell\otimes \Rr_\ell\otimes \Rr_k)\\
    &=\overline{\Rr}^i\overline{\Rr}^ja\Rr^k\Rr^\ell\otimes \overline{\Rr}_i\Rr_\ell\otimes\overline{\Rr}_j\Rr_k
    =(\rho^r\otimes\mathrm{Id})\rho^r(a).
\end{split}
\]
Similarly, one can prove that  $\rho^l$ is a left coaction on $H$.
\end{proof}

As a consequence, given a pre-Cartier quasitriangular bialgebra $(H,\Rr,\chi)$, we can rewrite \eqref{cqtr2} and \eqref{cqtr3} as 
\begin{align}
    (\mathrm{Id}\otimes\Delta)(\chi)
    &=\chi\otimes 1_{H}+(\rho^r\otimes\mathrm{Id})(\chi), \label{eq:CC3'2}\\
      (\Delta\otimes\mathrm{Id})(\chi)
    &=1_{H}\otimes\chi+(\mathrm{Id}\otimes\rho^l)(\chi).  \label{eq:CC4'2}
\end{align}
\begin{comment}
\textcolor{gray}{Note that by applying $\mathrm{Id}\otimes\varepsilon\otimes\mathrm{Id}$ to \eqref{eq:CC3'2} and \eqref{eq:CC4'2}, and multiplying, we get 
\begin{align}
(\mathrm{Id}\otimes \varepsilon)(\chi)&=0;%label{eq:eps-chi1}
\\
%$\chi^i\otimes  \chi_i=\chi^i\otimes\varepsilon(\chi_i) 1_H+\chi^i\otimes \chi_i$, thus 
(\varepsilon\otimes\mathrm{Id})(\chi)&=0.  %label{eq:eps-chi2}
\end{align}}
\end{comment}
%In fact we have that 
%\[
%(\epsilon\otimes\mathrm{id}_{H})\rho^l(a)=\varepsilon(\overline{\Rr}^i\Rr^j)\otimes\overline{\Rr}_ia\Rr_j=\epsilon(\overline{\Rr}^i)\epsilon(\Rr^j)\otimes \overline{\Rr}_ia\Rr_j=1_{\Bbbk}\otimes a
%\]
%and
%\[
%\begin{split}
    %(\Delta\otimes\mathrm{id})\rho^{l}(a)
    %&=\Delta(\overline{\Rr}^i\Rr^j)\otimes\overline{\Rr}_ia\Rr_j\\
    %&=(\Delta\otimes\mathrm{id})(\Rr^{-1})(1_{H}\otimes1_{H}\otimes a)(\Delta\otimes\mathrm{id})(\Rr)\\
    %&=(\Rr_{13}\Rr_{23})^{-1}(1_{H}\otimes1_{H}\otimes a)(\Rr_{13}\Rr_{23})\\
    %&=\Rr_{23}^{-1}\Rr_{13}^{-1}(1_{H}\otimes1_{H}\otimes a)(\Rr_{13}\Rr_{23})\\
    %&=\overline{\Rr}^j\Rr^k\otimes\overline{\Rr}^i\Rr^\ell\otimes \overline{\Rr}_i\overline{\Rr}_ja\Rr_k\Rr_\ell\\
    %&=(\mathrm{id}\otimes\rho^{l})\rho^{l}(a).
%\end{split}
%\]
%Similarly by applying $\mathrm{Id}\otimes\varepsilon\otimes\mathrm{Id}$ to \eqref{eq:CC4'2} and multiplying, we have 
%$\chi^{i}\otimes\chi_{i}=1_{H}\otimes\varepsilon(\chi^{i})\chi_{i}+\chi^{i}\otimes\chi_{i}$ and then

 Recall that if $(H, \Rr,\chi)$ is pre-Cartier quasitriangular, then by Theorem \ref{thm:cQYB} $\chi$ is a 2-cocycle and the property $\Rr^{-1}_{12}\chi_{13}\Rr_{12}=\Rr^{-1}_{23}\chi_{13}\Rr_{23}$ can be reformulated as $(\rho^r\otimes\mathrm{Id})(\chi)=(\mathrm{Id}\otimes\rho^l)(\chi)$.\medskip
 
Until the end of the section we consider $H$ as Hopf algebra with antipode $S$. We are now ready to show that for a pre-Cartier triangular Hopf algebra the right and left coactions defined above satisfy the following relations. 

\begin{lemma}\label{lem:chiS1}
Let $(H,\Rr,\chi)$ be a pre-Cartier triangular Hopf algebra. Then, for all $a\in H$, the following properties hold: 
\begin{eqnarray}
\rho^l(S(a))&=&\tau(S\otimes S)\rho^r(a);  \label{eq:Striang2} \\
%(\rho^r\otimes\mathrm{Id})(\chi)&=&(\mathrm{Id}\otimes\rho^l)(\chi);
%\label{eq:chibalanced2} \\
(m\otimes\mathrm{Id})(S\otimes\rho^{l})(\chi)&=&-\chi.  \label{eq:chiS2}
\end{eqnarray}
Furthermore, if $(H,\Rr,\chi)$ is Cartier, then we have
\begin{eqnarray}
 \tau(S\otimes S)(\chi)&=&\chi \label{eq:chiSS2}
\end{eqnarray}
in addition.
\end{lemma}

\begin{proof}
Recall that $\Rr^{-1}=\Rr^{\mathrm{op}}$. For all $a\in H$ we have that 
\[
\begin{split}
    \rho^l(S(a))
    &=\Rr_{i}\Rr^{j}\otimes \Rr^{i}S(a)\Rr_{j}
    \overset{\eqref{eq:SotimesS(R)}}=S(\Rr_{i})S(\Rr^{j})\otimes S(\Rr^{i})S(a)S(\Rr_{j})\\
    &=S(\Rr^{j}\Rr_{i})\otimes S(\Rr_{j}a\Rr^{i})
    =\tau(S\otimes S)\rho^r(a).
\end{split}
\]
Next, from \eqref{eq:eps-chi2} we have that
$\varepsilon(\chi^{i})\otimes\chi_{i}=0$ and then
\[
\begin{split}
0&=\varepsilon(\chi^{i})1_{H}\otimes\chi_{i}=(m\otimes\mathrm{Id})(S\otimes\mathrm{Id}\otimes\mathrm{Id})(\Delta\otimes\mathrm{Id})(\chi)\\&\overset{\eqref{eq:CC4'2}}=(m\otimes\mathrm{Id})(S\otimes\mathrm{Id}\otimes\mathrm{Id})(1_{H}\otimes\chi+(\mathrm{Id}\otimes\rho^{l})(\chi))\\&=\chi+(m\otimes\mathrm{Id})(S\otimes\rho^{l})(\chi).
\end{split}
\]
Thus, we get $(m\otimes\mathrm{Id})(S\otimes\rho^{l})(\chi)=-\chi$, that is, $S(\chi^{i})\Rr_{j}\Rr^{k}\otimes \Rr^{j}\chi_{i}\Rr_{k}=-\chi^{i}\otimes\chi_{i}$. Note that analogously from \eqref{eq:eps-chi1} we have $\chi^i\otimes\varepsilon(\chi_i)=0$ and then $(\mathrm{Id}\otimes m)(\rho^r\otimes S)(\chi)=-\chi$. \\
%$-\chi=(m\otimes\mathrm{Id})(S\otimes\mathrm{Id})(\mathrm{Id}\otimes\rho^l)(\chi)=(m\otimes\mathrm{Id})(S\otimes\mathrm{Id})(\rho^r\otimes\mathrm{Id})(\chi)$.
Furthermore, if $H$ is Cartier, i.e., \eqref{eq:ctr2} holds, then we have that 
\[
\begin{split}
    (S\otimes S)(\chi)
    &=S(\chi^{i})\otimes S(\chi_{i})=S(\chi^{i})\varepsilon(\Rr_{j}\Rr^{k})1_{H}\otimes S(\Rr^{j}\chi_{i}\Rr_{k})\\
    &=S(\chi^{i})\Rr_{j1}\Rr^{k}_{1}S(\Rr_{j2}\Rr^{k}_{2})\otimes S(\Rr^{j}\chi_{i}\Rr_{k})\\
    &\overset{\eqref{qtr2},\eqref{qtr3}}=S(\chi^{i})\Rr_{j}\Rr^{k}S(\Rr_{m}\Rr^{\ell})\otimes S(\Rr^{m}\Rr^{j}\chi_{i}\Rr_{k}\Rr_{\ell})\\
    &\overset{\eqref{eq:chiS2}}{=}-\chi^{i}S(\Rr_{m}\Rr^{\ell})\otimes S(\Rr^{m}\chi_{i}\Rr_{\ell})\\
    &=-\chi^{i}S(\Rr^{\ell})S(\Rr_{m})\otimes S(\Rr_{\ell})S(\chi_{i})S(\Rr^{m})\\
    &\overset{\eqref{eq:SotimesId(R)}}{=}-\chi^{i}\Rr_{\ell}S(\Rr_{m})\otimes S(\Rr^{\ell})S(\chi_{i})S(\Rr^{m})\\
    &=-\chi^{i}\Rr_{\ell}S(\Rr_{m})\otimes S(\Rr^{m}\chi_{i}\Rr^{\ell})
\end{split}
\]
%quarta uguaglianza per coassociatività coazione, quinta per la (51) del Lemma, settima per (S\otimes\mathrm{Id})(R)=R^{\mathrm{op}}%
and 
\[
\begin{split}
    \chi&\overset{\eqref{eq:chiS2}}{=}-S(\chi^{i})\Rr_{j}\Rr^{k}\otimes \Rr^{j}\chi_{i}\Rr_{k}
    \overset{\eqref{eq:SotimesS(R)}}{=}-S(\chi^{i})\Rr_{j}S(\Rr^{k})\otimes \Rr^{j}\chi_{i}S(\Rr_{k})\\
    &\overset{\eqref{eq:SotimesId(R)}}{=}-S(\chi^{i})S(\Rr^{j})S(\Rr^{k})\otimes \Rr_{j}\chi_{i}S(\Rr_{k})
    =-S(\Rr^{k}\Rr^{j}\chi^{i})\otimes \Rr_{j}\chi_{i}S(\Rr_{k})\\
    &\overset{\eqref{eq:ctr2}}{=}-S(\Rr^{k}\chi_{i}\Rr^{j})\otimes\chi^{i}\Rr_{j}S(\Rr_{k}),
\end{split}
\]
%prima per 51 Lemma, seconda per (S\otimes S)(R)=R, terza (S\otimes\mathrm{Id})(R)=R^{op},quarta R\chi=\chi^{op}R%
hence $\tau(S\otimes S)(\chi)=\chi$.
\end{proof}

In order to show that for a Cartier triangular Hopf algebra $(H,\Rr,\chi)$ on a field $\Bbbk$ with $\mathrm{char}\left( \Bbbk \right) \neq 2$, the infinitesimal $\Rr$-matrix $\chi$ is a 2-coboundary, we first prove the following result.

\begin{proposition}\label{pro:trivHoch-cohomol-coalg}
Let $H$ be a Hopf algebra and $\chi=\chi^i\otimes\chi_i \in \mathrm{Z}^2(H)$. Assume that $(\mathrm{Id}\otimes \varepsilon)(\chi)=0$.
	Define $\gamma \in H$ as $\gamma:= m(S\otimes\mathrm{Id})(\chi)=S(\chi^i)\chi_i$. Then, we get $b^{1}\left( \gamma \right) =\chi +\tau(S\otimes S)(\chi)$. 
	Moreover, if we	assume $\tau(S\otimes S)(\chi) =\chi $, then $b^{1}\left(\gamma \right) =2\chi $. In particular, if $\mathrm{char}\left( \Bbbk \right) \neq 2
	$ we have $\chi =b^{1}\left(\frac{\gamma}{2} \right) \in \mathrm{B}^2(H)$.
\end{proposition}
\proof
Let $\chi\in \mathrm{Z}^2(H)=\mathrm{Ker}(b^2)$, that is, 
%\begin{equation}\label{eq:b2(chi)}
$1\otimes \chi-(\Delta\otimes\mathrm{Id})(\chi)+(\mathrm{Id}\otimes\Delta)(\chi)-\chi\otimes 1=0,$
%\end{equation} 
i.e., 
\begin{equation}\label{eq:b2(chi)'}
\chi^{i}\otimes\chi_{i1}\otimes\chi_{i2}=\chi^{i}\otimes\chi_{i}\otimes1_{H}-1_{H}\otimes\chi^{i}\otimes\chi_{i}+\chi^{i}_{1}\otimes\chi^{i}_{2}\otimes\chi_{i}.
\end{equation} 
We have \begin{equation*}
\begin{split}
\Delta(\gamma)=&\Delta(S(\chi^{i})\chi_{i})=\Delta(S(\chi^{i}))\Delta(\chi_{i})\\\overset{\eqref{eq:b2(chi)'}}{=}&\Delta(S(\chi^{i}))(\chi_{i}\otimes1_{H})-\Delta(S(1_{H}))\chi+\Delta(S(\chi^{i}_{1}))(\chi^{i}_{2}\otimes\chi_{i})=\\=&(S(\chi^{i}_{2})\otimes S(\chi^{i}_{1}))(\chi_{i}\otimes1_{H})-\chi+(S(\chi^{i}_{1_{2}})\otimes S(\chi^{i}_{1_{1}}))(\chi^{i}_{2}\otimes\chi_{i})\\=&S(\chi^{i}_{2})\chi_{i}\otimes S(\chi^{i}_{1})-\chi+S(\chi^{i}_{1_{2}})\chi^{i}_{2}\otimes S(\chi^{i}_{1_{1}})\chi_{i}\\
\overset{\eqref{eq:b2(chi)'}}{=}&S(\chi_{i1})\chi_{i2}\otimes S(\chi^{i})-S(\chi_{i})1_{H}\otimes S(\chi^{i})+S(\chi^{i})\chi_{i}\otimes S(1_{H})-\chi+S(\chi^{i}_{2_{1}})\chi^{i}_{2_{2}}\otimes S(\chi^{i}_{1})\chi_{i}\\=&\varepsilon(\chi_{i})1_{H}\otimes S(\chi^{i})-\tau(S\otimes S)(\chi)+\gamma\otimes1_{H}-\chi+\varepsilon(\chi^{i}_{2})1_{H}\otimes S(\chi^{i}_{1})\chi_{i}\\=&1_{H}\otimes S(\chi^{i}\varepsilon(\chi_{i}))-\tau(S\otimes S)(\chi)+\gamma\otimes1_{H}-\chi+1_{H}\otimes\gamma\\
=&-\chi+1_{H}\otimes\gamma +\gamma\otimes1_{H}-\tau(S\otimes S)(\chi),
\end{split}	
\end{equation*}
where the last equality follows from the assumption $(\mathrm{Id}\otimes \varepsilon)(\chi)=0$. Thus, we have
\[
b^{1}\left( \gamma \right)=1_{H}\otimes \gamma -\Delta(\gamma)+\gamma\otimes1_{H}=\chi+\tau(S\otimes S)(\chi).
\]
If we assume $%
\tau(S\otimes S)(\chi)=\chi $, we get $b^{1}\left( \gamma \right) =2\chi $. In particular, if $\mathrm{char}\left( \Bbbk \right) \neq 2
	$ we have $\chi =b^{1}\left(\frac{\gamma}{2} \right) \in \mathrm{B}^2(H)$.
\endproof

Let $\left( H,\mathcal{R},\chi \right) $ be a pre-Cartier quasitriangular Hopf
algebra. By Theorem \ref{thm:cQYB}, we know that $\chi=\chi^i\otimes\chi_i \in \mathrm{Z}^2(H)$ and by \eqref{eq:eps-chi1} we have that $(\mathrm{Id}\otimes \varepsilon)(\chi)=0$. As a consequence, Proposition \ref{pro:trivHoch-cohomol-coalg} applies and we can consider the element $\gamma=S(\chi^i)\chi_i$ that we call  the \textbf{Casimir element} of the pre-Cartier quasitriangular Hopf algebra $\left( H,\mathcal{R},\chi \right) $. Note that $\gamma$ is central as, for every $h\in H$,  
\[\gamma h=S(\chi^i)\chi_ih=h_1S(\chi^ih_2)\chi_ih_3\overset{\eqref{cqtr1}}{=}h_1S(h_2\chi^i)h_3\chi_i=hS(\chi^i)\chi_i=h\gamma.\]

\begin{remark}
  Given an $\Rr$-matrix $\Rr=\Rr^i\otimes \Rr_i$ the element $u:=S(\Rr_i)\Rr^i$ is sometimes called the Drinfel'd element, see e.g. \cite[Definition 12.2.9]{Radford}. This definition resembles the one for the Casimir element $\gamma$. 
\end{remark}

We now prove the announced result for a Cartier triangular Hopf algebra.

\begin{theorem}\label{thm:Hoch-cohom-coalg}
Let $\gamma$ be the Casimir element of a Cartier triangular Hopf
algebra $\left( H,\mathcal{R},\chi \right) $. Then $b^{1}\left( \gamma
\right) =2\chi$. In particular,
if $\mathrm{char}\left( \Bbbk \right) \neq 2$ we have $\chi =b^{1}\left(
\frac{\gamma }{2}\right)\in \mathrm{B}^2( H) $, that is, $\chi$ is a 2-coboundary.
\end{theorem}

\proof
\begin{comment}
In view of \eqref{eq:CC3'2} and \eqref{eq:CC4'2}, we have $$b^2(\chi)=1\otimes \chi-(\Delta\otimes\mathrm{id}_H)(\chi)+(\mathrm{id}_H\otimes\Delta)(\chi)-\chi\otimes 1=-(\mathrm{id}\otimes\rho^l)(\chi)+(\rho^r\otimes\mathrm{id})(\chi)\overset{\eqref{eq:chibalanced2}}=0$$ so that $\chi\in \mathrm{Z}^2(H)$. 
\end{comment}
By Lemma \ref{lem:chiS1} we have that $\tau(S\otimes S)(\chi)=\chi$. Since \eqref{eq:eps-chi1} holds true, we conclude by Proposition \ref{pro:trivHoch-cohomol-coalg}.
\begin{comment}
Moreover, by applying $\mathrm{Id}\otimes\epsilon\otimes\mathrm{Id}$ to \eqref{eq:CC3'2} we get $\chi^i\otimes 1_\Bbbk\otimes \chi_i=\chi^i\otimes\varepsilon(\chi_i)\otimes 1_H+\chi^i\otimes 1_\Bbbk\otimes\chi_i$ and then $(\mathrm{Id}\otimes\epsilon)(\chi)=0$.  
\end{comment}
\endproof

Note that for a Cartier triangular Hopf algebra with Casimir element $\gamma$, one has $S(\gamma)=\chi^iS(\chi_i)$. The latter element is in the center even in the pre-Cartier quasitriangular case by the same argument used for $\gamma$.

\begin{example} Consider a cocommutative Hopf algebra $H$. Assume $(H,\Rr)$, with $\Rr=1\otimes 1$, is Cartier with infinitesimal $\Rr$-matrix $\chi$. We have observed in Example \ref{exa:chicntr}\,1) that $\chi$, fulfilling \eqref{cqtr2} and \eqref{cqtr3}, must belong to $ P(H)\otimes P(H)$. By Theorem \ref{thm:Hoch-cohom-coalg}, we have  $b^{1}\left( \gamma
\right) =2\chi$ where $\gamma =S(\chi^i)\chi_i=-\chi^i\chi_i$. Apart the sign, which here differs for our choice of the definition of $b^1$, this is essentially \cite[page 496]{Kassel}. A particular instance of this construction yields the Casimir element of the universal enveloping algebra $U\mathfrak{g}$ of a semisimple Lie algebra $\mathfrak{g}$.
\end{example}

\begin{example}
 For Sweedler's Hopf algebra $H$ with the infinitesimal $\Rr$-matrix $\chi_\alpha=\alpha xg\otimes x$ of Proposition \ref{prop:Sweedler}, the Casimir element is $\alpha S(xg) x=0$. Thus, in this case, $b^{1}\left( \gamma
\right) \neq 2\chi$ unless $\alpha=0$. This confirms that $(H,\Rr_\lambda,\chi)$ is not Cartier unless $\chi$ is trivial, as we already observed in Remark \ref{rmk:Sw}.    
\end{example}

\begin{example}\label{exa:psc} Consider a quasitriangular Hopf algebra $(H,\Rr)$  and let $\chi \in  P(H)\otimes P(H)$ be a symmetric central element. In view of Example \ref{exa:chisym}, we know that  $(H,\Rr,\chi)$ is Cartier and we can consider the Casimir element $\gamma=S(\chi^i)\chi_i=-\chi^i\chi_i$. Since $\chi \in  P(H)\otimes P(H)$ and $\chi$ is symmetric, we get 
$\tau(S\otimes S)(\chi)=\tau(\chi)=\chi$ so that $b^{1}\left( \gamma \right) =\chi +\tau(S\otimes S)(\chi)=2\chi$ and hence $\chi$ is a 2-coboundary if $\mathrm{char}\left( \Bbbk \right) \neq 2$. Thus, $\chi$ can be a 2-coboundary even though $(H,\Rr)$ is not triangular.
\end{example}

\section{The dual picture}\label{sec:dual}

\subsection{Pre-Cartier coquasitriangular bialgebras}

Here we want to investigate pre-Cartier categories modelled on the braided monoidal category $(\Mm^{H},\otimes,\sigma^\Rr)$ of right $H$-comodules, where $\sigma^\Rr$ is the braiding attached to a coquasitriangular bialgebra $(H,\Rr)$. We refer to \cite{Kassel} and \cite{Majid-book} for basic notions about coquasitriangular bialgebras. \\

Recall that, given a coalgebra $C$ and an algebra $A$, then $\mathrm{Hom}_{\Bbbk}(C,A)$ is an algebra such that, for every $f,g\in\mathrm{Hom}_{\Bbbk}(C,A)$, 
$(f*g)(c)=f(c_{1})g(c_{2})$ (\textit{convolution product}) and $1(c)=\varepsilon(c)1_{A}$. Fix a bialgebra $(H,m,u,\Delta,\varepsilon)$ in the following. Given a linear map $f:H\otimes H\to\Bbbk$ we will employ leg notation $f_{12}=f\otimes\varepsilon$, $f_{23}=\varepsilon\otimes f$ and $f_{13}=(\varepsilon\otimes f)(\tau_{H,H}\otimes\mathrm{Id}_{H})$. We also denote $f\circ\tau_{H,H}$ by $f^{\mathrm{op}}$. \\

A bialgebra $H$ is called \textbf{coquasitriangular} when the category $\Mm^H$ of right $H$-comodules is braided. This condition is equivalent to the existence of a convolution invertible linear map $\Rr:H\otimes H\to \Bbbk$, called \textit{universal $\Rr$-form}, such that
\begin{align}
 \Rr*m_H=m_H^{\mathrm{op}}*\Rr,\text{ i.e., } \Rr(a_1\otimes b_1)a_2 b_2&=b_1 a_1\Rr(a_2\otimes b_2),\text{ for every } a,b\in H,\label{eq:ct1}\\
 \Rr(\mathrm{Id}_{H}\otimes m_{H})=\Rr_{13}*\Rr_{12},\text{ i.e., } \Rr(a\otimes bc)&=\Rr(a_1\otimes c)\Rr(a_2\otimes b),\text{ for every } a,b,c\in H,\label{eq:ct2}\\
 \Rr(m_{H}\otimes\mathrm{Id}_{H})=\Rr_{13}*\Rr_{23}, \text{ i.e., } \Rr(ab\otimes c)&=\Rr(a\otimes c_1)\Rr(b\otimes c_2),\text{ for every } a,b,c\in H\label{eq:ct3}.
\end{align}
The corresponding braiding on $\Mm^H$ is defined for all $X,Y$ in $\Mm^H$ by setting $$\sigma^{\Rr}_ {X,Y}:X\otimes Y\to Y\otimes X,\qquad
x\otimes y\mapsto y_0\otimes x_0\Rr(x_{1}\otimes y_{1})$$
as discussed in \cite[Exercise 9.2.9]{Majid-book}. The conditions \eqref{eq:ct2} and \eqref{eq:ct3} say that $\Rr$ is a bialgebra bicharacter (see \cite{Majid-book}, after Definition 2.2.1).  Furthermore, $H$ is called \textbf{cotriangular} when the category $\Mm^H$ is symmetric, i.e., if $\Rr$ further satisfies
\begin{align}\Rr^{-1}=\Rr^{\mathrm{op}}, \text{ i.e., }\Rr(a_1\otimes b_1)\Rr( b_2\otimes a_2)&=\varepsilon(a)\varepsilon(b)\ \text{ for every }a,b\in H.\label{eq:ct4}\end{align}
As a consequence, by \cite[Lemma 2.2.3]{Majid-book}, a coquasitriangular bialgebra $(H,\Rr)$ satisfies the \textit{quantum Yang-Baxter} equations
\begin{equation}\label{eq:cQYB}
\Rr_{12}*\Rr_{13}*\Rr_{23}=\Rr_{23}*\Rr_{13}*\Rr_{12},\qquad\Rr^{-1}_{12}*\Rr^{-1}_{13}*\Rr^{-1}_{23}=\Rr^{-1}_{23}*\Rr^{-1}_{13}*\Rr^{-1}_{12}.
\end{equation}

Recall that, if $(H,\Rr)$ is a coquasitriangular bialgebra then $\Rr(a\otimes1_{H})=\varepsilon(a)=\Rr(1_{H}\otimes a)$ for all $a\in H$ and, if $H$ is a Hopf algebra, then $\Rr(S(a)\otimes b)=\Rr^{-1}(a\otimes b)$, $\Rr^{-1}(a\otimes S(b))=\Rr(a\otimes b)$ and $\Rr(S(a)\otimes S(b))=\Rr(a\otimes b)$ for all $a,b\in H$ in addition (see \cite[Lemma 2.2.2]{Majid-book}).
%Recall that a \textbf{Cartier category} is a pre-additive symmetric monoidal category $(\Mm,c)$ together with a natural transformation $t_{X,Y}:X\otimes Y\to X\otimes Y$ that satisfies the following conditions
%\begin{align}
%c_{X,Y}t_{X,Y}&=t_{Y,X}c_{X,Y},\text{ for every } X,Y\in \Mm,\label{eq:Cartier1}\\
%t_{X,Y\otimes Z}&=t_{X,Y}\otimes Z+(c_{Y,X}\otimes Z)(Y\otimes t_{X,Z})(c_{X,Y}\otimes Z),\text{ for every } X,Y,Z\in \Mm.\label{eq:Cartier2}
%\end{align}
Dual to Definition \ref{defn:curved(quasi)triang}, we have the following:
	\begin{definition}\label{defn:curvedco(quasi)triang}
		We call a triple $(H,\Rr,\chi)$ a \textbf{pre-Cartier} co(quasi)triangular bialgebra if $(H,\Rr)$ is a co(quasi)triangular bialgebra and $\chi:H\otimes H\to \Bbbk$ is such that
\begin{align}
 (u_{H}\chi)*m_{H}&=m_{H}*(u_{H}\chi),\label{eq:CC1}\\
 \chi(\mathrm{Id}_{H}\otimes m_{H})&=\chi_{12}+\Rr^{-1}_{12}*\chi_{13}*\Rr_{12},\label{eq:CC2}\\
\chi(m_{H}\otimes\mathrm{Id}_{H})&=\chi_{23}+\Rr^{-1}_{23}*\chi_{13}*\Rr_{23},\label{eq:CC3}
\end{align}
hold. We call such a $\chi$ an \textbf{infinitesimal $\Rr$-form}. \\
Moreover, the triple $(H,\Rr,\chi)$ is called \textbf{Cartier} if it is pre-Cartier and $\chi$ satisfies
  \begin{equation}\label{eq:CC4}
	\mathcal{R}*\chi=\chi^{\mathrm{op}}*\mathcal{R}, 
\end{equation}
in addition.
\end{definition}

Observe that \eqref{eq:CC1}, \eqref{eq:CC2}, \eqref{eq:CC3} and \eqref{eq:CC4} on elements are
\begin{gather*}
\chi(a_1\otimes b_1)a_2b_2=a_1b_1\chi(a_2\otimes b_2),\\
\chi(a\otimes bc)=\chi(a\otimes b)\varepsilon(c)+\Rr^{-1}(a_1\otimes b_1)\chi(a_2\otimes c)\Rr(a_3\otimes b_2),\\
\chi (ab\otimes c)=\varepsilon (a)\chi (b\otimes c)+\mathcal{R}^{-1}
        (b_{1}\otimes c_{1})\chi (a\otimes c_{2})\mathcal{R}(b_{2}\otimes c_{3}),\\
   \Rr(a_1\otimes b_1)\chi(a_2\otimes b_2)=\chi(b_1\otimes a_1)\Rr(a_2\otimes b_2).     
\end{gather*}
for any $a,b,c\in H$, respectively.

\begin{comment}	
	\begin{theorem}\label{thm:curved-coquasi}
		A bialgebra $H$ is curved coquasitriangular if and only if there is a convolution invertible linear map $\Rr:H\otimes H\to \Bbbk$ such that $(H,\Rr)$ is coquasitriangular and there is a linear map $\chi:H\otimes H\to \Bbbk$ such that
		\begin{align}
		(u_{H}\chi) * m_{H}&= m_{H}*(u_{H}\chi) \label{c-coqtr1}\\
		\chi(\mathrm{id}_{H}\otimes m_{H})&=\chi_{12}+\Rr^{-1}_{12}*\chi_{13}*\Rr_{12}\label{c-coqtr2}\\
\chi(m_{H}\otimes\mathrm{id}_{H})&=\chi_{23}+\Rr^{-1}_{23}*\chi_{13}*\Rr_{23}\label{c-coqtr3}
		\end{align}
		hold. Moreover, a cotriangular bialgebra $(H,\Rr)$ is a co-Cartier bialgebra if and only if the corresponding $\chi:H\otimes H\to \Bbbk$  satisfies
		\begin{equation}\label{eq:c-cotr2}
		\mathcal{R}*\chi=\chi\tau*\mathcal{R}
		\end{equation}
		in addition, where $\tau$ is the canonical flip map.
	\end{theorem}
\end{comment}

%\begin{definition}
%A bialgebra $H$ will be called \textbf{curved cotriangular} when the category $\mm^H$ is a Cartier category.
%\end{definition}

\begin{theorem}\label{thm:curv-cotr}
Let $(H,\Rr)$ be a co(quasi)triangular bialgebra. Then, there is a bijection between pre-Cartier structures of $(H,\Rr)$ and pre-Cartier structures of $(\Mm^{H},\otimes,\sigma^{\Rr})$. The corresponding infinitesimal braiding on $\Mm^H$ is defined for all objects $M,N$ in $\Mm^H$ by $$t_ {M,N}:M\otimes N\to M\otimes N,\qquad m\otimes n\mapsto m_0\otimes n_0\chi(m_{1}\otimes n_{1}),$$ where $\chi:H \otimes H\to\Bbbk$ is the infinitesimal $\Rr$-form for $H$. Moreover, there is a bijective correspondence between Cartier structures of $(H,\Rr)$ and Cartier structures of $(\Mm^{H},\otimes,\sigma^{\Rr})$.
%A co(quasi)triangular bialgebra $(H,\Rr)$ is Cartier if and only if $(\Mm^{H},\otimes,\sigma^{\Rr})$ is (braided) symmetric Cartier.
\end{theorem}

\begin{proof}
Suppose there is an infinitesimal braiding $t$ on $\Mm^H$ and set $\chi:=(\varepsilon\otimes\varepsilon)t_{H,H}$.
For  every $M$ in $\Mm^H$ and $\alpha\in M^*$, define $\widehat{\alpha}:M\to H$ in $\Mm^H$ by setting $\widehat{\alpha}(m)=\alpha(m_0)m_{1}$ and observe that $\varepsilon\circ\widehat{\alpha}=\alpha$. Now, let $M,N$ in $\Mm^H$ and consider $\alpha\in M^*$ and $\beta\in N^*$. Since $t$ is natural, we get
$(\alpha\otimes \beta)t_{M,N}(m\otimes n)
=(\varepsilon\otimes\varepsilon)(\widehat{\alpha}\otimes \widehat{\beta})t_{M,N}(m\otimes n)
=(\varepsilon\otimes\varepsilon)t_{H,H}(\widehat{\alpha}\otimes \widehat{\beta})(m\otimes n)
=\chi(\widehat{\alpha}\otimes \widehat{\beta})(m\otimes n)
=\alpha(m_0)\beta(n_0)\chi(m_{1}\otimes n_{1})
=(\alpha\otimes \beta)(m_0\otimes n_0\chi(m_{1}\otimes n_{1})).
$ Since $\alpha$ and $\beta$ are arbitrary, we arrive at $t_{M,N}(m\otimes n)=m_0\otimes n_0\chi(m_{1}\otimes n_{1})$. Now, the fact that $t_{M,N}$ is a morphism in $\Mm^H$ for all $M,N$ in $\Mm^H$ is equivalent to \eqref{eq:CC1}.
The condition \eqref{qC-I} is equivalent to \eqref{eq:CC2} and the condition \eqref{qC-II} is equivalent to \eqref{eq:CC3}. Finally, \eqref{cart} is equivalent to \eqref{eq:CC4}.
\end{proof}

%We should point out that, if we assume the remaining equalities, (???)
%\eqref{ctr3}
%is equivalent to
%\begin{equation}\chi (ab\otimes c)=\varepsilon (a)\chi (b\otimes c)+\mathcal{R}%
%(c_{1}\otimes b_{1})\chi (a\otimes c_{2})\mathcal{R}(b_{2}\otimes c_{3})%
%\text{ for every }a,b,c\in H.  
%\end{equation}

The following example is the dual version of Example \ref{exa:chicntr}\,2).

\begin{example} Let $(H,\Rr)$ be a co(quasi)triangular bialgebra and let $\chi:H\ot H\to\Bbbk$ be a linear map. Assume $H$ is cocommutative and let us check when $(H,\Rr,\chi)$ is a pre-Cartier bialgebra. 
First note that  \eqref{eq:ct1} is equivalent to commutativity of $H$ while \eqref{eq:CC1} is always true. Furthermore, \eqref{eq:CC2} is equivalent to $\chi(a\otimes bc)=\chi(a\otimes b)\varepsilon(c)+\varepsilon(b)\chi(a\otimes c)$ which means that the map $\chi(a,-):H\to \Bbbk$ is a derivation.
\begin{invisible}  
First note that $b_1 a_1\Rr(a_2\ot b_2)=b_2 a_2\Rr(a_1\ot b_1)=\Rr(a_1\ot b_1)b_2 a_2$  so that \eqref{eq:ct1} is equivalent to $\Rr(a_1\ot b_1) a_2 b_2=\Rr(a_1\ot b_1)b_2 a_2$. Since $\Rr$ is convolution invertible, this leads to $ab=ba$, i.e., to commutativity of $H$. Similarly, we have $a_1b_1\chi(a_2\ot b_2)=a_2b_2\chi(a_1\ot b_1)=\chi(a_1\ot b_1)a_2b_2$ so that \eqref{eq:CC1} is always true.      
Furthermore 
\[
\begin{split}
\chi(a\otimes b)\varepsilon(c)+\Rr^{-1}(a_1\otimes b_1)\chi(a_2\otimes c)\Rr(a_3\otimes b_2)&=\chi(a\otimes b)\varepsilon(c)+\Rr^{-1}(a_1\otimes b_1)\chi(a_3\otimes c)\Rr(a_2\otimes b_2)\\&=\chi(a\otimes b)\varepsilon(c)+\Rr^{-1}(a_1\otimes b_1)\Rr(a_2\otimes b_2)\chi(a_3\otimes c)\\&=\chi(a\otimes b)\varepsilon(c)+\varepsilon(b)\chi(a\otimes c)
\end{split}
 \] 
 \end{invisible}
 Similarly, \eqref{eq:CC3} is equivalent to $\chi(ab\otimes c)=\varepsilon(a)\chi(b\otimes c)+\varepsilon(b)\chi(a\otimes c)$ which means that the map $\chi(-,c):H\to\Bbbk$ is a derivation. Thus $(H,\Rr,\chi)$ is pre-Cartier if and only if $\chi$ is a biderivation, i.e., a derivation in both entries. 
 Consider $Q(H)=H^+/(H^+)^2$, the so-called space of indecomposables of $H$, where $H^+=\mathrm{Ker}(\varepsilon)$ is the augmentation ideal, and the canonical projection $\pi:H\to Q(H),h\mapsto h-\varepsilon(h)1+(H^+)^2$. Since any derivation vanishes on  $\mathrm{Ker}(\pi)=(H^+)^2\oplus\Bbbk1$, we get that $\chi:H\otimes H\to \Bbbk$ is a biderivation if and only if it factors through $\pi\otimes\pi$, i.e., there exists $\chi^+:Q(H)\otimes Q(H)\to\Bbbk$ such that $\chi^+\circ (\pi\otimes\pi)=\chi$.
 
 Moreover, since $\Rr$ is convolution invertible, \eqref{eq:CC4} is equivalent to $\chi(a\ot b)=\chi(b\ot a)$, i.e., to the fact that $\chi$ is symmetric.
 \begin{invisible}
 Since $\chi(b_1\ot a_1)\Rr(a_2\ot b_2)=\chi(b_2\ot a_2)\Rr(a_1\ot b_1)=\Rr(a_1\ot b_1)\chi(b_2\ot a_2)$, we have that \eqref{eq:CC4} is equivalent to $\Rr(a_1\ot b_1)\chi(a_2\ot b_2)=\Rr(a_1\ot b_1)\chi(b_2\ot a_2)$.    
 \end{invisible}
  Hence $(H,\Rr,\chi)$ is Cartier if and only if $\chi$ is a symmetric biderivation.

    Let us specify two instances of this situation.
  \begin{itemize}
    \item Let $H:=\Bbbk G$, where $(G,+)$ is an abelian monoid and let $\Rr=\varepsilon\otimes\varepsilon$. Then, $\chi:H\ot H\to\Bbbk$ gives $(H,\Rr,\chi)$ a structure of Cartier cotriangular bialgebra if and only if $G\times G\to\Bbbk,(a,b)\mapsto \chi(a\ot b)$, is a symmetric additive bicharacter.
    This is consistent with \cite[Remark 2.2 (3)]{HV}, where we note that the authors work with left comodules.
    \item Let $H:=S(V)$ be the symmetric algebra over a vector space $V$, regarded as a bialgebra where the elements in $V$ are primitive. Since $\chi$ is a derivation on each entry, by the foregoing it vanishes on $\mathrm{Ker}(\pi)=(S(V)^+)^2=V^{\otimes 2}\oplus V^{\otimes 3}\oplus\cdots$ on each entry and hence it vanishes everywhere excepted on its restriction to $V\ot V$. Thus, $\chi$ is completely determined by $\chi_{\mid V\ot V}:V\ot V\to \Bbbk$  on which it only has to be symmetric in order to  conclude that $H$ is Cartier.
  \end{itemize}
\end{example}

\begin{remark}
Note that any derivation $\partial:H\to \Bbbk$ can be regarded as a primitive element in the finite dual $H^\circ$ and, by a result due to Michaelis, $P(H^\circ)$ is isomorphic to  $Q(H)^*$. Explicitly, see \cite[\S2]{Michaelis90}, we can write this isomorphism as $P(H^\circ)\to Q(H)^*,\partial\mapsto [x+(H^+)^2\mapsto \partial(x)].$   
\end{remark} 

Dually to Example \ref{exa:chisym} we have the following one.
\begin{example}\label{exa:chisymco} Let $(H,\Rr)$ be a co(quasi)triangular bialgebra and let $\chi\in \mathscr{Z}((H\ot H)^*)$. 
Then \eqref{eq:CC1} is trivially satisfied, \eqref{eq:CC2} and \eqref{eq:CC3} together mean that $\chi $ is a biderivation while \eqref{eq:CC4} means that $\chi$ is symmetric. As a consequence, $(H,\Rr,\chi)$ is pre-Cartier if $\chi$ is a central biderivation and it is Cartier when it is further symmetric.
\end{example} 

\textbf{The bialgebra $\mathrm{M}_q(2)$ and the quantum groups $\mathrm{GL}_q(2)$ and $\mathrm{SL}_q(2)$.}
Let $q\in\mathbb{C}$ be a non-zero complex number. Following the convention of \cite[Section IV]{Kassel} we recall that $\mathrm{M}_q(2)$ is the (unital) algebra freely generated by $\alpha,\beta,\gamma,\delta$, modulo the Manin relations
\begin{center}
    $\beta\alpha=q\alpha\beta,\qquad
    \gamma\alpha=q\alpha\gamma,\qquad
    \delta\beta=q\beta\delta,\qquad
    \delta\gamma=q\gamma\delta,$\\
    $\gamma\beta=\beta\gamma,\qquad
    \alpha\delta-\delta\alpha=(q^{-1}-q)\beta\gamma$.
\end{center}
It is a bialgebra, which is neither commutative nor cocommutative, with coproduct and counit determined on generators by
\begin{equation}\label{Mq2bialg}
\Delta\begin{pmatrix}
    \alpha & \beta\\
    \gamma & \delta
\end{pmatrix}=\begin{pmatrix}
    \alpha & \beta\\
    \gamma & \delta
\end{pmatrix}\otimes\begin{pmatrix}
    \alpha & \beta\\
    \gamma & \delta
\end{pmatrix},\qquad
\varepsilon\begin{pmatrix}
    \alpha & \beta\\
    \gamma & \delta
\end{pmatrix}=\begin{pmatrix}
    1 & 0\\
    0 & 1
\end{pmatrix},
\end{equation}
where the former is to be understood as matrix comultiplication, e.g., $\Delta(\alpha)=\alpha\otimes\alpha+\beta\otimes\gamma$. There is a coquasitriangular structure $\Rr$ on $\mathrm{M}_q(2)$ determined on generators by
\begin{equation}\label{Mq2R}
\Rr\begin{pmatrix}
    \alpha\otimes\alpha & \beta\otimes\beta &\alpha\otimes\beta & \beta\otimes\alpha\\
    \gamma\otimes\gamma & \delta\otimes\delta & \gamma\otimes\delta & \delta\otimes\gamma\\
    \alpha\otimes\gamma & \beta\otimes\delta & \alpha\otimes\delta & \beta\otimes\gamma\\
    \gamma\otimes\alpha & \delta\otimes\beta & \gamma\otimes\beta & \delta\otimes\alpha
\end{pmatrix}=q^{-\frac{1}{2}}\begin{pmatrix}
    q & 0 & 0 & 0\\
    0 & q & 0 & 0\\
    0 & 0 & 1 & q-q^{-1}\\
    0 & 0 & 0 & 1
\end{pmatrix},    
\end{equation}
c.f. \cite[Corollary VIII.7.2]{Kassel}. It is straightforward to check that
\begin{eqnarray*}
\mathcal{R}^{-1}%
\begin{pmatrix}
\alpha \otimes \alpha  & \beta \otimes \beta  & \alpha \otimes \beta  &
\beta \otimes \alpha  \\
\gamma \otimes \gamma  & \delta \otimes \delta  & \gamma \otimes \delta  &
\delta \otimes \gamma  \\
\alpha \otimes \gamma  & \beta \otimes \delta  & \alpha \otimes \delta  &
\beta \otimes \gamma  \\
\gamma \otimes \alpha  & \delta \otimes \beta  & \gamma \otimes \beta  &
\delta \otimes \alpha
\end{pmatrix}
&=&q^{\frac{1}{2}}%
\begin{pmatrix}
q^{-1} & 0 & 0 & 0 \\
0 & q^{-1} & 0 & 0 \\
0 & 0 & 1 & q^{-1}-q \\
0 & 0 & 0 & 1%
\end{pmatrix}.%
\end{eqnarray*}

\begin{remark}\label{rmk:hateps}
Note that the free algebra $\Bbbk\{\alpha,\beta,\gamma,\delta\}$  is a $\mathbb{N}$-graded algebra. The degree of an arbitrary homogeneous element $\xi$ is its length $|\xi|$, so for instance $|\alpha^{n}\beta^{m}\gamma^{k}\delta^{\ell}|=n+m+k+\ell$, where $n,m,k,\ell\in\mathbb{N}$.
Since $\mathrm{M}_q(2)$ is obtained by factoring out $\Bbbk\{\alpha,\beta,\gamma,\delta\}$ by an ideal generated by homogeneous elements, then $\mathrm{M}_q(2)$ inherits the above graduation and hence becomes a $\mathbb{N}$-graded algebra, as well. In particular, $\mathrm{M}_q(2)$ becomes a left $\Bbbk[X]$-comodule algebra, where the coalgebra structure of the polynomial ring $\Bbbk[X]$ is the one with $X$ grouplike, with coaction $\rho^l:\mathrm{M}_q(2)\to \Bbbk[X]\otimes \mathrm{M}_q(2)$ defined by $\rho^l(\xi):=\xi_{-1}\otimes\xi_0:=X^{|\xi|}\otimes\xi$, which is an algebra map. On the other hand, consider the derivation $\partial:\Bbbk[X]\to \Bbbk$ defined by $\partial(X^n)=n1_\Bbbk$.
\begin{invisible}
Given $a,b\in\mathbb{N}$, we have $\partial(X^a\cdot X^b)=\partial(X^{a+b})=(a+b)1=a1+b1=\partial(X^a)\varepsilon(X^b)+\varepsilon(X^a)\partial(X^b).$
\end{invisible}
Thus the algebra map $\varepsilon:\mathrm{M}_q(2)\to \Bbbk$ induces the derivation $\partial_q:\mathrm{M}_q(2)\to \Bbbk$ defined by $\partial_q(\xi)= \partial(\xi_{-1})\varepsilon(\xi_{0})$, for every $\xi\in \mathrm{M}_q(2)$,
in the sense that $\partial_q(\xi\eta)=\partial_q(\xi)\varepsilon(\eta)+\varepsilon(\xi)\partial_q(\eta)$ for all $\xi,\eta\in \mathrm{M}_q(2)$.
If $\xi$ is homogeneous, we get 
$\partial_q(\xi)=\partial(X^{|\xi|})\varepsilon(\xi)=|\xi|\varepsilon(\xi)$. 
Note that $\partial_q(\xi_1)\xi_2=|\xi|\varepsilon(\xi_1)\xi_2=|\xi|\xi=\xi_1\varepsilon(\xi_2)|\xi|=\xi_1\partial_q(\xi_2)$.
This is equivalent to the condition $\partial_q\in\mathscr{Z}(H^*)$.
\end{remark}

\begin{proposition}\label{prop:Mq2chi}
There is a non-trivial infinitesimal $\Rr$-form on $\mathrm{M}_q(2)$ given by $\chi=\partial_q\otimes\partial_q$
\begin{invisible}
\begin{equation}
    \chi\begin{pmatrix}
    \alpha\otimes\alpha & \beta\otimes\beta &\alpha\otimes\beta & \beta\otimes\alpha\\
    \gamma\otimes\gamma & \delta\otimes\delta & \gamma\otimes\delta & \delta\otimes\gamma\\
    \alpha\otimes\gamma & \beta\otimes\delta & \alpha\otimes\delta & \beta\otimes\gamma\\
    \gamma\otimes\alpha & \delta\otimes\beta & \gamma\otimes\beta & \delta\otimes\alpha
\end{pmatrix}=\begin{pmatrix}
    1 & 0 & 0 & 0\\
    0 & 1 & 0 & 0\\
    0 & 0 & 1 & 0\\
    0 & 0 & 0 & 1
\end{pmatrix}
\end{equation}
\end{invisible}
such that $(\mathrm{M}_q(2),\Rr,\chi)$ is Cartier.
On homogeneous elements $\xi$ and $\eta$  it reads $ \chi(\xi\otimes\eta)=|\xi||\eta|\varepsilon(\xi\eta).$
\begin{invisible}
\begin{equation}\label{chiSL2}
    \chi(\xi\otimes\eta)=|\xi|\cdot|\eta|\varepsilon(\xi\eta).
\end{equation}
\end{invisible}
\end{proposition}

\begin{proof}
According to Remark \ref{rmk:hateps}, $\chi=\partial_q\otimes \partial_q:\mathrm{M}_q(2)\otimes \mathrm{M}_q(2)\to \Bbbk$ is a well-defined symmetric biderivation. Moreover, since we observed that $\partial_q\in\mathscr{Z}(H^*)$ we have that $\chi\in\mathscr{Z}((H\otimes H)^*)$, i.e., it is central. By Example \ref{exa:chisymco}, we conclude that $(\mathrm{M}_q(2),\Rr,\chi)$ is Cartier.
\begin{invisible}
    For the moment, we just verify that \eqref{chiSL2} is well-defined. The only Manin relation which is not obviously compatible with the definition is $\alpha\delta-\delta\alpha=(q-q^{-1})\beta\gamma$. Let $\xi,\eta\in \mathrm{M}_q(2)$ be arbitrary homogeneous elements of length $|\xi|,|\eta|\in\mathbb{N}$. Note that the length is well-defined since the Manin relations are homogeneous. Then
\begin{align*}
    \chi(\xi\otimes\alpha\delta\eta)
    &=|\xi|(2+|\eta|)\varepsilon(\xi\alpha\delta\eta)\\
    &=|\xi|(2+|\eta|)\varepsilon(\xi\eta)\\
    &=|\xi|(2+|\eta|)\varepsilon(\xi\delta\alpha\eta)+(q-q^{-1})|\xi|(2+|\eta|)\underbrace{\varepsilon(\xi\beta\gamma)}_{=0}\\
    &=\chi(\xi\otimes(\delta\alpha+(q-q^{-1})\beta\gamma)\eta)
\end{align*}
(and similarly with the arguments of $\chi$ exchanged) shows that \eqref{chiSL2} is further compatible with the last Manin relation.
\end{invisible}
\end{proof}

As we are going to see in Proposition~\ref{Prop:xFRT} the above result also follows as  an instance of the infinitesimal FRT construction. \medskip

The algebra $\mathrm{GL}_q(2)$ is defined as $\mathrm{M}_q(2)[\theta]/\left( \theta {\det}_{q}-1\right)$ where we used the quantum determinant
$$
{\det}_{q}:=\alpha\delta-q^{-1}\beta\gamma
$$ and it turns out that the bialgebra structure \eqref{Mq2bialg} and coquasitriangular structure \eqref{Mq2R} pass to $\mathrm{GL}_q(2)$, see e.g. \cite[Theorem IV.6.1 and Corollary VIII.7.4]{Kassel} in such a way that $\theta$ is grouplike.

Now, the derivation $\partial_q:\mathrm{M}_q(2)\to \Bbbk$ of Remark \ref{rmk:hateps} induces a unique derivation 
$\hat\partial_q:\mathrm{M}_q(2)[\theta] \rightarrow \Bbbk $ that
restricted to $\mathrm{M}_q(2)$ is $\partial_q$ and maps $\theta$ to $-2.$ Note that this forces  $\hat\partial_q(\theta^{n})=-2n$.

\begin{proposition}\label{prop:GLq} The derivation $\hat\partial_q:\mathrm{M}_q(2)[\theta] \rightarrow \Bbbk $ descends to a derivation $\hat\partial_q:\mathrm{GL}_q(2) \rightarrow \Bbbk $. Moreover, $\chi :=\hat\partial_q\otimes \hat\partial_q:\mathrm{GL}%
_{q}(2)\otimes \mathrm{GL}_{q}(2)\rightarrow \Bbbk$ is a non-trivial infinitesimal $\Rr$-form on $\mathrm{GL}_{q}(2)$ 
such that $(\mathrm{GL}_{q}(2),\Rr,\chi)$ is Cartier.
\end{proposition}

\begin{proof}
We compute
$
 \hat\partial _q\left( \theta {\det}_q\right)
 = \hat{\partial }_{q}\left( \theta \right) \varepsilon
\left( {\det}_q \right) +\varepsilon \left(
\theta \right)  \hat{\partial }_{q}\left(  {\det}_q
\right)  
=-2+\partial _{q}\left( \alpha \delta -q^{-1}\beta \gamma \right) =0.
$
Thus $\theta \det_{q}-1\in \mathrm{Ker}\left( \hat{\partial }_{q}\right) \cap
\mathrm{Ker}\left( \varepsilon\right) $ and hence the derivation $\partial_q:\mathrm{M}_q(2)\to \Bbbk$ induces the derivation $\hat\partial_q:\mathrm{GL}%
_{q}(2)\rightarrow \Bbbk $. Consider now
$\chi :=\hat\partial_q\otimes \hat\partial_q:\mathrm{GL}%
_{q}(2)\otimes \mathrm{GL}_{q}(2)\rightarrow \Bbbk $. 
This is a symmetric biderivation.  
Now, for every $\xi \in {M}_{q}(2)$ we have
$\hat\partial_q(\xi \theta ^{n})
=\hat\partial_q(\xi )\varepsilon(\theta ^{n})+\varepsilon(\xi)\hat\partial_q(\theta ^{n} )
=\partial_q(\xi )-2n\varepsilon(\xi)$ so that
we get
$ \hat{\partial }_q\left( \xi _{1}\theta ^{n}\right) \xi _{2}\theta ^{n} 
=\partial _{q}(\xi_1)\xi _{2}\theta ^{n}-2n\varepsilon(\xi_1)\xi _{2}\theta ^{n}
=\xi _{1}\theta ^{n}\partial _{q}(\xi_2)-\xi _{1}\theta ^{n}2n\varepsilon(\xi_2)
=\xi _{1}\theta ^{n} \hat{\partial }_q\left( \xi _{2}\theta ^{n}\right).$
As a consequence, since every monomial $\zeta \in \mathrm{GL}%
_{q}(2)$ is of the form $\zeta =\xi \theta ^{n}$ where $\xi \in {M}_{q}(2)$, we deduce that 
$ \hat{\partial }_q\left( \zeta _{1}\right) \zeta _{2} =\zeta _{1} \hat{\partial }_q\left( \zeta _{2}\right) $
for every $\zeta \in \mathrm{GL}_{q}(2)$. Thus $ \hat{\partial }_q$ is central and hence $\chi$ is central. By Example \ref{exa:chisymco}, we conclude that $(\mathrm{GL}_{q}(2),\Rr,\chi)$ is Cartier.
\end{proof}

The algebra $\mathrm{SL}_q(2)$ is defined as $\mathrm{M}_q(2)$ where we further quotient the quantum determinant relation
$$
\alpha\delta-q^{-1}\beta\gamma=1
$$
and it turns out that the bialgebra structure \eqref{Mq2bialg} and coquasitriangular structure \eqref{Mq2R} descend to the quotient, see e.g. \cite[Corollary VIII.7.4]{Kassel}.
However, the infinitesimal $\Rr$-form $\chi$ of Proposition~\ref{prop:Mq2chi} does not descend, since for example
\begin{align*}
    \chi(\alpha\otimes(\alpha\delta-q^{-1}\beta\gamma))
    =2\varepsilon(\alpha^2\delta)
    -2q^{-1}\varepsilon(\alpha\beta\gamma)
    =2
    \neq 0=\chi(\alpha\otimes 1).
\end{align*}
Moreover, we can show that $\mathrm{SL}_q(2)$ does not admit any non-trivial infinitesimal $\Rr$-form (excluding certain values of $q$).

\begin{proposition}
If $q\neq \pm1$, there is no non-trivial infinitesimal $\Rr$-form on $\mathrm{SL}_q(2)$. 
\end{proposition} 
\begin{proof}
Let $\chi$ be an infinitesimal $\Rr$-form on $\mathrm{SL}_q(2)$. It is straightforward to check that \eqref{eq:CC1}, i.e., $\chi \left( a_{1}\otimes b_{1}\right)
a_{2}b_{2}=a_{1}b_{1}\chi \left( a_{2}\otimes b_{2}\right) $ with $a,b\in \mathrm{SL}_q(2)$, holds for $a,b$ generators, if and only if 
\begin{gather*}
\chi
\begin{pmatrix}
\alpha \otimes \alpha & \beta \otimes \beta & \alpha \otimes \beta & \beta
\otimes \alpha \\
\gamma \otimes \gamma & \delta \otimes \delta & \gamma \otimes \delta &
\delta \otimes \gamma \\
\alpha \otimes \gamma & \beta \otimes \delta & \alpha \otimes \delta & \beta
\otimes \gamma \\
\gamma \otimes \alpha & \delta \otimes \beta & \gamma \otimes \beta & \delta
\otimes \alpha%
\end{pmatrix}%
=%
\begin{pmatrix}
k & 0 & 0 & 0 \\
0 & k & 0 & 0 \\
0 & 0 & h & q^{-1}k-q^{-1}h \\
0 & 0 & q^{-1}k-q^{-1}h &\left( 1-q^{-2}\right) k+q^{-2}h%
\end{pmatrix}
\end{gather*} for some $k,h\in \Bbbk$. 
For instance let us compute $\chi \left( a_{1}\otimes b_{1}\right)
a_{2}b_{2}=a_{1}b_{1}\chi \left( a_{2}\otimes b_{2}\right) $ for $a=b=\alpha.$
The left-hand side is
\begin{align*}
\chi \left( \alpha _{1}\otimes \alpha
_{1}\right) \alpha _{2}\alpha _{2}
&=\chi \left( \alpha \otimes \alpha \right)
\alpha \alpha +\chi \left( \alpha \otimes \beta \right) \alpha \gamma +\chi
\left( \beta \otimes \alpha \right) \gamma \alpha +\chi \left( \beta \otimes
\beta \right) \gamma \gamma \\
&=\chi \left( \alpha \otimes \alpha \right) \alpha \alpha +q^{-1}\chi \left(
\alpha \otimes \beta \right) \gamma \alpha +\chi \left( \beta \otimes \alpha
\right) \gamma \alpha +\chi \left( \beta \otimes \beta \right) \gamma \gamma
\end{align*}
while the right-hand side is
\begin{align*}\alpha _{1}\alpha _{1}\chi \left( \alpha _{2}\otimes \alpha _{2}\right)
&=\alpha \alpha \chi \left( \alpha \otimes \alpha \right) +\alpha \beta \chi
\left( \alpha \otimes \gamma \right) +\beta \alpha \chi \left( \gamma
\otimes \alpha \right) +\beta \beta \chi \left( \gamma \otimes \gamma
\right) \\
&=\alpha \alpha \chi \left( \alpha \otimes \alpha \right) +q^{-1}\beta \alpha
\chi \left( \alpha \otimes \gamma \right) +\beta \alpha \chi \left( \gamma
\otimes \alpha \right) +\beta \beta \chi \left( \gamma \otimes \gamma
\right).\end{align*}
Equating the two sides yields
\[q^{-1}\chi \left( \alpha \otimes \beta \right) \gamma \alpha +\chi
\left( \beta \otimes \alpha \right) \gamma \alpha +\chi \left( \beta \otimes
\beta \right) \gamma \gamma =q^{-1}\beta \alpha \chi \left( \alpha \otimes \gamma
\right) +\beta \alpha \chi \left( \gamma \otimes \alpha \right) +\beta \beta
\chi \left( \gamma \otimes \gamma \right) . \]
Since $\gamma \gamma$ and $\beta \beta$ cannot be written as a linear combination of the remaining generators in degree two, the respective coefficients must be zero, i.e.,  $\chi \left( \beta \otimes \beta \right) =\chi \left( \gamma \otimes \gamma
\right) =0$. Also $\gamma\alpha$ and $\beta\alpha$ cannot be written as a linear combination of the remaining generators so that we get $\chi \left( \beta \otimes \alpha \right) =-q^{-1}\chi \left( \alpha
\otimes \beta \right) $ and $\chi \left( \gamma \otimes \alpha \right) =-q^{-1}\chi
\left( \alpha \otimes \gamma \right)$. Proceeding this way one shows that $\chi$ is as desired.

Now the equality $\alpha\delta-q^{-1}\beta\gamma=1$ and the fact that $\chi(\alpha\otimes 1)=0$ imply
$\chi(\alpha\otimes \alpha\delta)=q^{-1}\chi (\alpha\otimes \beta\gamma) $.
If we now apply \eqref{eq:CC2} to compute explicitly the two sides of this equality, we get $\chi(\alpha\otimes \alpha\delta)=\chi(\alpha\otimes \alpha)+\chi(\alpha\otimes \delta)=k+h$ and $\chi (\alpha\otimes \beta\gamma)=0$ so that we get $h=-k$. Similarly, one has $\chi(\delta\otimes \alpha\delta)=q^{-1}\chi (\delta\otimes \beta\gamma) $ and one checks that  
$\chi(\delta\otimes \alpha\delta)=\chi(\delta\otimes\alpha)+\chi(\delta\otimes\delta)$ and $\chi (\delta\otimes \beta\gamma)=0$ so that we get $0=\chi(\delta\otimes\alpha)+\chi(\delta\otimes\delta)=\left( 1-q^{-2}\right) k+q^{-2}h+k
=\left( 1-q^{-2}\right) k-q^{-2}k+k=2\left( 1-q^{-2}\right) k.$ Therefore, if $q\neq\pm 1$, we obtain $k=0$ and also $h=-k=0$.
This means that $\chi$ is zero on generators. By means of \eqref{eq:CC2} and \eqref{eq:CC3} one easily gets that $\chi=0$.
\end{proof}

\subsection{Twisting pre-Cartier coquasitriangular bialgebras}

Let $(H,m,u,\Delta,\varepsilon)$ be a bialgebra.
Recall from \cite[Section 2.3]{Majid-book} that a convolution invertible map $\Ff:H\otimes H\to\Bbbk$ is a \textbf{$2$-cocycle twist} if the normalization properties and the 2-cocycle condition are satisfied. Explicitly, 
\begin{align}
 &\Ff(1\otimes a)=\varepsilon(a)=\Ff(a\otimes 1), \text{ for every } a\in H,\label{Dt1} \\ & \Ff(a_{1}\otimes b_{1})\Ff(a_{2}b_{2}\otimes c)=\Ff(b_{1}\otimes c_{1})\Ff(a\otimes b_{2}c_{2}), \text{ for every }a,b,c \in H \label{Dt2}.
\end{align}
Note that \eqref{Dt2} is equivalent to
\begin{equation}\label{Dt2'}
(\Ff\otimes\varepsilon)*\Ff(m\otimes\mathrm{Id})=(\varepsilon\otimes \Ff)*\Ff(\mathrm{Id}\otimes m), \text{ i.e., } \Ff_{12}*\Ff(m\otimes\mathrm{Id})=\Ff_{23}*\Ff(\mathrm{Id}\otimes m) .
\end{equation}
Given a 2-cocycle twist $\Ff:H\otimes H\to\Bbbk$ we consider the linear map $m_\Ff:H\otimes H\to H$ defined for any $a,b\in H$ by
$$m_{\Ff}(a\otimes b)=a\cdot_{\Ff}b:=\Ff(a_{1}\otimes b_{1})a_{2}b_{2}\Ff^{-1}(a_{3}\otimes b_{3}).$$
Then, by the dual of \cite[Theorem 2.3.4]{Majid-book}, $H_{\Ff}:=(H,m_{\Ff},u,\Delta,\varepsilon)$ is a bialgebra. If $H$ is a Hopf algebra with antipode $S$, then also $H_\Ff$ is a Hopf algebra with antipode $S_\Ff$ given by $S_{\Ff}(a)=U(a_{1})S(a_{2})U^{-1}(a_{3})$ where $U(a):=\Ff(a_{1}\otimes S(a_{2}))$, for any $a\in H$. Moreover, if $(H,\Rr)$ is co(quasi)triangular, then also $(H_{\Ff},\Rr_{\Ff})$ is co(quasi)triangular with universal $\Rr$-form
$$\Rr_{\Ff}(a\otimes b)=\Ff(b_{1}\otimes a_{1})\Rr(a_{2}\otimes b_{2})\Ff^{-1}(a_{3}\otimes b_{3}).$$
\begin{comment}
\begin{theorem}
Given $(H,m,u,\Delta,\varepsilon,S)$ a Hopf algebra and $F:H\otimes H\to\Bbbk$ a $2$-cocycle twist then $H_{\Ff}=(H,m_{\Ff},u,\Delta,\epsilon,S_{F})$ is a Hopf algebra with $$m_{F}(a\otimes b)=a\cdot_{F}b:=F(a_{1}\otimes b_{1})a_{2}b_{2}F^{-1}(a_{3}\otimes b_{3}) \text{ and }
S_{F}(a)=U(a_{1})S(a_{2})U^{-1}(a_{3})$$ where $U(a):=F(a_{1}\otimes S(a_{2}))$.
\end{theorem}
\begin{remark}\label{rmk:coquasitr}

If $(H,\Rr)$ is co(quasi)triangular then $(H_{F},\Rr_{F})$ is co(quasi)triangular where
\[
\Rr_{F}(a\otimes b)=F(b_{1}\otimes a_{1})\Rr(a_{2}\otimes b_{2})F^{-1}(a_{3}\otimes b_{3}).
\]
\end{remark}
\end{comment}
Thus, it follows that $(\Mm^{H_{\Ff}},\otimes_{\Ff},\sigma^{\Rr_{\Ff}})$ is a braided monoidal category. Note that, for any $X,Y$ in $\Mm^{H_{\Ff}}$, $X\otimes_\Ff Y$ is $X\otimes Y$ as vector space with right coaction $x\otimes y\mapsto x_0\otimes y_0\otimes x_1\cdot_\Ff y_1$, and the braiding is given by $\sigma^{\Rr_{\Ff}}_{X,Y}(x\otimes y)=y_{0}\otimes x_{0}\Rr_{\Ff}(x_{1}\otimes y_{1})$. We can consider the functor $\mathrm{Drin}:(\Mm^H, \otimes,\sigma^{\Rr})\to (\Mm^{H_\Ff},\otimes_\Ff,\sigma^{\Rr_\Ff})$ given by $\mathrm{Drin}(M)=M_\Ff$ where $M_\Ff$ is $M$ viewed as right $H_\Ff$-comodule. Given a coquasitriangular bialgebra $(H,\Rr)$ and $\Ff:H\otimes H\to\Bbbk$ a $2$-cocycle twist, then $\mathrm{Drin}:(\Mm^H, \otimes,\sigma^\Rr)\to (\Mm^{H_\Ff},\otimes_\Ff,\sigma^{\Rr_\Ff})$ is a braided monoidal equivalence of categories (dual of \cite[Lemma XV.3.7]{Kassel}).

The following result is dual to Theorem \ref{thm:twist}.

\begin{theorem}
Let $(H,\Rr,\chi)$ be a pre-Cartier co(quasi)triangular bialgebra and let $\Ff:H\otimes H\to\Bbbk$ be a 2-cocycle twist on $H$. Then, $(H_\Ff,\Rr_\Ff,\chi_\Ff)$ is a pre-Cartier co(quasi)triangular bialgebra, where $\chi_\Ff:=\Ff*\chi*\Ff^{-1}:H_\Ff\otimes H_\Ff\to\Bbbk$. % is defined by   for any $a,b\in H$ by $\chi_{\Ff}(a\otimes b)=\Ff(a_{1}\otimes b_{1})\chi(a_{2}\otimes b_{2})\Ff^{-1}(a_{3}\otimes b_{3}).$ 
Moreover, if $(H,\Rr,\chi)$ is Cartier, then so is $(H_\Ff,\Rr_\Ff,\chi_\Ff)$.
\end{theorem}

\begin{invisible} 
\begin{proof}
Assume that $(H,\Rr,\chi)$ is a pre-Cartier coquasitriangular bialgebra and define $\chi_\Ff:H_\Ff\otimes H_\Ff\to\Bbbk$, $\chi_{\Ff}(a\otimes b)=\Ff(a_{1}\otimes b_{1})\chi(a_{2}\otimes b_{2})\Ff^{-1}(a_{3}\otimes b_{3})$. We show that $\chi_{\Ff}$ satisfies \eqref{eq:CC1}, \eqref{eq:CC2}, \eqref{eq:CC3}, knowing that so does $\chi$.
Note that $m_\Ff=(u \Ff)*m*(u \Ff^{-1})$, $\Rr_{\Ff}=\Ff^{\mathrm{op}}*\Rr*\Ff^{-1}$ and $\chi_{\Ff}=\Ff*\chi*\Ff^{-1}$. We have that
\begin{displaymath}
	\begin{split}
		(u\chi_\Ff)&*m_\Ff=u(\Ff*\chi*\Ff^{-1})*(u \Ff)*m*(u \Ff^{-1})
		=(u \Ff)*(u\chi)*(u \Ff^{-1})*(u \Ff)*m*(u \Ff^{-1})\\
		&=(u \Ff)*(u\chi)*m*(u \Ff^{-1})
		=(u \Ff)*m*(u\chi)*(u \Ff^{-1})\\        
		&=(u \Ff)*m*(u \Ff^{-1})*(u \Ff)*(u\chi)*(u \Ff^{-1})
		=(u \Ff)*m*(u \Ff^{-1})*u (\Ff*\chi* \Ff^{-1})\\
		&=m_\Ff*(u\chi_\Ff)
	\end{split}
\end{displaymath}
and then $\chi_\Ff$ verifies \eqref{eq:CC1}. 

Next, we show that $\chi_\Ff$ satisfies \eqref{eq:CC2}. First we have 
\begin{displaymath}
	\begin{split}
		&\chi_{\Ff}(\mathrm{Id}\otimes m_{\Ff})=\chi_\Ff((u\varepsilon*\mathrm{Id}*u\varepsilon)\otimes(u \Ff*m*u \Ff^{-1}))=\chi_\Ff((u\varepsilon\otimes u \Ff)*(\mathrm{Id}\otimes m)*(u\varepsilon\otimes u\Ff^{-1}))\\&=\Ff_{23}*\chi_{\Ff}(\mathrm{Id}\otimes m)*\Ff^{-1}_{23}=\Ff_{23}*\Ff(\mathrm{Id}\otimes m)*\chi(\mathrm{Id}\otimes m)*\Ff^{-1}(\mathrm{Id}\otimes m)*\Ff^{-1}_{23}\\
		&\overset{\eqref{Dt2}}{=}\Ff_{12}*\Ff(m\otimes\mathrm{Id})*\chi(\mathrm{Id}\otimes m)*\Ff^{-1}(m\otimes\mathrm{Id})*\Ff^{-1}_{12}\\
		&=\Ff_{12}*\Ff(m\otimes\mathrm{Id})*(\chi_{12}+\Rr^{-1}_{12}*\chi_{13}*\Rr_{12})*\Ff^{-1}(m\otimes\mathrm{Id})*\Ff^{-1}_{12}\\&=\Ff_{12}*\Ff(m\otimes\mathrm{Id})*\chi_{12}*\Ff^{-1}(m\otimes\mathrm{Id})*\Ff^{-1}_{12}+\Ff_{12}*\Ff(m\otimes\mathrm{Id})*(\Rr^{-1}_{12}*\chi_{13}*\Rr_{12})*\Ff^{-1}(m\otimes\mathrm{Id})*\Ff^{-1}_{12}\\&\overset{(\star)}{=}\Ff_{12}*\chi_{12}*\Ff(m\otimes\mathrm{Id})*\Ff^{-1}(m\otimes\mathrm{Id})*\Ff^{-1}_{12}+\Ff_{12}*\Ff(m\otimes\mathrm{Id})*\Rr^{-1}_{12}*\chi_{13}*\Rr_{12}*\Ff^{-1}(m\otimes\mathrm{Id})*\Ff^{-1}_{12}
		\\&=\Ff_{12}*\chi_{12}*\Ff^{-1}_{12}+\Ff_{12}*\Ff(m\otimes\mathrm{Id})*\Rr^{-1}_{12}*\chi_{13}*\Rr_{12}*\Ff^{-1}(m\otimes\mathrm{Id})*\Ff^{-1}_{12}\\&=(\Ff\otimes\varepsilon)*(\chi\otimes\varepsilon)*(\Ff^{-1}\otimes\varepsilon)+\Ff_{12}*\Ff(m\otimes\mathrm{Id})*\Rr^{-1}_{12}*\chi_{13}*\Rr_{12}*\Ff^{-1}(m\otimes\mathrm{Id})*\Ff^{-1}_{12}\\&=(\Ff*\chi*\Ff^{-1})\otimes\varepsilon+\Ff_{12}*\Ff(m\otimes\mathrm{Id})*\Rr^{-1}_{12}*\chi_{13}*\Rr_{12}*\Ff^{-1}(m\otimes\mathrm{Id})*\Ff^{-1}_{12}\\&=(\chi_{\Ff})_{12}+\Ff_{12}*\Ff(m\otimes\mathrm{Id})*\Rr^{-1}_{12}*\chi_{13}*\Rr_{12}*\Ff^{-1}(m\otimes\mathrm{Id})*\Ff^{-1}_{12}
	\end{split}
\end{displaymath}
where $(\star)$ follows as $(\Ff(m\otimes\mathrm{Id})*\chi_{12})(a\otimes b\otimes c)=\Ff(a_{1}b_{1}\otimes c)\chi(a_{2}\otimes b_{2})\overset{\eqref{eq:CC1}}{=}\chi(a_{1}\otimes b_{1})\Ff(a_{2}b_{2}\otimes c)=(\chi_{12}*\Ff(m\otimes\mathrm{Id}))(a\otimes b\otimes c)
$ for every $a,b,c\in H$. Note that $(\Rr_{12}*\Ff^{-1}(m\otimes\mathrm{Id}))(a\otimes b\otimes c)=\Rr(a_1\otimes b_1)\Ff^{-1}(a_2b_2\otimes c)\overset{\eqref{eq:ct1}}{=}\Ff^{-1}(b_1a_1\otimes c)\Rr(a_2\otimes b_2)=(\Ff^{-1}(m^\mathrm{op}\otimes\mathrm{Id})*\Rr_{12})(a\otimes b\otimes c)$, hence $\Rr_{12}*\Ff^{-1}(m\otimes\mathrm{Id})=\Ff^{-1}(m^\mathrm{op}\otimes\mathrm{Id})*\Rr_{12}$, and then we have that $\Ff(m\otimes\mathrm{Id})*\Rr^{-1}_{12}=\Rr^{-1}_{12}*\Ff(m^{\mathrm{op}}\otimes\mathrm{Id})$. Thus we obtain 
\[
\chi_{\Ff}(\mathrm{Id}\otimes m_{\Ff})=(\chi_\Ff)_{12}+\Ff_{12}*\Rr^{-1}_{12}*\Ff(m^{\mathrm{op}}\otimes\mathrm{Id})*\chi_{13}*\Ff^{-1}(m^{\mathrm{op}}\otimes\mathrm{Id})*\Rr_{12}*\Ff^{-1}_{12}.
\]
On the other hand, we have that
\begin{displaymath}
	\begin{split}
		(\Rr_{\Ff}^{-1})_{12}*(\chi_{\Ff})_{13}*(\Rr_{\Ff})_{12}&=(\Ff*\Rr^{-1}*(\Ff^{\mathrm{op}})^{-1})_{12}*\Ff_{13}*\chi_{13}*\Ff^{-1}_{13}*\Ff^{\mathrm{op}}_{12}*\Rr_{12}*\Ff^{-1}_{12}\\&=\Ff_{12}*\Rr^{-1}_{12}*(\Ff^{\mathrm{op}})^{-1}_{12}*\Ff_{13}*\chi_{13}*\Ff^{-1}_{13}*\Ff^{\mathrm{op}}_{12}*\Rr_{12}*\Ff^{-1}_{12}\\&=\Ff_{12}*\Rr^{-1}_{12}*(\Ff^{-1})^{\mathrm{op}}_{12}*\Ff_{13}*\chi_{13}*\Ff^{-1}_{13}*\Ff^{\mathrm{op}}_{12}*\Rr_{12}*\Ff^{-1}_{12}
	\end{split}
\end{displaymath}

From \eqref{Dt2'} we obtain that $\Ff(m\otimes\mathrm{Id})=(\Ff^{-1}\otimes\varepsilon)*(\varepsilon\otimes \Ff)*\Ff(\mathrm{Id}\otimes m)$, and then
\[
\begin{split}
	(\Ff^{-1})^{\mathrm{op}}_{12}*\Ff_{13}&=(\Ff^{-1}\tau_{H,H}\otimes\varepsilon)*(\varepsilon\otimes \Ff)(\tau_{H,H}\otimes\mathrm{Id})\\&=(\Ff^{-1}\otimes\varepsilon)(\tau_{H,H}\otimes\mathrm{Id})*(\varepsilon\otimes \Ff)(\tau_{H,H}\otimes\mathrm{Id})\\&=((\Ff^{-1}\otimes\varepsilon)*(\varepsilon\otimes \Ff))(\tau_{H,H}\otimes\mathrm{Id})\\&=((\Ff^{-1}\otimes\varepsilon)*(\varepsilon\otimes \Ff)*\Ff(\mathrm{Id}\otimes m)*\Ff^{-1}(\mathrm{Id}\otimes m))(\tau_{H,H}\otimes\mathrm{Id})\\&=(\Ff(m\otimes\mathrm{Id})*\Ff^{-1}(\mathrm{Id}\otimes m))(\tau_{H,H}\otimes\mathrm{Id})\\&=\Ff(m\otimes\mathrm{Id})(\tau_{H,H}\otimes\mathrm{Id})*\Ff^{-1}(\mathrm{Id}\otimes m)(\tau_{H,H}\otimes\mathrm{Id})\\&=\Ff(m^{\mathrm{op}}\otimes\mathrm{Id})*\Ff^{-1}(\mathrm{Id}\otimes m)(\tau_{H,H}\otimes\mathrm{Id})
\end{split}
\]
Note that $\chi_{13}*\Ff(\mathrm{Id}\otimes m)(\tau_{H,H}\otimes\mathrm{Id})(a\otimes b\otimes c)=\chi(a_1\otimes c_1)\Ff(b\otimes a_2c_2)\overset{\eqref{eq:CC1}}{=}\Ff(b\otimes a_1c_1)\chi(a_2\otimes c_2)=\Ff(\mathrm{Id}\otimes m)(\tau_{H,H}\otimes\mathrm{Id})*\chi_{13}(a\otimes b\otimes c)$ and hence $\chi_{13}*\Ff(\mathrm{Id}\otimes m)(\tau_{H,H}\otimes\mathrm{Id})=\Ff(\mathrm{Id}\otimes m)(\tau_{H,H}\otimes\mathrm{Id})*\chi_{13}$.

Then we obtain $(\Rr_{\Ff}^{-1})_{12}*(\chi_{\Ff})_{13}*(\Rr_{\Ff})_{12}=\Ff_{12}*\Rr^{-1}_{12}*\Ff(m^{\mathrm{op}}\otimes\mathrm{Id})*\Ff^{-1}(\mathrm{Id}\otimes m)(\tau_{H,H}\otimes\mathrm{Id})*\chi_{13}*\Ff(\mathrm{Id}\otimes m)(\tau_{H,H}\otimes\mathrm{Id})*\Ff^{-1}(m^{\mathrm{op}}\otimes\mathrm{Id})*\Rr_{12}*\Ff^{-1}_{12}=\Ff_{12}*\Rr^{-1}_{12}*\Ff(m^{\mathrm{op}}\otimes\mathrm{Id})*\Ff^{-1}(\mathrm{Id}\otimes m)(\tau_{H,H}\otimes\mathrm{Id})*\Ff(\mathrm{Id}\otimes m)(\tau_{H,H}\otimes\mathrm{Id})*\chi_{13}*\Ff^{-1}(m^{\mathrm{op}}\otimes\mathrm{Id})*\Rr_{12}*\Ff^{-1}_{12}=\Ff_{12}*\Rr^{-1}_{12}*\Ff(m^{\mathrm{op}}\otimes\mathrm{Id})*\chi_{13}*\Ff^{-1}(m^{\mathrm{op}}\otimes\mathrm{Id})*\Rr_{12}*\Ff^{-1}_{12}.$

Thus,
$
\chi_{\Ff}(\mathrm{Id}\otimes m_{\Ff})=(\chi_{\Ff})_{12}+(\Rr_{\Ff}^{-1})_{12}*(\chi_{\Ff})_{13}*(\Rr_{\Ff})_{12},
$
so \eqref{eq:CC2} holds true for $\chi_\Ff$. Similarly, one can prove that $\chi_\Ff$ satisfies \eqref{eq:CC3}, and then $(H_\Ff,\Rr_\Ff,\chi_\Ff)$ is pre-Cartier. If $(H,\Rr,\chi)$ is Cartier, i.e., $\chi$ satisfies $\Rr*\chi=\chi^{\mathrm{op}}*\Rr$ in addition, we have that
\begin{displaymath}
\begin{split}    \Rr_\Ff*\chi_\Ff&=\Ff^\mathrm{op}*\Rr*\Ff^{-1}*\Ff*\chi*\Ff^{-1}=\Ff^\mathrm{op}*\Rr*\chi*\Ff^{-1}
=\Ff^\mathrm{op}*\chi^{\mathrm{op}}*\Rr*\Ff^{-1}\\
    &=\Ff^\mathrm{op}*\chi^\mathrm{op}*{\Ff^{-1}}^{\mathrm{op}}*\Ff^\mathrm{op}*\Rr*\Ff^{-1}
    =(\Ff*\chi*{\Ff^{-1}})^{\mathrm{op}}*\Ff^\mathrm{op}*\Rr*\Ff^{-1}
    =\chi_\Ff^\mathrm{op}*\Rr_\Ff,
\end{split}
\end{displaymath}
thus $\chi_\Ff$ verifies \eqref{eq:CC4}, and so $(H_\Ff,\Rr_\Ff,\chi_\Ff)$ is Cartier. 
\end{proof}\end{invisible}

\subsection{The role of Hochschild cohomology of algebras}
In this subsection we will show that, if $\mathrm{char}\left( \Bbbk \right) \neq 2$, the infinitesimal $\Rr$-form $\chi$ of a Cartier cotriangular Hopf
algebra $\left( H,\mathcal{R},\chi \right) $ is always a  $2$-coboundary in Hochschild cohomology of $H$ with coefficients in the $H$-bimodule $\Bbbk $ (regarded as a bimodule via the counit $\varepsilon $). \\

We use the dual of \cite[XVIII.5]{Kassel} for the Hochschild cohomology of algebras for this setting, see also \cite[Chapter 9]{Weibel94}. Consider $(H, m, u,\Delta,\varepsilon)$ a bialgebra over $\Bbbk$. Then we can consider the standard complex $$\xymatrix{\Bbbk\ar[r]^-{b^0}&\mathrm{Hom}_\Bbbk(H,\Bbbk)\ar[r]^-{b^1}&\mathrm{Hom}_\Bbbk(H\otimes H,\Bbbk)\ar[r]^-{b^2}&\mathrm{Hom}_\Bbbk(H\otimes H\otimes H,\Bbbk)\ar[r]^-{b^3}&\cdots}$$
where $b^0(k)(a)=\varepsilon(a)k-k\varepsilon(a)=0$,  $b^1(f)(a\otimes b)=\varepsilon(a)f(b)-f(ab)+f(a)\varepsilon(b)$, $b^2(f)(a\otimes b\otimes c)=\varepsilon(a)f(b\otimes c)-f(ab\otimes c)+f(a\otimes bc)-f(a\otimes b)\varepsilon(c)$ for $a,b,c\in H$, and so on.
%and $b^3(f)(a\otimes b\otimes c\otimes d)=af(b\otimes c\otimes d)-f(ab\otimes c\otimes d)+f(a\otimes bc\otimes d)-f(a\otimes b\otimes cd)+f(a\otimes b\otimes c)d.$
The elements in $\mathrm{Z}^n(H,\Bbbk ):=\mathrm{Ker}(b^n)$ are the $n$-\textit{cocycles} while the elements in $\mathrm{B}^n(H,\Bbbk):=\mathrm{Im}(b^{n-1})$ are the $n$-\textit{coboundaries}. The $n$-th Hochschild cohomology group is given by $\mathrm{H}^n(H,\Bbbk)=\frac{\mathrm{Z}^n(H,\Bbbk)}{\mathrm{B}^n(H,\Bbbk)}$. \\
\begin{comment}
Let $A$ be a $\Bbbk$-algebra and $M$ be an $A$-bimodule. Then we can consider the standard complex $$\xymatrix{0\ar[r]& M^A\ar[r]^-{b^0}&\mathrm{Hom}_\Bbbk(A,M)\ar[r]^-{b^1}&\mathrm{Hom}_\Bbbk(A\otimes A,M)\ar[r]^-{b^2}&\mathrm{Hom}_\Bbbk(A\otimes A\otimes A,M)\ar[r]^-{b^3}&\cdots}$$
where $M^A=\{m\in M\mid am=ma \text{ for all } a\in A\}$, $b^0(m)(a)=am-ma$,  $b^1(f)(a\otimes b)=af(b)-f(ab)+f(a)b$, $b^2(f)(a\otimes b\otimes c)=af(b\otimes c)-f(ab\otimes c)+f(a\otimes bc)-f(a\otimes b)c$ and $b^3(f)(a\otimes b\otimes c\otimes d)=af(b\otimes c\otimes d)-f(ab\otimes c\otimes d)+f(a\otimes bc\otimes d)-f(a\otimes b\otimes cd)+f(a\otimes b\otimes c)d.$
The elements in $\mathrm{Z}^n(A,M):=\mathrm{Ker}(b^n)$ are the $n$-\textit{cocycles} while the elements in $\mathrm{B}^n(A,M):=\mathrm{Im}(b^{n-1})$ are the $n$-\textit{coboundaries}. The Hochschild cohomology group is given by $\mathrm{H}^n(A,M)=\frac{\mathrm{Z}^n(A,M)}{\mathrm{B}^n(A,M)}.$  Let $H$ be a Hopf algebra and regard $\Bbbk$ as an $H$-bimodule via the counit $\varepsilon $.
Then we can consider $\mathrm{H}^n(H,\Bbbk)=\frac{\mathrm{Z}^n(H,\Bbbk)}{\mathrm{B}^n(H,\Bbbk)}$.
\end{comment}

Now we state the analogue results of Subsection \ref{secCohom} and the proofs follow similarly.

\begin{theorem}\label{thm:ccQYB}
    Let $(H,\Rr,\chi)$ be a pre-Cartier coquasitriangular bialgebra. Then the following equivalent statements hold.
\begin{enumerate}[i)]
\item $\chi$ is a Hochschild $2$-cocycle, i.e., 
$\chi_{12}+\chi(m\otimes\mathrm{Id})=\chi_{23}+\chi(\mathrm{Id}\otimes m)$;
\medskip
\item the \textbf{infinitesimal QYB equation} 
\begin{equation}\label{ccQYB}
\begin{split}
    \Rr_{12}*\chi_{12}&*\Rr_{13}*\Rr_{23}
    +\Rr_{12}*\Rr_{13}*\chi_{13}*\Rr_{23}
    +\Rr_{12}*\Rr_{13}*\Rr_{23}*\chi_{23}\\
    &=\Rr_{23}*\chi_{23}*\Rr_{13}*\Rr_{12}
    +\Rr_{23}*\Rr_{13}*\chi_{13}*\Rr_{12}
    +\Rr_{23}*\Rr_{13}*\Rr_{12}*\chi_{12}
\end{split}
\end{equation}
holds true;
\medskip
\item $\Rr_{12}^{-1}*\chi_{13}*\Rr_{12}=\Rr_{23}^{-1}*\chi_{13}*\Rr_{23}$.
\end{enumerate}
\end{theorem}

Dually to Definition \ref{def:coact}, we introduce the following right and left ``triangle'' $H$-actions on $H$:
\begin{align}
a\triangleleft b &:=\mathcal{R}^{-1}(a_{1}\otimes b_{1})a_{2}\mathcal{R}%
(a_{3}\otimes b_{2}),\label{lhd} \\
b\triangleright a &:=\mathcal{R}^{-1}(b_{1}\otimes a_{1})a_{2}\mathcal{R}%
(b_{2}\otimes a_{3}). \label{rhd} 
\end{align}%
Employing the above actions  \eqref{eq:CC2} and \eqref{eq:CC3}  can be rewritten as
\begin{align}
\chi (a\otimes bc)& =\chi (a\otimes b)\varepsilon (c)+\chi (a\triangleleft
b\otimes c),  \label{eq:CC2'} \\
\chi (ab\otimes c)& =\varepsilon (a)\chi (b\otimes c)+\chi (a\otimes
b\triangleright c)  \label{eq:CC3'}
\end{align}
for every $a,b,c\in H$.

\begin{comment}
We have that $\triangleleft$ is a right $H$-action on $H$, indeed 
\[
a\triangleleft 1_{H}=\mathcal{R}^{-1}(a_{1}\otimes1_{H})a_{2}\mathcal{R}(a_{3}\otimes
1_{H})=\varepsilon (a_{1})a_{2}\varepsilon (a_{3})=a 
\]
and
\begin{eqnarray*}
\left( a\triangleleft b\right) \triangleleft c &=&\mathcal{R}^{-1}(a_{1}\otimes
b_{1})a_{2}\triangleleft c\mathcal{R}(a_{3}\otimes b_{2}) \\
&=&\mathcal{R}^{-1}(a_{1}\otimes b_{1})\mathcal{R}^{-1}(a_{2}\otimes c_{1})a_{3}%
\mathcal{R}(a_{4}\otimes c_{2})\mathcal{R}(a_{5}\otimes b_{2}) \\
&=&\mathcal{R}^{-1}(a_{1}\otimes b_{1}c_{1})a_{2}\mathcal{R}(a_{3}\otimes
b_{2}c_{2})=a\triangleleft \left( bc\right) .
\end{eqnarray*}%
Similarly one can check that $\triangleright$ is a left $H$-action on $H$.
\end{comment}

%We have $1\triangleright a:=\mathcal{R}(a_{1}\otimes 1)a_{2}\mathcal{R}(1\otimes
%a_{3})=\varepsilon (a_{1})a_{2}\varepsilon (a_{3})=a$ and
%\begin{eqnarray*}
%c\triangleright \left( b\triangleright a\right)  &=&\mathcal{R}(a_{1}\otimes
%b_{1})c\triangleright a_{2}\mathcal{R}(b_{2}\otimes a_{3}) \\
%&=&\mathcal{R}(a_{1}\otimes b_{1})\mathcal{R}(a_{2}\otimes c_{1})a_{3}%
%\mathcal{R}(c_{2}\otimes a_{4})\mathcal{R}(b_{2}\otimes a_{5}) \\
%&=&\mathcal{R}(a_{1}\otimes c_{1}b_{1})a_{2}\mathcal{R}(c_{2}b_{2}\otimes
%a_{3})=\left( cb\right) \triangleright a
%\end{eqnarray*}%

 Note that by Theorem \ref{thm:ccQYB} one has  $\Rr^{-1}_{12}*\chi_{13}*\Rr_{12}=\Rr^{-1}_{23}*\chi_{13}*\Rr_{23}$, i.e., $\chi (a\triangleleft b\otimes c)=\chi (a\otimes b\triangleright c)$.

\begin{lemma}\label{lem:co-chiS1}
Let $\left( H,\mathcal{R},\chi \right) $ be a pre-Cartier cotriangular Hopf
algebra with antipode $S$. Then, for all $a,b\in H$, the following properties hold:
\begin{align}
S\left( a\triangleright b\right)  &=S\left( b\right) \triangleleft S\left(
a\right) ;  \label{eq:Striang} \\
\chi (S\left( a_{1}\right) \otimes a_{2}\triangleright b) &=-\chi (a\otimes
b).  \label{eq:chiS}
\end{align}
Furthermore, if $\left( H,\mathcal{R},\chi \right) $ is Cartier, then
\begin{equation}\label{eq:chiSS}
\chi (S\left( b\right) \otimes S(a))=\chi (a\otimes
b) \text{   for all   }a,b\in H,
\end{equation}
in addition.
\end{lemma}

We now  give for future reference the dual of Proposition \ref{pro:trivHoch-cohomol-coalg} and Theorem \ref{thm:Hoch-cohom-coalg}. Note that we include a second proof of the latter, employing the Casimir operator.

\begin{proposition}\label{pro:trivHochCohom}
Let $H$ be a Hopf algebra and $\chi \in \mathrm{Z}^2(H,\Bbbk)$ be such that $\chi \left(
1\otimes 1\right) =0$ and define $\gamma:H\to \Bbbk,x\mapsto \chi
\left( S\left( x_{1}\right) \otimes x_{2}\right) $. Then, we get $b^{1}\left( \gamma \right)  =\chi +\chi^\op \left( S\otimes S\right)$. Moreover, if we
assume $\chi^\op(S\otimes S)=\chi$, then $b^{1}\left( \gamma
\right) =2\chi $. In particular, if $\mathrm{char}\left( \Bbbk \right) \neq 2
$, we have $\chi =b^{1}\left( \frac{\gamma }{2}\right) \in \mathrm{B}^2(H,\Bbbk)$.
\end{proposition}

In analogy to the dual setting we can prove that $\gamma (x)=\chi \left( S\left( x_{1}\right) \otimes x_{2}\right)$ defined for $x\in H$ is a central element $\gamma\in \mathscr{Z}(H^*)$. We call it the \textbf{Casimir form} of the pre-Cartier coquasitriangular Hopf
algebra $\left( H,\mathcal{R},\chi \right) $.

\begin{theorem}
Let $\gamma$ be the Casimir form of a Cartier cotriangular Hopf
algebra $\left( H,\mathcal{R},\chi \right) $. Then $b^{1}\left( \gamma
\right)=2\chi$. In particular,
if $\mathrm{char}\left( \Bbbk \right) \neq 2$ we have $\chi =b^{1}\left(
\frac{\gamma }{2}\right)\in \mathrm{B}^2(H,\Bbbk) $, i.e., $\chi$ is a 2-coboundary.
\end{theorem}

\begin{proof}
\begin{comment}
In view of \eqref{eq:CC2'} and \eqref{eq:CC3'}, we have $$b^2(\chi)(a\otimes b\otimes c)=\varepsilon \left(
a\right) \chi \left( b\otimes c\right) -\chi \left( ab\otimes c\right) +\chi
\left( a\otimes bc\right) -\chi \left( a\otimes b\right) \varepsilon \left(
c\right) =-\chi(a\otimes b\triangleright c)+\chi(a\triangleleft b\otimes c)\overset{\eqref{eq:chibalanced}}=0$$ so that $\chi\in \mathrm{Z}^2(H,\Bbbk)$. 
\end{comment}
By Theorem \ref{thm:ccQYB}, we know that $\chi\in \mathrm{Z}^2(H,\Bbbk)$. If in \eqref{eq:CC2'} we take $a=b=c=1$, we get $\chi(1\otimes 1)=0$. Moreover, by Lemma \ref{lem:co-chiS1}, we have $\chi (S\left( b\right) \otimes S(a)) =\chi (a\otimes b)$. We conclude by Proposition \ref{pro:trivHochCohom}. 

We also include a second proof. Note that the category $\Mm_{f}^{H}$ of finite-dimensional comodules is rigid. Given an object $V$ in $\Mm_{f}^{H}$, its dual
is the linear dual $V^{\ast }$ where the comodule structure is given, for
every $f\in V^{\ast },$ by the equality $f_{0}\left( v\right) f_{1}=f\left(
v_{0}\right) S\left( v_{1}\right) $, for every $v\in V$. By \cite[page 495]{Kassel}, for $V,W$ in $\Mm
_{f}^{H}$, we have that the infinitesimal braiding $t _{V,W}$ satisfies the equality $2t _{V,W}=C_{V\otimes W}-C_{V}\otimes W-V\otimes C_{W}$
where $C_{V}:V\rightarrow V$, the so-called Casimir operator of $V$,
is defined by setting $C_{V}:=-\left( V\otimes \left( \mathrm{ev}_{V}\circ t
_{V^{\ast },V}\right) \right) \circ \left( \mathrm{coev}_{V}\otimes V\right)
.$ If $\left\{ e_{i}\mid i\in I\right\} $ is a basis of $V$ and $\left\{ e^{i}\mid i\in I\right\} $ is the corresponding dual basis of $V^*$, we can compute%
\begin{align*}
C_{V}\left( v\right)  &=-\left( V\otimes \left( \mathrm{ev}_{V}t
_{V^{\ast },V}\right) \right) \left( \mathrm{coev}_{V}\otimes V\right)
\left( v\right)  
=-\left( V\otimes \left( \mathrm{ev}_{V}t _{V^{\ast },V}\right) \right)
\left( e_{i}\otimes e^{i}\otimes v\right)  \\
&=-\left( V\otimes \mathrm{ev}_{V}\right) \left( e_{i}\otimes
e_{0}^{i}\otimes v_{0}\right) \chi \left( e_{1}^{i}\otimes v_{1}\right)  
=-e_{i}e_{0}^{i}\left( v_{0}\right) \chi \left( e_{1}^{i}\otimes
v_{1}\right)  
=-e_{i}\chi \left( e_{0}^{i}\left( v_{0}\right) e_{1}^{i}\otimes
v_{1}\right)\\  
&=-e_{i}\chi \left( e^{i}\left( v_{0}\right) S\left( v_{1}\right) \otimes
v_{2}\right)  
=-e^{i}\left( v_{0}\right) e_{i}\chi \left( S\left( v_{1}\right) \otimes
v_{2}\right) 
=-v_{0}\gamma \left( v_{1}\right) .
\end{align*}%
As a consequence, for every $V,W$ in $\Mm_{f}^{H}$ and $v\in V,w\in
W$, we have
\begin{eqnarray*}
2t _{V,W}\left( v\otimes w\right)  &=&C_{V\otimes W}\left( v\otimes
w\right) -C_{V}\left( v\right) \otimes w-v\otimes C_{W}\left( w\right)  \\
&=&-v_{0}\otimes w_{0}\gamma \left( v_{1}w_{1}\right) +v_{0}\otimes w\gamma
\left( v_{1}\right) +v\otimes w_{0}\gamma \left( w_{1}\right)  \\
&=&-v_{0}\otimes w_{0}\gamma \left( v_{1}w_{1}\right) +v_{0}\otimes
w_{0}\gamma \left( v_{1}\right) \varepsilon \left( w_{1}\right)
+v_{0}\otimes w_{0}\varepsilon \left( v_{1}\right) \gamma \left(
w_{1}\right)  \\
&=&v_{0}\otimes w_{0}b^{1}\left( \gamma \right) \left( v_{1}\otimes
w_{1}\right) .
\end{eqnarray*}%
Now, given $x,y\in H,$ by the fundamental theorem of coalgebras, there are
finite-dimensional subcoalgebras $D_{x},D_{y}$ of $H$ such that $x\in D_{x}$
and $y\in D_{y}.$ Clearly $D_{x}$ and $D_{y}$ are in $\Mm_{f}^{H}$
so that%
\begin{equation*}
2x_{1}\otimes y_{1}\chi \left( x_{2}\otimes y_{2}\right) =2t _{H,H}\left(
x\otimes y\right) =2t _{D_{x},D_{y}}\left( x\otimes y\right)
=x_{1}\otimes y_{1}b^{1}\left( \gamma \right) \left( x_{2}\otimes
y_{2}\right) .
\end{equation*}%
By applying $\varepsilon \otimes \varepsilon $ we get $2\chi \left( x\otimes
y\right) =b^{1}\left( \gamma \right) \left( x\otimes y\right) $, i.e., $2\chi
=b^{1}\left( \gamma \right) $ as desired.
\end{proof}

Dually to Example \ref{exa:psc}, we have the following one.
\begin{example}
Let $(H,\Rr)$ be a co(quasi)triangular Hopf algebra and let $\chi:H\otimes H\to \Bbbk$ be a central symmetric biderivation.
By Example \ref{exa:chisymco}, we know that $(H,\Rr,\chi)$ is Cartier so that we can consider the Casimir form $\gamma:H\to \Bbbk$. Note that if $d:H\to\Bbbk$ is a derivation, then $0=d(x_1S(x_2))=d(x)+dS(x)$ for $x\in H$, so that $dS=-d$. Since $\chi$ is a derivation in both entries, we get 
$\gamma (x)=\chi \left( S\left( x_{1}\right) \otimes x_{2}\right)=-\chi \left( x_{1}\otimes x_{2}\right) $ and 
$\chi \left( S\left( y\right) \otimes S\left( x\right)
\right)=\chi \left( y \otimes  x\right)=\chi \left( x \otimes  y\right)$. By Proposition \ref{pro:trivHochCohom}, we get that 
$b^{1}\left( \gamma \right) =2 \chi $ and hence $\chi$ is a 2-coboundary if $\mathrm{char}\left( \Bbbk \right) \neq 2$.
An instance of this situation is the Cartier coquasitriangular Hopf algebra $(\mathrm{GL}_{q}(2),\Rr,\chi)$ of Proposition \ref{prop:GLq}. Thus, being cotriangular is not necessary to have that $\chi$ is a 2-coboundary.
\end{example}

We include here two results that will be needed afterwards. The first one provides compatibility conditions between the triangle actions and the product. We omit their elementary proof, merely noting that these conditions are part of the definition of \emph{matched pair of bialgebras}, see \cite[Definition 7.2.1]{Majid-book}, of which $(H,H)$ together with the triangle actions provides an instance in view of \cite[Theorem 7.2.3]{Majid-book} and \cite[Example 7.2.7]{Majid-book} applied to $\sigma=\Rr$. 
\begin{invisible}
    Example 7.2.7 says that we have $A\bowtie_\sigma H$. One checks that this bialgebra is the  setting of Theorem 7.2.3 so that $A,H$ are a matched pair.
\end{invisible}

The second result  will be needed for the infinitesimal FRT construction in Section \ref{Sec:FRT}.

\begin{lemma}\label{lem:bitriang}
Let $H$ be a bialgebra and let $\mathcal{R}:H\otimes H\to \Bbbk$ be a convolution-invertible map which is a bialgebra bicharacter, i.e., it satisfies \eqref{eq:ct2} and \eqref{eq:ct3}. Define $%
\triangleright :H\otimes H\rightarrow H$ and $\triangleleft :H\otimes
H\rightarrow H$ as in \eqref{lhd} and \eqref{rhd}. Then the following identities
hold true for every $a,b,c\in H$.%
\begin{gather}
a\triangleright \left( bc\right) =\left( a_{1}\triangleright b_{1}\right)
\left( \left( a_{2}\triangleleft b_{2}\right) \triangleright c\right) \qquad
\text{and}\qquad \left( ab\right) \triangleleft c=\left( a\triangleleft
\left( b_{1}\triangleright c_{1}\right) \right) \left( b_{2}\triangleleft
c_{2}\right) ,  \label{form:trimult} \\
\varepsilon \left( a\triangleleft b\right) =\varepsilon \left( a\right)
\varepsilon \left( b\right) =\varepsilon \left( a\triangleright b\right) .
\label{form:trieps}
\end{gather}
\end{lemma}

\begin{proposition}\label{prop:bitriang}
Let $C$ be a coalgebra together with two actions $\triangleright :TC\otimes
TC\rightarrow TC$ and $\triangleleft :TC\otimes TC\rightarrow TC$ such that 
\eqref{form:trimult} and \eqref{form:trieps} holds true for all $a,b,c\in TC$
and such that $TC\triangleright C^{\otimes n}\subseteq C^{\otimes n}$ and $%
C^{\otimes n}\triangleleft TC\subseteq C^{\otimes n}.$ Then every linear map
$\chi _{11}:C\otimes C\rightarrow \Bbbk $ induces a linear map $\chi
:TC\otimes TC\rightarrow \Bbbk $ satisfying \eqref{eq:CC2'} and \eqref{eq:CC3'}.%
%\begin{eqnarray*}
%\chi \left( a\otimes bc\right)  &=&\chi \left( a\otimes b\right) \varepsilon
%\left( c\right) +\chi \left( a\triangleleft %b\otimes c\right)  \\
%\chi \left( ab\otimes c\right)  &=&\varepsilon %\left( a\right) \chi \left(
%b\otimes c\right) +\chi \left( a\otimes %b\triangleright c\right)
%\end{eqnarray*}
\end{proposition} 

\begin{proof} We define $\chi :TC\otimes TC\rightarrow \Bbbk $ through its components $%
\chi _{mn}:C^{\otimes m}\otimes C^{\otimes n}\rightarrow \Bbbk $ with $%
m,n\in \mathbb{N}$. We set $\chi _{mn}$ to be zero if either $m=0$ or $n=0$.
For $n>0$, define $\chi _{1n}:C\otimes C^{\otimes n}\rightarrow \Bbbk $, by
setting for every $a,b^{1},b^{2},\ldots ,b^{i},\ldots ,b^{n}\in C$%
\begin{equation*}
\chi _{1n}\left( a\otimes b^{1}b^{2}\cdots b^{n}\right) :=\chi _{11}\left(
a\otimes b^{1}\right) \varepsilon \left( b^{2}\cdots b^{n}\right)
+\sum\limits_{i=2}^{n}\chi _{11}\left( a\triangleleft b^{1}b^{2}\cdots
b^{i-1}\otimes b^{i}\right) \varepsilon \left( b^{i+1}b^{i+2}\cdots
b^{n}\right) .  \label{def:chi1n}
\end{equation*}% 
Note that this makes sense since $a\triangleleft b^{1}b^{2}\cdots b^{i-1}\in
C\triangleleft TC\subseteq C.$ 
By using the definition of $\chi _{1n}$ and associativity of $\triangleleft$ one easily checks that 
\begin{equation}\label{eq:CC2'chi1n}
 \chi _{1t}\left( a\otimes b\right) \varepsilon \left( c\right) +\chi
_{1\left( n-t\right) }\left( a\triangleleft b\otimes c\right)=\chi _{1n}\left( a\otimes bc\right)
\end{equation}
for all $a\in C$, $b\in
C^{\otimes t}$ and $c\in C^{\otimes \left(n-t\right) }$.
 Define now $\chi
_{mn}:C^{\otimes m}\otimes C^{\otimes n}\rightarrow \Bbbk $, by setting for
every $a^{1},a^{2},\ldots ,a^{i},\ldots ,a^{m}\in C$ and $c\in C^{\otimes n}$%
\begin{equation*}
\chi _{mn}\left( a^{m}\cdots a^{2}a^{1}\otimes c\right) :=\varepsilon \left(
a^{m}\cdots a^{2}\right) \chi _{1n}\left( a^{1}\otimes c\right)
+\sum\limits_{i=2}^{m}\varepsilon \left( a^{m}\cdots a^{i+2}a^{i+1}\right)
\chi _{1n}\left( a^{i}\otimes a^{i-1}\cdots a^{2}a^{1}\triangleright
c\right) .  \label{def:chimn}
\end{equation*}%
Again this makes sense since $a^{i-1}\cdots a^{2}a^{1}\triangleright c\in
TC\triangleright C^{\otimes n}\subseteq C^{\otimes n}.$
On the one hand, by means of the definition of $\chi_{m,n}$ and the associativity of $\triangleright$,
one easily checks that 
$\varepsilon \left( a\right) \chi _{tn}\left( b\otimes c\right) +\chi
_{\left( m-t\right) n}\left( a\otimes b\triangleright c\right)
=\chi _{mn}\left( ab\otimes c\right)$ for every $a\in C^{\otimes (m-t)}$, $b\in C^{\otimes t}$ and  $c\in C^{\otimes n}$.
On the other hand, by using the definition of $\chi _{mn}$ together with the equalities  \eqref{form:trimult} (and an iterated version of it), \eqref{form:trieps}  and  \eqref{eq:CC2'chi1n}, we can prove that
$\chi _{mn}\left( a\otimes bc\right) =\chi _{(m+s)(n-s)}\left( ab\otimes c\right)$ for every  $a\in C^{\otimes m},b\in C^{\otimes s}$ and $c\in
C^{\otimes \left( n-s\right) }$. 
\end{proof}

\subsection{The finite dual}\label{Sec:FinDual}
Let $(H,\Rr)$ be a quasitriangular bialgebra. By \cite[Proposition 14.2.2]{Radford}, if we consider the finite dual $H^\circ$ of $H$ and the linear map $\Rr^\circ:H^\circ\otimes H^\circ\to\Bbbk$, defined by $\Rr^\circ(p\otimes q):=(p\otimes q)(\Rr)$, for every $p,q\in H^\circ$, then $(H^\circ,\Rr^\circ)$ results to be a coquasitriangular bialgebra. The convolution inverse of $\Rr^\circ$ is given by $(\Rr^{\circ})^{-1}:H^{\circ}\otimes H^\circ\to\Bbbk,\ p\otimes q\mapsto(p\otimes q)(\Rr^{-1})$, where $\Rr^{-1}=\overline{\Rr}^i\otimes \overline{\Rr}_i\in H\otimes H$ is the inverse of $\Rr$. Note that if $(H,\Rr)$ is a triangular bialgebra then $(H^\circ,\Rr^\circ)$ is a cotriangular bialgebra. \\
\begin{comment}
In fact, for every $p,q\in H^\circ$,
	\[
	\begin{split}
	(\tilde{\Rr}*\tilde{\Rr}^{-1})(p\otimes q)&=\tilde{\Rr}(p_{1}\otimes q_{1})\tilde{\Rr}^{-1}(p_{2}\otimes q_{2})=(p_{1}\otimes q_{1})(\Rr)(p_{2}\otimes q_{2})(\Rr^{-1})\\&=(p\otimes q)(\Rr\Rr^{-1})=(p\otimes q)(1_H\otimes 1_H)=p(1_{H})\otimes q(1_{H})\\&=\epsilon_{H^\circ}(p)\otimes\epsilon_{H^\circ}(q)=\varepsilon_{H^\circ}\otimes\varepsilon_{H^\circ}(p\otimes q)
	\end{split}
	\]
	Analogously, we have: \[
	\begin{split}
	(\tilde{\Rr}^{-1}*\tilde{\Rr})(p\otimes q)&=\tilde{\Rr}^{-1}(p_{1}\otimes q_{1})\tilde{\Rr}(p_{2}\otimes q_{2})=(p_{1}\otimes q_{1})(\Rr^{-1})(p_{2}\otimes q_{2})(\Rr)\\&=(p\otimes q)(\Rr^{-1}\Rr)=(p\otimes q)(1_H\otimes 1_H)=p(1_{H})\otimes q(1_{H})\\&=\varepsilon_{H^\circ}(p)\otimes\varepsilon_{H^\circ}(q)=\varepsilon_{H^\circ}\otimes\varepsilon_{H^\circ}(p\otimes q)
	\end{split}
	\]
\end{comment}

Now we show that the finite dual of a (pre-)Cartier (quasi)triangular bialgebra is a (pre-)Cartier co(quasi)triangular bialgebra. Recall that the bialgebra structure of $H^\circ$ is given by $m_{H^\circ}:H^\circ\otimes H^\circ\to H^\circ$, $\alpha\otimes\beta\mapsto \alpha *\beta$, $u_{H^\circ}:\Bbbk\to H^\circ$, $k\mapsto k\varepsilon_H$, $\Delta_{H^\circ}:H^\circ\to H^\circ\otimes H^\circ$, $\alpha\mapsto \alpha_1\otimes\alpha_2$, where $\alpha(ab)=\alpha_1(a)\alpha_2(b)$, for all $a,b\in H$, and $\varepsilon_{H^\circ}:H^\circ\to\Bbbk$, $\alpha\mapsto\alpha(1_H)$. 

\begin{proposition}\label{prop:finitedual}
	Let $(H,\Rr, \chi)$ be a pre-Cartier (quasi)triangular bialgebra. Then, $(H^\circ,\Rr^\circ, \chi^\circ)$ is a pre-Cartier co(quasi)triangular bialgebra, where $\chi^\circ :H^\circ\otimes H^\circ\to\Bbbk,\ p\otimes q\mapsto(p\otimes q)(\chi)$. Moreover, if $(H,\Rr,\chi)$ is a Cartier (quasi)triangular bialgebra, then $(H^\circ,\Rr^\circ, \chi^\circ)$ is a Cartier co(quasi)triangular bialgebra.
 %\rd{\\ PROPOSAL: I would use the notations $\Rr^\circ,\chi^\circ$. }
\end{proposition}

\begin{proof}
If $(H,\Rr,\chi)$ is a pre-Cartier quasitriangular bialgebra, then $\chi=\chi^i\otimes\chi_i\in H\otimes H$ satisfies \eqref{cqtr1}, \eqref{cqtr2}, \eqref{cqtr3}. %	If $(H,\Rr)$ is a pre-Cartier quasitriangular bialgebra, then there is an element $\chi=\chi^i\otimes\chi_i\in H\otimes H$ such that \eqref{cqtr1}, \eqref{cqtr2}, \eqref{cqtr3} hold.
 Define $\chi^\circ :H^\circ\otimes H^\circ\to\Bbbk,\ p\otimes q\mapsto(p\otimes q)(\chi)=p(\chi^i)q(\chi_i)$, for every $p,q\in H^\circ$. We show that $\chi^\circ$ satisfies \eqref{eq:CC1}, \eqref{eq:CC2}, \eqref{eq:CC3}. Indeed, we have $(u_{H^\circ} \chi^\circ)*m_{H^\circ}=m_{H^\circ}*(u_{H^\circ}\chi^\circ)$, as for every $p,q\in H^\circ$,
	\begin{displaymath}
	\begin{split}
	\chi^\circ (p_{1}\otimes q_{1})(p_{2}*q_{2})(a)&=p_{1}(\chi^{i})q_{1}(\chi_{i})p_{2}(a_{1})q_{2}(a_{2})=p_{1}(\chi^{i})p_{2}(a_{1})q_{1}(\chi_{i})q_{2}(a_{2})\\&=p(\chi^ia_{1})q(\chi_{i}a_{2})\overset{\eqref{cqtr1}}{=}p(a_{1}\chi^{i})q(a_{2}\chi_{i})=p_{1}(a_{1})p_{2}(\chi^i)q_{1}(a_{2})q_{2}(\chi_i)\\&=p_{1}(a_{1})q_{1}(a_{2})p_{2}(\chi^i)q_{2}(\chi_i)=(p_{1}*q_{1})(a) \chi^\circ (p_{2}\otimes q_{2})
	\end{split}
	\end{displaymath}
	for any $a\in H$, thus \eqref{eq:CC1} holds. Now, for every $p,q,r\in H^\circ$ we have that $\chi^\circ (\mathrm{Id}_{H^\circ}\otimes m_{H^\circ})(p\otimes q\otimes r)=\chi^\circ (p\otimes q*r)=p(\chi^i)q(\chi_{i_1})r(\chi_{i_2})$. On the other hand, using \eqref{cqtr2}, i.e., $(\mathrm{Id}_H\otimes\Delta)(\chi)=\chi_{12}+\Rr^{-1}_{12}\chi_{13}\Rr_{12}$, for every $p,q,r\in H^\circ$ we have that 
	\[
	\begin{split}
	\chi^\circ_{12}&(p\otimes q\otimes r)+(\Rr^\circ)^{-1}_{12}(p_{1}\otimes q_{1}\otimes r_{1})\chi^\circ_{13}(p_{2}\otimes q_{2}\otimes r_{2})\Rr^\circ_{12}(p_{3}\otimes q_{3}\otimes r_{3})\\
 &=\chi^\circ (p\otimes q)\varepsilon_{H^\circ}(r)+(\Rr^\circ)^{-1}(p_{1}\otimes q_{1})\varepsilon_{H^\circ}(r_{1})\chi(p_{2}\otimes r_{2})\varepsilon_{H^\circ}(q_{2})\Rr^\circ (p_{3}\otimes q_{3})\varepsilon_{H^\circ}(r_{3})\\
 &=p(\chi^i)q(\chi_i)r(1_H)+p_{1}(\overline{\Rr}^i)q_{1}(\overline{\Rr}_{i})r_{1}(1_H)p_{2}(\chi^j)r_{2}(\chi_j)q_{2}(1_H)p_{3}(\Rr^k)q_{3}(\Rr_k)r_{3}(1_H)\\
 &=p(\chi^i)q(\chi_i)r(1_H)+p(\overline{\Rr}^{i}\chi^{j}\Rr^{k})q(\overline{\Rr}_{i}\Rr_{k})r(\chi_j)\overset{\eqref{cqtr2}}{=}p(\chi^i)q(\chi_{i_1})r(\chi_{i_2})
	\end{split}
	\]
	hence $\chi^\circ(\mathrm{Id}_{H^{\circ}}\otimes m_{H^\circ})=\chi^\circ_{12}+(\Rr^\circ)^{-1}_{12}*\chi^\circ_{13}*\Rr^\circ_{12}$ and \eqref{eq:CC2} holds. Similarly, since $H$ satisfies \eqref{cqtr3}, then  \eqref{eq:CC3} holds for $H^{\circ}$.
 \begin{comment} 
Finally, we show that $\tilde{\chi}(m_{H^\circ}\otimes\mathrm{Id}_{H^\circ})=\tilde{\chi}_{23}+\tilde{\Rr}^{-1}_{23}*\tilde{\chi}_{13}*\tilde{\Rr}_{23}$. For every $p,q,r\in H^\circ$ we have that $\tilde{\chi}(m_{H^\circ}\otimes\mathrm{Id}_{H^\circ})(p\otimes q\otimes r)=\tilde{\chi}(p*q\otimes r)=p(\chi^i_1)q(\chi^i_2)r(\chi_i)$, and by using \eqref{cqtr3}, i.e., $(\Delta\otimes\mathrm{Id}_H)(\chi)=\chi_{23}+\Rr^{-1}_{23}\chi_{13}\Rr_{23}$, we have that
	\[
	\begin{split}
	&\tilde{\chi}_{23}(p\otimes q\otimes r)+\tilde{\Rr}^{-1}_{23}(p_{1}\otimes q_{1}\otimes r_{1})\tilde{\chi}_{13}(p_{2}\otimes q_{2}\otimes r_{2})\tilde{\Rr}_{23}(p_{3}\otimes q_{3}\otimes r_{3})=\\&\varepsilon_{H^\circ}(p)\tilde{\chi}(q\otimes r)+\varepsilon_{H^\circ}(p_{1})\tilde{\Rr}^{-1}(q_{1}\otimes r_{1})\tilde{\chi}(p_{2}\otimes r_{2})\varepsilon_{H^\circ}(q_{2})\tilde{\Rr}(q_{3}\otimes r_{3})\epsilon_{H^\circ}(p_{3})=\\&p(1_H)q(\chi^i)r(\chi_i)+p_{1}(1_H)q_{1}(\overline{\Rr}^i)r_{1}(\overline{\Rr}_{i})p_{2}(\chi^j)r_{2}(\chi_{j})q_{2}(1_H)q_{3}(\Rr^{k})r_{3}(\Rr_{k})p_{3}(1_H)=\\&p(1_H)q(\chi^i)r(\chi_i)+p(\chi^j)q(\overline{\Rr}^{i}\Rr^k)r(\overline{\Rr}_{i}\chi_{j}\Rr_{k})\overset{\eqref{cqtr3}}{=}p(\chi^i_{1})q(\chi^i_{2})r(\chi_i)
	\end{split}
	\]
	thus \eqref{eq:CC3} holds. 
\end{comment} 
Moreover, if $(H,\Rr,\chi)$ is a Cartier quasitriangular bialgebra, that is, $\chi$ satisfies $\Rr\chi=\chi^{\mathrm{op}}\Rr$ in addition, then we obtain that
	\[
	\begin{split}
	(\Rr^\circ *\chi^\circ )(p\otimes q)&=\Rr^\circ (p_{1}\otimes q_{1})\chi^\circ (p_{2}\otimes q_{2})=p_{1}(\Rr^i)q_{1}(\Rr_i)p_{2}(\chi^j)q_{2}(\chi_{j})=p(\Rr^{i}\chi^{j})q(\Rr_{i}\chi_{j})\\&=p(\chi_{j}\Rr^{i})q(\chi^{j}\Rr_{i})=p_{1}(\chi_{j})p_{2}(\Rr^{i})q_{1}(\chi^j)q_{2}(\Rr_{i})
	\\&=q_{1}(\chi^j)p_{1}(\chi_{j})p_{2}(\Rr^{i})q_{2}(\Rr_{i})=(\chi^\circ\tau_{H^\circ,H^\circ}*\Rr^\circ)(p\otimes q).
	\end{split}
	\]
	Thus, $(H^\circ,\Rr^\circ,\chi^\circ)$ is a Cartier coquasitriangular bialgebra.
\end{proof}
\begin{corollary} Let $(H,\Rr,\chi)$ be a finite-dimensional (pre-)Cartier (quasi)triangular bialgebra. Then, the dual $(H^*,\Rr^\circ, \chi^\circ)$ is a (pre-)Cartier co(quasi)triangular bialgebra. %Moreover, if $(H,\Rr)$ is a finite-dimensional Cartier (quasi)triangular bialgebra, then $(H^*,\tilde{\Rr})$ is a Cartier co(quasi)triangular bialgebra.
\end{corollary}
\begin{proof}
	If $H$ is finite-dimensional, then $H^*=H^\circ$, and we conclude by Proposition \ref{prop:finitedual}.
\end{proof}

Given a linear map $\beta:H\otimes H\to\Bbbk$ we can consider $\beta_l,\beta_r:H\to H^{*}$ such that $\beta_l(u)(v)=\beta(u\otimes v)=\beta_r(v)(u)$. Note that $\beta=\sum_{i=1}^{n}{f_{i}\otimes g_{i}}$, where $f_i,g_i\in H^{*}$ and $1\leq i\leq n$, is equivalent to $\text{dim(Im}(\beta_l))$ finite and to $\text{dim(Im}(\beta_r))$ finite by \cite[Exercise 1.3.18]{Radford} and if this happens $\beta$ is said of $\textit{finite type}$. By \cite[Proposition 14.2.3]{Radford} we know that if $(H,\Rr)$ is a coquasitriangular bialgebra with universal $\Rr$-form of finite type then $\Rr=\Rr^i\otimes \Rr_i\in H^{\circ}\otimes H^{\circ}$ and $(H^{\circ},\Rr)$ is a quasitriangular bialgebra.  We can now show the following result.
\par

\begin{proposition}\label{prop:finitedual-coquasi}
	Let $(H,\Rr,\chi)$ be a (pre-)Cartier coquasitriangular bialgebra with $\Rr$ and $\chi$ of finite type. Then, $(H^\circ,\Rr,\chi)$ is a (pre-)Cartier quasitriangular bialgebra. In particular, if $H$ is finite-dimensional, the result holds for $H^*=H^\circ$.
 %Moreover, if $(H,r,\chi)$ is a co-Cartier bialgebra, then $(H^\circ,r, \chi)$ is a Cartier bialgebra. 
\end{proposition}

\begin{proof}
If $(H,\Rr,\chi)$ is a pre-Cartier coquasitriangular bialgebra with infinitesimal $\Rr$-form $\chi:H\otimes H\to\Bbbk$ of finite type, then we have that $\chi=\sum_{i=1}^{n}{p_i\otimes q_i}\in H^{*}\otimes H^{*}$, where $n$ is as small as possible such that $\{p_1,...,p_n\}$ and $\{q_1,...,q_n\}$ are independent. Then we get that $\chi_l(h)=\sum_{i=1}^{n}{p_i(h)q_i}$ and $\chi_r(h)=\sum_{i=1}^{n}{q_i(h)p_i}$ for all $h\in H$. If $h\in\mathrm{ker}(\chi_l)$ then $p_{i}(h)=0$ for every $i=1,...,n$ by linear independence and then $\mathrm{ker}(\chi_l)\subseteq\mathrm{ker}(p_{i})$ for every $i=1,...,n$ and $H/\mathrm{ker}(\chi_l)\cong\mathrm{Im}(\chi_l)$ which is of finite-dimension since $\chi$ is of finite type and then the kernel of every $p_i$ contains a cofinite ideal, thus $p_{i}\in H^{\circ}$ for every $i=1,...,n$. Similarly, the kernel of every $q_i$ contains $\mathrm{ker}(\chi_r)$ which is cofinite and then also $q_i\in H^{\circ}$ for every $i=1,...,n$. Then, $\chi\in H^{\circ}\otimes H^{\circ}$ and we write $\chi=\chi^i\otimes\chi_i$. Now, knowing that \eqref{eq:CC1}, \eqref{eq:CC2}, \eqref{eq:CC3} hold for $\chi:H\otimes H\to\Bbbk $, it is straightforward to prove that $\chi=\chi^i\otimes\chi_i\in H^\circ\otimes H^\circ$ satisfies \eqref{cqtr1}, \eqref{cqtr2}, \eqref{cqtr3}. 
\end{proof}

\section{Related constructions}\label{Sec:constructions}

\subsection{Infinitesimal FRT construction}\label{Sec:FRT}

We first recall the well-known FRT construction \cite{FRT} following \cite[Section VIII.6]{Kassel}, where a coquasitriangular bialgebra $A(c)$ is obtained from a solution $c\in\mathrm{End}_\Bbbk(V\otimes V)$ of the \emph{braid equation} %\textcolor{red}{(I would call the following QYB following \cite[Section VIII.1]{Kassel}) and omit \eqref{eq:qyb}.}
\begin{equation}\label{eq:braid}
c_{12}c_{23}c_{12}=c_{23}c_{12}c_{23}
\end{equation}
on a finite-dimensional vector space $V$. Note that \eqref{eq:braid} is equivalent to the \emph{quantum Yang-Baxter equation}
\begin{equation}\label{eq:qyb}
c_{12}c_{13}c_{23}=c_{23}c_{13}c_{12}
\end{equation}
(see \cite[Definition 12.1.1]{Radford}). Explicitly, $c$ is a solution of \eqref{eq:braid} if and only if $\tau_{V,V}\circ c$ satisfies \eqref{eq:qyb}. Afterwards, this procedure is
extended to obtain infinitesimal braidings on the
quotient $A(c,t)=A(c)/I(t)$ from another endomorphism 
$t\in\mathrm{End}_\Bbbk(V\otimes V)$ satisfying a series of relations
(see Definition \ref{def:infbrvs}). At the end of this section we give a canonical solution
to those equations and describe the resulting non-trivial infinitesimal
braidings on all FRT bialgebras $A(c)$.

\begin{definition}\label{def:brvec}
 Recall that a \textbf{braided vector space} $(V,c)$ is a $\Bbbk$-vector space $V$  endowed with an invertible solution $c\in\mathrm{End}_\Bbbk(V\otimes V)$ of the braid equation \eqref{eq:braid}, see e.g. \cite[Definition 1.1]{AS02}.
\end{definition}

Fix a finite-dimensional braided vector space $(V,c)$. Let $\mathrm{dim}(V)=N$, choose a basis $\{v_i\}_{1\leq i\leq N}$ of $V$ and denote the coefficients of $c$ by $c_{ij}^{k\ell}\in\Bbbk$,
i.e., $c(v_i\otimes v_j)=c_{ij}^{k\ell}v_k\otimes v_\ell$. \footnote{
Here and in the following we employ Einstein sum convention: we sum over
repeated indices, where the indices run from $1$ to $N$.}
We adopt the compact notation
$v_{i_1i_2\cdots i_n}:=v_{i_1}\otimes v_{i_2}\otimes\cdots\otimes v_{i_n}$ so that, for instance, the previous equality rewrites as $c(v_{ij})=c_{ij}^{k\ell}v_{k\ell}$.
Consider the set of indeterminates $\{T_i^j\}_{1\leq i,j\leq N}$
and denote the free (unital associative) $\Bbbk$-algebra generated by
these indeterminates by $F:=\Bbbk\{T_i^j\}_{1\leq i,j\leq N}$.
Then the quotient of $F$ by its two-sided ideal $I(c)$ generated by
$$
C_{ij}^{k\ell}:=c_{ij}^{mn}T_m^kT_n^\ell-T_i^mT_j^nc_{mn}^{k\ell},\qquad\qquad 1\leq i,j,k,\ell\leq N
$$
is denoted by $A(c):=F/I(c)$. On $F$ there is a canonical bialgebra
structure with coproduct and counit determined on generators by
$$
\Delta(T_i^j):=T_i^k\otimes T_k^j,\qquad\qquad
\varepsilon(T_i^j):=\delta_i^j
$$
and then extended as algebra morphisms. As one easily verifies 
$\varepsilon(I(c))=0$ and $\Delta(C_{ij}^{k\ell})=C_{ij}^{mn}\otimes T_m^kT_n^\ell+T_i^mT_j^n\otimes C_{mn}^{k\ell}\in I(c)\otimes F+F\otimes I(c)$ and thus the bialgebra structure of $F$ descends to $A(c)$.
The bialgebra $A(c)$ coacts from the left on $V$ via
${}_V\Delta\colon V\to A(c)\otimes V$, ${}_V\Delta(v_i):=T_i^j\otimes v_j$ and $c$ is left $A(c)$-colinear with respect to the diagonal
coaction given by ${}_{V\otimes V}\Delta(v_{ij})=T_i^kT_j^l\otimes v_{kl}$. Moreover, given another bialgebra $A'$ such that
$(V,{}_V\Delta')$ is a left $A'$-comodule and $c$ is left $A'$-colinear, there is a unique
bialgebra map $f\colon A(c)\to A'$ such that ${}_V\Delta'=(f\otimes\mathrm{Id})\circ{}_V\Delta$. It is shown in
\cite[Theorem VIII.6.1]{Kassel} that $A(c)$ is the unique bialgebra
(up to isomorphism) with this universal property.
Up to now, no assumptions on $c$ had to be made. As shown in the
following result (see e.g. \cite[Theorem VIII.6.4]{Kassel}) the braid equation ensures that
$A(c)$ is coquasitriangular.
\begin{proposition}[FRT]\label{Prop:FRT}
Let $(V,c)$  be a finite-dimensional braided vector space. There is a convolution-invertible bialgebra bicharacter $\Rr\colon F\otimes F\to\Bbbk$
on the free algebra $F=\Bbbk\{T_i^j\}_{1\leq i,j\leq N}$ given   on generators, for all $1\leq i,j,k,\ell\leq N$, by
\begin{equation}
    \Rr(T_i^k\otimes T_j^\ell)=c_{ji}^{k\ell}.
\end{equation}
Moreover, this map descends to a coquasitriangular structure  $\Rr$  on $A(c)$
such that $\sigma^\Rr_{V,V}=c$. 
\end{proposition}
Recall that the convolution inverse $\Rr^{-1}\colon A(c)\otimes A(c)
\to\Bbbk$ of $\Rr$ is determined on generators by
$\Rr^{-1}(T_i^k\otimes T_j^\ell)=\overline{c}_{ij}^{\ell k}$,
where $\overline{c}_{ij}^{\ell k}\in\Bbbk$ are the coefficients of the
inverse $c^{-1}\in\mathrm{End}_\Bbbk(V\otimes V)$ of $c$, i.e.,
$c^{-1}(v_{ij})=\overline{c}_{ij}^{\ell k}v_{\ell k}$.\medskip

In the following, given our fixed braided vector space $(V,c)$, we denote the corresponding coquasitriangular structure on $A(c)$ by
$\Rr\colon A(c)\otimes A(c)\to\Bbbk$.
Consider another endomorphism $t\in\mathrm{End}_\Bbbk(V\otimes V)$,
of arbitrary nature for the moment, with coefficients $t_{ij}^{k\ell}\in\Bbbk$ determined by 
$t(v_{ij})=t_{ij}^{k\ell}v_{k \ell}$. Let $I(t)$ be the two-sided
ideal in $A(c)$ generated by
$$
D_{ij}^{k\ell}:=t_{ij}^{mn}T_m^kT_n^\ell-T_i^mT_j^n t_{mn}^{k\ell},
\qquad\qquad 1\leq i,j,k,\ell\leq N
$$
and denote the quotient algebra by $A(c,t)=A(c)/I(t)$.
\begin{lemma}\label{Lem:FRT1}
$A(c,t)$ is a bialgebra. Moreover, the coquasitriangular structure $\Rr$ on $A(c)$ descends
to a coquasitriangular structure of $A(c,t)$ if and only if
\begin{align}
    c_{12}c_{23}t_{12}&=t_{23}c_{12}c_{23}\label{eq:t1}
   \qquad \text{or, equivalently,}\qquad c_{23}t_{12}c_{23}^{-1}=c_{12}^{-1}t_{23}c_{12}\\ 
        t_{12}c_{23}c_{12}&=c_{23}c_{12}t_{23}\label{eq:t2}
        \qquad \text{or, equivalently,}\qquad c_{23}^{-1}t_{12}c_{23}=c_{12}t_{23}c_{12}^{-1}
\end{align}
hold as equations $V^{\otimes 3}\to V^{\otimes 3}$.
\end{lemma}
\begin{proof}
We show that $\Delta(I(t))\subseteq I(t)\otimes A(c)+A(c)\otimes I(t)$.
For this, it is sufficient to prove
\begin{align*}
    \Delta(D_{ij}^{k\ell})
    &=t_{ij}^{mn}\Delta(T_m^k)\Delta(T_n^\ell)
    -\Delta(T_i^m)\Delta(T_j^n)t_{mn}^{k\ell}\\
    &=t_{ij}^{mn}T_m^pT_n^q\otimes T_p^kT_q^\ell
    -T_i^pT_j^q\otimes T_p^mT_q^nt_{mn}^{k\ell}\\
    &=D_{ij}^{pq}\otimes T_p^kT_q^\ell
    +T_i^mT_j^nt_{mn}^{pq}\otimes T_p^kT_q^\ell
    +T_i^pT_j^q\otimes D_{pq}^{kl}
    -T_i^pT_j^q\otimes t_{pq}^{mn}T_m^kT_n^\ell\\
    &=D_{ij}^{pq}\otimes T_p^kT_q^\ell
    +T_i^pT_j^q\otimes D_{pq}^{kl}
    \in I(t)\otimes A(c)+A(c)\otimes I(t)
\end{align*}
since $\Delta$ is an algebra morphism.
Clearly $\varepsilon(I(t))=0$ and thus $A(c,t)$ is a bialgebra.

Next, observe that by \eqref{eq:ct2}
\begin{align*}
    \Rr(T_i^j\otimes D_{mn}^{k\ell})
    &=t_{mn}^{pq}\Rr(T_i^j\otimes T_p^kT_q^\ell)
    -\Rr(T_i^j\otimes T_m^pT_n^q)t_{pq}^{k\ell}\\
    &=t_{mn}^{pq}\Rr(T_i^s\otimes T_q^\ell)\Rr(T_s^j\otimes T_p^k)
    -\Rr(T_i^s\otimes T_n^q)\Rr(T_s^j\otimes T_m^p)t_{pq}^{k\ell}\\
    &=t_{mn}^{pq}c_{qi}^{s\ell}c_{ps}^{jk}
    -c_{ni}^{sq}c_{ms}^{jp}t_{pq}^{k\ell}
\end{align*}
for all indices $1\leq i,j,k,\ell,m,n\leq N$. We show that the above
vanishes if and only if \eqref{eq:t1} holds. This is the case, since,
fixing $1\leq i,m,n\leq N$,
\begin{align*}
    v_{mni}
    &\xmapsto{t_{12}}t_{mn}^{pq}v_{pqi}
    \xmapsto{c_{23}}t_{mn}^{pq}c_{qi}^{s\ell}v_{ps\ell}
    \xmapsto{c_{12}}t_{mn}^{pq}c_{qi}^{s\ell}c_{ps}^{jk}v_{jk\ell}
\end{align*}
equals
\begin{align*}
    v_{m n i}
    &\xmapsto{c_{23}}c_{ni}^{sq}v_{msq}
    \xmapsto{c_{12}}c_{ni}^{sq}c_{ms}^{jp}v_{jpq}
    \xmapsto{t_{23}}c_{ni}^{sq}c_{ms}^{jp}t_{pq}^{k\ell}v_{jk\ell}
\end{align*}
if and only if $t_{mn}^{pq}c_{qi}^{s\ell}c_{ps}^{jk}
=c_{ni}^{sq}c_{ms}^{jp}t_{pq}^{k\ell}$ for all $1\leq j,k,\ell\leq N$.
By \eqref{eq:ct2},\eqref{eq:ct3} and using that $\Delta(I(t))\subseteq I(t)\otimes A(c)+A(c)\otimes I(t)$ we observe that
$\Rr(T_i^j\otimes D_{mn}^{k\ell})=0$ for all $1\leq i,j,k,\ell,m,n\leq N$
is necessary and sufficient to the condition $\Rr(A(c)\otimes I(t))=0$.
Thus, $\Rr(A(c)\otimes I(t))=0$ if and only if \eqref{eq:t1} holds.
Similarly, one proves that $\Rr(I(t)\otimes A(c))=0$ if and only if \eqref{eq:t2} holds.
\begin{invisible}
\begin{align*}
    \Rr(D_{ij}^{k\ell}\otimes T_m^n)
    &=t_{ij}^{pq}\Rr(T_p^kT_q^\ell\otimes T_m^n)
    -\Rr(T_i^pT_j^q\otimes T_m^n)t_{pq}^{k\ell}\\
    &\overset{\eqref{eq:ct3}}{=}t_{ij}^{pq}\Rr(T_p^k\otimes T_m^s)\Rr(T_q^\ell\otimes T_s^n)
    -\Rr(T_i^p\otimes T_m^s)\Rr(T_j^q\otimes T_s^n)t_{pq}^{k\ell}\\
    &=t_{ij}^{pq}c_{mp}^{ks}c_{sq}^{\ell n}
    -c_{mi}^{ps}c_{sj}^{qn}t_{pq}^{k\ell}
\end{align*}
vanishes if and only if 
\begin{align*}
    v_m\otimes v_i\otimes v_j
    &\xmapsto{t_{23}}t_{ij}^{pq}v_m\otimes v_p\otimes v_q
    \xmapsto{c_{12}}t_{ij}^{pq}c_{mp}^{ks}v_k\otimes v_s\otimes v_q
    \xmapsto{c_{23}}t_{ij}^{pq}c_{mp}^{ks}c_{sq}^{\ell n}v_k\otimes v_\ell\otimes v_n
\end{align*}
equals
\begin{align*}
    v_m\otimes v_i\otimes v_j
    &\xmapsto{c_{12}}c_{mi}^{ps}v_p\otimes v_s\otimes v_j
    \xmapsto{c_{23}}c_{mi}^{ps}c_{sj}^{qn}v_p\otimes v_q\otimes v_n
    \xmapsto{t_{12}}c_{mi}^{ps}c_{sj}^{qn}t_{pq}^{k\ell}v_k\otimes v_\ell\otimes v_n.
\end{align*}
As before, this condition is sufficient and necessary for
$\Rr(I(t)\otimes A(c))=0$.
\end{invisible}
This means that $\Rr$ descends to a map $\Rr\colon A(c,t)\otimes A(c,t)
\to\Bbbk$ (denoted by the same symbol by abuse of notation) if and only
if \eqref{eq:t1} and \eqref{eq:t2} are satisfied. Since the bialgebra
structure of $A(c,t)$ is induced from $A(c)$ it is immediate that
$\Rr\colon A(c,t)\otimes A(c,t)\to\Bbbk$ satisfies the hexagon
equations \eqref{eq:ct2},\eqref{eq:ct3} and it is quasi-commutative, i.e.,
\eqref{eq:ct1} is satisfied. It remains to prove that the convolution inverse
$\Rr^{-1}\colon A(c)\otimes A(c)\to\Bbbk$ also descends to the
quotient $A(c,t)$. Repeating the above computations with
$\Rr^{-1}$ instead of $\Rr$ this is the case if and only if
\begin{equation*}
    t_{12}c^{-1}_{23}c^{-1}_{12}=c^{-1}_{23}c^{-1}_{12}t_{23}\qquad \text{and}\qquad
    c^{-1}_{12}c^{-1}_{23}t_{12}=t_{23}c^{-1}_{12}c^{-1}_{23}
\end{equation*}
hold as equations $V^{\otimes 3}\to V^{\otimes 3}$; those equations
being clearly equivalent to \eqref{eq:t1} and \eqref{eq:t2}.
The induced map $\Rr^{-1}\colon A(c,t)\otimes A(c,t)\to\Bbbk$ is
then the convolution inverse of $\Rr$, which is sufficient to be
checked on generators; an immediate consequence.
\end{proof}

Let $(V,c)$ be a braided vector space and $t\in\mathrm{End}_\Bbbk(V\otimes V)$ an arbitrary endomorphism with coefficients $c_{ij}^{k\ell},t_{ij}^{k\ell}\in\Bbbk$ as before.
Consider for a moment the convolution-invertible bialgebra bicharacter $\Rr\colon F\otimes F\to\Bbbk$
on the free algebra $F=\Bbbk\{T_i^j\}_{1\leq i,j\leq N}$ of Proposition~\ref{Prop:FRT} 
(we know that it descends to a coquasitriangular structure on $A(c)$).
Note that $F$ is the tensor algebra $TC$ of the matrix coalgebra $C$ spanned by the $T_i^j$'s.
By Lemma \ref{lem:bitriang}, the free algebra $F$ together with its convolution-invertible bialgebra bicharacter $\Rr\colon F\otimes F\to\Bbbk$, allows two triangle actions 
$\triangleright$  and $\triangleleft$ that fulfil the requirement of Proposition \ref{prop:bitriang}.
Thus the map $\chi_{11}:C\otimes C\to \Bbbk$, defined on generators by the assignment $T_i^k\otimes T_j^\ell\mapsto t_{ij}^{k\ell}$, induces  a $\Bbbk$-linear map $\chi\colon F\otimes F\to\Bbbk$ satisfying \eqref{eq:CC2'} and \eqref{eq:CC3'}. Explicitly, it is defined
inductively by the $\Bbbk$-linear extension of
\begin{equation}
    \chi(T_i^k\otimes T_j^\ell)=t_{ij}^{k\ell},\qquad\qquad
    \chi(1\otimes T_i^j)=0=\chi(T_i^j\otimes 1),
\end{equation}
\begin{equation}\label{eq:X2}
    \chi(T_i^j\otimes T_k^mT_\ell^n)
    =\chi(T_i^j\otimes T_k^m)\varepsilon(T_\ell^n)
    +(\Rr^{-1}_{12}*\chi_{13}*\Rr_{12})(T_i^j\otimes T_k^m\otimes T_\ell^n)
\end{equation}
and
\begin{equation}\label{eq:X3}
    \chi(T_i^kT_j^\ell\otimes T_m^n)
    =\varepsilon(T_i^k)\chi(T_j^\ell\otimes T_m^n)
    +(\Rr^{-1}_{23}*\chi_{13}*\Rr_{23})(T_i^k\otimes T_j^\ell\otimes T_m^n)
\end{equation}
for all indices $1\leq i,j,k,\ell,m,n\leq N$.
In terms of coefficients \eqref{eq:X2} and \eqref{eq:X3} read
\begin{equation}\label{eq:X4}
\begin{split}
    \chi(T_i^j\otimes T_k^mT_\ell^n)
    &=t_{ik}^{jm}\delta_\ell^n
    +\Rr^{-1}(T_i^{j_1}\otimes T_k^{m_1})\chi(T_{j_1}^{j_2}\otimes T_\ell^n)\Rr(T_{j_2}^j\otimes T_{m_1}^m)\\
    &=t_{ik}^{jm}\delta_\ell^n
    +\overline{c}_{ik}^{m_1j_1}t_{j_1\ell}^{j_2n}c_{m_1j_2}^{jm}
\end{split}
\end{equation}
and
\begin{equation}\label{eq:X5}
\begin{split}
    \chi(T_i^kT_j^\ell\otimes T_m^n)
    &=\delta_i^kt_{jm}^{\ell n}
    +\Rr^{-1}(T_j^{\ell_1}\otimes T_m^{n_1})\chi(T_i^k\otimes T_{n_1}^{n_2})\Rr(T_{\ell_1}^\ell\otimes T_{n_2}^n)\\
    &=\delta_i^kt_{jm}^{\ell n}
    +\overline{c}_{jm}^{n_1\ell_1}t_{in_1}^{kn_2}c_{n_2\ell_1}^{\ell n}.
\end{split}
\end{equation}

\begin{lemma}\label{Lem:FRT2}
Let $(V,c)$ be a braided vector space  and $t\in\mathrm{End}_\Bbbk(V\otimes V)$ arbitrary. Then $\chi\colon F\otimes F\to\Bbbk$ descends to
a map $A(c)\otimes A(c)\to\Bbbk$ if and only if
\begin{align}
    t_{12}c_{23}+c_{12}t_{23}c^{-1}_{12}c_{23}
    &=c_{23}t_{12}+c_{23}c_{12}t_{23}c^{-1}_{12}\label{eq:t3'}\\
    t_{23}c_{12}+c_{23}t_{12}c^{-1}_{23}c_{12}
    &=c_{12}t_{23}+c_{12}c_{23}t_{12}c^{-1}_{23}\label{eq:t4'}
\end{align}
hold as equations $V^{\otimes 3}\to V^{\otimes 3}$.
If \eqref{eq:t1},\eqref{eq:t2} are satisfied the above equations
are together equivalent to the equation
\begin{equation}
    c_{23}t_{12}c^{-1}_{23}
    =c_{12}t_{23}c_{12}^{-1}    \label{eq:t3}
\end{equation}
of linear maps $V^{\otimes 3}\to V^{\otimes 3}$.
\end{lemma}
\begin{proof}
Let $t\in\mathrm{End}_\Bbbk(V\otimes V)$ be an endomorphism and
$c\in\mathrm{End}_\Bbbk(V\otimes V)$ an invertible solution of the
braid equation.
For all indices $1\leq i,j,k,\ell,m,n\leq N$ we have
\begin{align*}
    \chi(T_i^j\otimes C_{k\ell}^{mn})
    &=c_{k\ell}^{pq}\chi(T_i^j\otimes T_p^mT_q^n)
    -\chi(T_i^j\otimes T_k^pT_\ell^q)c_{pq}^{mn}\\
    &\overset{\eqref{eq:X4}}{=}c_{k\ell}^{pq}t_{ip}^{jm}\delta_q^n
    +c_{k\ell}^{pq}\overline{c}_{ip}^{m_1j_1}t_{j_1q}^{j_2n}c_{m_1j_2}^{jm}
    -t_{ik}^{jp}\delta_\ell^qc_{pq}^{mn}
    -\overline{c}_{ik}^{p_1j_1}t_{j_1\ell}^{j_2q}c_{p_1j_2}^{jp}c_{pq}^{mn}.
\end{align*}
Using the fact that $\Delta(I(c))\subseteq I(c)\otimes F+F\otimes I(c)$
and \eqref{eq:X2},\eqref{eq:X3} it follows that for the condition
$\chi(F\otimes I(c))=0$ to hold it is
necessary and sufficient to prove $\chi(T_i^j\otimes C_{k\ell}^{mn})=0$.
The latter is the case if and only if \eqref{eq:t3'} holds.
This can be shown on a basis, fixing indices $1\leq i,k,\ell\leq N$.
Namely, the sum of
\begin{align*}
    v_{ik\ell}
    \xmapsto{c_{23}}c_{k\ell}^{pn}v_{ipn}
    \xmapsto{t_{12}}c_{k\ell}^{pn}t_{ip}^{jm}v_{jmn}
\end{align*}
and
\begin{align*}
    v_{ik\ell}
    &\xmapsto{c_{23}}c_{k\ell}^{pq}v_{ipq}
    \xmapsto{c^{-1}_{12}}c_{k\ell}^{pq}\overline{c}_{ip}^{m_1j_1}
    v_{m_1j_1q}
    \xmapsto{t_{23}}c_{k\ell}^{pq}\overline{c}_{ip}^{m_1j_1}t_{j_1q}^{j_2n}v_{m_1j_2n}
    \xmapsto{c_{12}}c_{k\ell}^{pq}\overline{c}_{ip}^{m_1j_1}t_{j_1q}^{j_2n}c_{m_1j_2}^{jm}v_{jmn}
\end{align*}
equals the sum of
\begin{align*}
    v_{ik\ell}
    \xmapsto{t_{12}}t_{ik}^{jp}v_{jp\ell}
    \xmapsto{c_{23}}t_{ik}^{jp}c_{p\ell}^{mn} v_{jmn}
\end{align*}
and
\begin{align*}
    v_{ik\ell}
    &\xmapsto{c^{-1}_{12}}\overline{c}_{ik}^{p_1j_1}v_{p_1j_1\ell}
    \xmapsto{t_{23}}\overline{c}_{ik}^{p_1j_1}t_{j_1\ell}^{j_2q}
    v_{p_1j_2q}
    \xmapsto{c_{12}}\overline{c}_{ik}^{p_1j_1}t_{j_1\ell}^{j_2q}c_{p_1j_2}^{jp}v_{jpq}
    \xmapsto{c_{23}}\overline{c}_{ik}^{p_1j_1}t_{j_1\ell}^{j_2q}c_{p_1j_2}^{jp}c_{pq}^{mn}v_{jmn}
\end{align*}
if and only if $\chi(T_i^j\otimes C_{k\ell}^{mn})=0$ for all $1\leq j,m,n\leq N$. This shows the equivalence of $\chi(F\otimes I(c))=0$
with \eqref{eq:t3'}.
Similarly, one proves that $\chi(I(c)\otimes F)=0$ if and only if
\eqref{eq:t4'} holds.
\begin{invisible}
We have
\begin{align*}
    \chi(C_{ij}^{k\ell}\otimes T_m^n)
    &=c_{ij}^{pq}\chi(T_p^kT_q^\ell\otimes T_m^n)
    -\chi(T_i^pT_j^q\otimes T_m^n)c_{pq}^{k\ell}\\
    &\overset{\eqref{eq:X5}}{=}c_{ij}^{pq}\delta_p^kt_{qm}^{\ell n}
    +c_{ij}^{pq}\overline{c}_{qm}^{n_1\ell_1}t_{pn_1}^{kn_2}c_{n_2\ell_1}^{\ell n}
    -\delta_i^pt_{jm}^{qn}c_{pq}^{k\ell}
    -\overline{c}_{jm}^{n_1q_1}t_{in_1}^{pn_2}c_{n_2q_1}^{qn}c_{pq}^{k\ell}
\end{align*}
which vanishes for all $1\leq i,j,k,\ell,m,n\leq N$ if and only if
the sum of
\begin{align*}
    v_i\otimes v_j\otimes v_m
    \xmapsto{c_{12}}c_{ij}^{kq}v_k\otimes v_q\otimes v_m
    \xmapsto{t_{23}}c_{ij}^{kq}t_{qm}^{\ell n}v_k\otimes v_\ell\otimes v_n
\end{align*}
and
\begin{align*}
    v_i\otimes v_j\otimes v_m
    &\xmapsto{c_{12}}c_{ij}^{pq}v_p\otimes v_q\otimes v_m
    \xmapsto{c^{-1}_{23}}c_{ij}^{pq}\overline{c}_{qm}^{n_1\ell_1}
    v_p\otimes v_{n_1}\otimes v_{\ell_1}
    \xmapsto{t_{12}}c_{ij}^{pq}\overline{c}_{qm}^{n_1\ell_1}t_{pn_1}^{kn_2}v_k\otimes v_{n_2}\otimes v_{\ell_1}\\
    &\xmapsto{c_{23}}c_{ij}^{pq}\overline{c}_{qm}^{n_1\ell_1}t_{pn_1}^{kn_2}c_{n_2\ell_1}^{\ell n}v_k\otimes v_\ell\otimes v_n
\end{align*}
equals the sum of
\begin{align*}
    v_i\otimes v_j\otimes v_m
    \xmapsto{t_{23}}t_{jm}^{qn}v_i\otimes v_q\otimes v_n
    \xmapsto{c_{12}}t_{jm}^{qn}c_{iq}^{k\ell}v_k\otimes v_\ell\otimes v_n
\end{align*}
and
\begin{align*}
    v_i\otimes v_j\otimes v_m
    &\xmapsto{c^{-1}_{23}}\overline{c}_{jm}^{n_1q_1}v_i\otimes v_{n_1}\otimes v_{q_1}
    \xmapsto{t_{12}}\overline{c}_{jm}^{n_1q_1}t_{in_1}^{pn_2}v_p\otimes v_{n_2}\otimes v_{q_1}
    \xmapsto{c_{23}}\overline{c}_{jm}^{n_1q_1}t_{in_1}^{pn_2}c_{n_2q_1}^{qn}v_p\otimes v_q\otimes v_n\\
    &\xmapsto{c_{12}}\overline{c}_{jm}^{n_1q_1}t_{in_1}^{pn_2}c_{n_2q_1}^{qn}c_{pq}^{k\ell}v_k\otimes v_\ell\otimes v_n
\end{align*}
for all $1\leq i,j,k,\ell,m,n\leq N$.
\end{invisible}

Assume now that $c$ and $t$ satisfy the equations \eqref{eq:t1}
and \eqref{eq:t2} (in addition to $c$ satisfying the braid equation).
Since $c$ is invertible \eqref{eq:t1} is equivalent to 
 $c_{12}c_{23}t_{12}c_{23}^{-1}=t_{23}c_{12}$. 
Thus, \eqref{eq:t4'} holds if and only if 
$c_{23}t_{12}c^{-1}_{23}c_{12}
    =c_{12}t_{23}$, which is equivalent to \eqref{eq:t3}. Similarly, \eqref{eq:t2} is 
equivalent to $t_{12}c_{23}=c_{23}c_{12}t_{23}c^{-1}_{12}$. 
Thus, \eqref{eq:t3'} is satisfied if and only if so is  $c_{12}t_{23}c^{-1}_{12}c_{23}=c_{23}t_{12}$ if and only if (again) \eqref{eq:t3} holds.
\end{proof}

We now introduce a new  notion which will play a central role in  Theorem \ref{thm:curvedFRT}.

\begin{definition}\label{def:infbrvs}
An \textbf{infinitesimally braided vector space} $(V,c,t)$ is a braided vector space $(V,c)$, see Definition \ref{def:brvec}, together with an endomorphism
$t\in\mathrm{End}_\Bbbk(V\otimes V)$ such that
\begin{equation}\label{eq:ibv1}
c_{23}t_{12}c_{23}^{-1}=c_{12}^{-1}t_{23}c_{12}=c_{23}^{-1}t_{12}c_{23}=c_{12}t_{23}c_{12}^{-1}    
\end{equation}
%\begin{align*}
%\begin{split}
%    c_{12}c_{23}t_{12}&=t_{23}c_{12}c_{23},\\
 %   t_{12}c_{23}c_{12}&=c_{23}c_{12}t_{23},
%\end{split}
%\begin{split}
%    c_{12}t_{23}c^{-1}_{12}
%    &=c^{-1}_{12}t_{23}c_{12},\\
%    c_{23}t_{12}c^{-1}_{23}
%    &=c^{-1}_{23}t_{12}c_{23},
%\end{split}
%\end{align*}
\begin{align}
    t_{12}t_{23}+c_{12}t_{23}c^{-1}_{12}t_{23}
    &=t_{23}t_{12}+t_{23}c_{12}t_{23}c^{-1}_{12}\label{eq:t5},\quad\text{i.e.,}\quad [t_{12}+c_{12}t_{23}c^{-1}_{12},t_{23}]=0,\\
    t_{23}t_{12}+c_{23}t_{12}c^{-1}_{23}t_{12}
    &=t_{12}t_{23}+t_{12}c_{23}t_{12}c^{-1}_{23}\label{eq:t6},\quad\text{i.e.,}\quad [t_{23}+c_{23}t_{12}c^{-1}_{23},t_{12}]=0,
\end{align}
are equalities of linear maps $V^{\otimes 3}\to V^{\otimes 3}$.    
\end{definition}

    Note that we can actually combine \eqref{eq:t5} and \eqref{eq:t6} in a unique formula, namely
  \begin{equation}\label{eqt5&6}
   [t_{23},c_{12}t_{23}c^{-1}_{12}]=[t_{12},t_{23}]=[c_{23}t_{12}c^{-1}_{23},t_{12}].
  \end{equation}

\begin{example}\label{exa:infbraid} Any vector space becomes a braided vector space $(V,\tau)$ through the canonical flip $\tau$. Now the terms in \eqref{eq:ibv1} are automatically all equal to $t_{13}$ while \eqref{eq:t5} and \eqref{eq:t6} become 
\[[t_{12}+t_{13},t_{23}]=0=[t_{23}+t_{13},t_{12}].\]
These are \textbf{infinitesimal pure braid relations} in the sense of \cite[Section 1]{kohno}. %, see also \cite[XIX.2]{Kassel}.
\end{example}

\begin{example}\label{ex:canT}
    Any braided vector space becomes an infinitesimally braided vector space $(V,c,t)$ by taking $t=\lambda\cdot\mathrm{Id}_{V\ot V}$ for some $\lambda\in\Bbbk$. In particular $t=0$ and $t=\mathrm{Id}_{V\ot V}$ are solutions.
\end{example}

In the following result we characterise the possible $t$'s on an infinitesimally
braided vector space of diagonal type.

\begin{proposition}\label{prop:diagt}
Consider a braided vector space $(V,c)$ of diagonal type, i.e., such that $%
c(v_{i}\otimes v_{j})=q_{ij}v_{j}\otimes v_{i}$ for some $q_{ij}\in \Bbbk
\setminus \{0\}$.  Set $t(v_{i}\otimes
v_{j})=t_{ij}^{kl}v_{k}\otimes v_{l}$. Then $(V,c,t)$ is an infinitesimally
braided vector space if and only if the following equalities hold true for all possible indexes.%
\begin{gather}
q_{kj}^{-1}q_{bj}t_{ik}^{ab}
=q_{ij}q_{aj}^{-1}t_{ik}^{ab}=q_{jk}q_{jb}^{-1}t_{ik}^{ab}=q_{ji}^{-1}q_{ja}t_{ik}^{ab}, \label{eq:diagt1}
\\ 
q_{ji}^{-1}q_{ja}t_{ik}^{ab}t_{jb}^{uv}-q_{ui}^{-1}q_{ua}t_{ib}^{av}t_{jk}^{ub} =t_{jk}^{bv}t_{ib}^{au}-t_{ij}^{ab}t_{bk}^{uv}=q_{vu}q_{ku}^{-1}t_{bk}^{av}t_{ij}^{bu}-q_{vj}q_{kj}^{-1}t_{ik}^{bv}t_{bj}^{au}. \label{eq:diagt2}
\end{gather}
\end{proposition}

\begin{proof} Note that $c^{-1}(v_{i}\otimes
v_{j})=q_{ji}^{-1}v_{j}\otimes v_{i}$.
Set $v_{ijk}:=v_{i}\otimes v_{j}\otimes v_{k}.$ Then

\begin{itemize}
\item $%
c_{23}t_{12}c_{23}^{-1}(v_{ijk})=q_{kj}^{-1}c_{23}t_{12}(v_{ikj})=q_{kj}^{-1}t_{ik}^{ab}c_{23}(v_{abj})=q_{kj}^{-1}t_{ik}^{ab}q_{bj}v_{ajb};
$

\item $%
c_{12}^{-1}t_{23}c_{12}(v_{ijk})=q_{ij}c_{12}^{-1}t_{23}(v_{jik})=q_{ij}t_{ik}^{ab}c_{12}^{-1}(v_{jab})=q_{ij}t_{ik}^{ab}q_{aj}^{-1}v_{ajb};
$

\item $%
c_{23}^{-1}t_{12}c_{23}(v_{ijk})=q_{jk}c_{23}^{-1}t_{12}(v_{ikj})=q_{jk}t_{ik}^{ab}c_{23}^{-1}(v_{abj})=q_{jk}t_{ik}^{ab}q_{jb}^{-1}v_{ajb};
$

\item $%
c_{12}t_{23}c_{12}^{-1}(v_{ijk})=q_{ji}^{-1}c_{12}t_{23}(v_{jik})=q_{ji}^{-1}t_{ik}^{ab}c_{12}(v_{jab})=q_{ji}^{-1}t_{ik}^{ab}q_{ja}v_{ajb}.
$
\end{itemize}

Thus, \eqref{eq:ibv1} is equivalent to \eqref{eq:diagt1}.
Moreover

\begin{itemize}
\item $t_{23}\left( c_{12}t_{23}c_{12}^{-1}\right) (v_{ijk})=t_{23}\left(
q_{ji}^{-1}t_{ik}^{ab}q_{ja}v_{ajb}\right)
=q_{ji}^{-1}t_{ik}^{ab}q_{ja}t_{jb}^{uv}v_{auv};$

\item $\left( c_{12}t_{23}c_{12}^{-1}\right)
t_{23}(v_{ijk})=t_{jk}^{ub}\left( c_{12}t_{23}c_{12}^{-1}\right)
(v_{iub})=t_{jk}^{ub}q_{ui}^{-1}t_{ib}^{av}q_{ua}v_{auv};$

\item $%
t_{12}t_{23}(v_{ijk})=t_{jk}^{bv}t_{12}(v_{ibv})=t_{jk}^{bv}t_{ib}^{au}v_{auv};
$

\item $%
t_{23}t_{12}(v_{ijk})=t_{ij}^{ab}t_{23}(v_{abk})=t_{ij}^{ab}t_{bk}^{uv}v_{auv};
$

\item $\left( c_{23}t_{12}c_{23}^{-1}\right)
t_{12}(v_{ijk})=t_{ij}^{bu}\left( c_{23}t_{12}c_{23}^{-1}\right)
(v_{buk})=t_{ij}^{bu}q_{ku}^{-1}t_{bk}^{av}q_{vu}v_{auv};$

\item $t_{12}\left( c_{23}t_{12}c_{23}^{-1}\right)
(v_{ijk})=q_{kj}^{-1}t_{ik}^{bv}q_{vj}t_{12}\left( v_{bjv}\right)
=q_{kj}^{-1}t_{ik}^{bv}q_{vj}t_{bj}^{au}v_{auv}.$
\end{itemize}

Thus, \eqref{eqt5&6} is equivalent to \eqref{eq:diagt2}.
\end{proof}

\begin{example}\label{ex:diag}
 Consider a braided vector space $(V,c)$ of diagonal type  as before for some $q_{ij}\in\Bbbk\setminus\{0\}$. Set $t(v_i\otimes v_j)=p_{ij}v_i\otimes v_j$. With notations as in Proposition \ref{prop:diagt}, we have  $t(v_{i}\otimes
v_{j})=t_{ij}^{kl}v_{k}\otimes v_{l}$ where $t_{ij}^{kl}:=p_{ij}\delta_i^k\delta_j^l$ and $\delta_i^k$ is the Kronecker delta.
 It is then easily checked that \eqref{eq:diagt1} and \eqref{eq:diagt2} hold true in this case so that $(V,c,t)$ is an infinitesimally braided vector space.  We point out that, although $c$ is not the canonical flip, still all the terms in \eqref{eq:ibv1} equal $t_{13}$ as in Example \ref{exa:infbraid} so that also in this case \eqref{eq:t5} and \eqref{eq:t6} reduce to infinitesimal pure braid relations which are easily verified.
Note that $A(c,t)$ is the quotient of the free algebra  $F=\Bbbk\{T_i^j\}_{1\leq i,j\leq N}$ by the two-sided ideal generated by the elements $q_{ij}T_j^kT_i^\ell-q_{k\ell} T_i^\ell T_j^k$ and $(p_{ij}-p_{k\ell}) T_i^kT_j^\ell$ so that $A(c,t)\neq A(c)$, in general, in this case. 
 
 \begin{invisible} Here we include the old direct proof. Set $x_{ijk}:=x_i\ot x_j\ot x_k.$ Then
 
$c_{23}t_{12}c_{23}^{-1}(x_{ijk})
=q_{kj}^{-1}c_{23}t_{12}(x_{ikj})
=q_{kj}^{-1}p_{ik}c_{23}(x_{ikj})
=q_{kj}^{-1}p_{ik}q_{kj}x_{ijk}
=p_{ik}x_{ijk}=t_{13}(x_{ijk})$

$c_{12}^{-1}t_{23}c_{12}(x_{ijk})
=q_{ij}c_{12}^{-1}t_{23}(x_{jik})
=q_{ij}p_{ik}c_{12}^{-1}(x_{jik})
=q_{ij}p_{ik}q_{ij}^{-1}x_{ijk}
=p_{ik}x_{ijk}=t_{13}(x_{ijk})$

$c_{23}^{-1}t_{12}c_{23}(x_{ijk})
=q_{jk}c_{23}^{-1}t_{12}(x_{ikj})
=q_{jk}p_{ik}c_{23}^{-1}(x_{ikj})
=q_{jk}p_{ik}q_{jk}^{-1}x_{ijk}
=p_{ik}x_{ijk}=t_{13}(x_{ijk})$

$c_{12}t_{23}c_{12}^{-1}(x_{ijk})
=q_{ji}^{-1}c_{12}t_{23}(x_{jik})
=q_{ji}^{-1}p_{ik}c_{12}(x_{jik})
=q_{ji}^{-1}p_{ik}q_{ji}x_{ijk}
=p_{ik}x_{ijk}=t_{13}(x_{ijk})$  

Thus $c_{23}t_{12}c_{23}^{-1}=c_{12}^{-1}t_{23}c_{12}=c_{23}^{-1}t_{12}c_{23}=c_{12}t_{23}c_{12}^{-1}=t_{13}$. Moreover

$[t_{23},t_{13}](x_{ijk})=t_{23}t_{13}(x_{ijk})-t_{13}t_{23}(x_{ijk})=p_{ik}p_{jk}x_{ijk}-p_{jk}p_{ik}x_{ijk}=0$,

$[t_{12},t_{23}](x_{ijk})=t_{12}t_{23}(x_{ijk})-t_{23}t_{12}(x_{ijk})=p_{ij}p_{jk}x_{ijk}-p_{ij}p_{jk}x_{ijk}=0$,

$[t_{13},t_{12}](x_{ijk})=t_{13}t_{12}(x_{ijk})-t_{12}t_{13}(x_{ijk})=p_{ij}p_{ik} x_{ijk}-p_{ik}p_{ij}x_{ijk}=0$.

Thus $[t_{23},t_{13}]=[t_{12},t_{23}]=[t_{13},t_{12}]=0$.  

\begin{comment}
$(t_{12}t_{23}+c_{12}t_{23}c^{-1}_{12}t_{23})(x_{ijk})
=t_{12}t_{23}(x_{ijk})+t_{13}t_{23}(x_{ijk})
 =p_{ij}p_{jk}x_{ijk}+p_{jk}p_{ik}x_{ijk}$
    
$(t_{23}t_{12}+t_{23}c_{12}t_{23}c^{-1}_{12})(x_{ijk})
=t_{23}t_{12}(x_{ijk})+t_{23}t_{13}(x_{ijk})
=p_{ij}p_{jk}x_{ijk}+p_{ik}p_{jk}x_{ijk}$

$(t_{23}t_{12}+c_{23}t_{12}c^{-1}_{23}t_{12})(x_{ijk})
=t_{23}t_{12}(x_{ijk})+t_{13}t_{12}(x_{ijk})
=p_{ij}p_{jk}x_{ijk}+p_{ij}p_{ik} x_{ijk}    $
    
$(t_{12}t_{23}+t_{12}c_{23}t_{12}c^{-1}_{23})(x_{ijk})
=t_{12}t_{23}(x_{ijk})+t_{12}t_{13}(x_{ijk})
=p_{jk}p_{ij} x_{ijk}+p_{ik}p_{ij}x_{ijk}$
\end{comment}

  \rd{[This example should also be related to the fact that $(V,c)$ can be regarded as an object in the category of color-vector spaces where infinitesimal braiding are known (see e.g. \cite[Remark 2.2(3)]{HV}), together with Proposition \ref{prop:ibvsCartier}.]  }
  It can be shown (to be written) that $(\chi^\op * \Rr)(T_i^k\otimes T_j^l)=t_{ji}^{nm}c_{nm}^{kl}=(c\circ t)(v_j\otimes v_i)$ 
    \begin{gather*}
      (\chi^\op * \Rr)(T_i^k\otimes T_j^l)
      = \chi^\op(T_i^m\otimes T_j^n)\Rr(T_m^k\otimes T_n^l)
      = \chi(T_j^n\otimes T_i^m)\Rr(T_m^k\otimes T_n^l)
      =t_{ji}^{nm}c_{nm}^{kl}\\
      (c\circ t)(v_j\otimes v_i)=t_{ji}^{nm}c(v_n\otimes v_m)=t_{ji}^{nm}c_{nm}^{kl}v_k\otimes v_l
  \end{gather*}
  and that 
  \begin{gather*}
      (\Rr*\chi)(T_i^k\otimes T_j^l)
      = \Rr(T_i^m\otimes T_j^n)\chi(T_m^k\otimes T_n^l)
      =c_{ji}^{mn}t_{mn}^{kl}\\
      (t\circ c)(v_j\otimes v_i)=c_{ji}^{mn}t(v_m\otimes v_n)=c_{ji}^{mn}t_{mn}^{kl}v_k\otimes v_l
  \end{gather*}
    Thus, if a $q$-version of \eqref{eq:CC4} holds, i.e., $\Rr*\chi=q\chi^\op * \Rr$, then one has $t\circ c=q c\circ t$ (still to be checked if $t\circ c=q c\circ t$ implies $\Rr*\chi=q\chi^\op * \Rr$). 
Define the transpose of $t$ by $t^T(v_i\otimes v_j)=p_{ji}v_i\otimes v_j$. Then 
$tc(v_j\otimes v_i)
=q_{ji}t(v_i\otimes v_j)
=q_{ji}p_{ij}v_i\otimes v_j
=p_{ij}c(v_j\otimes v_i)
=c(p_{ij}v_j\otimes v_i)
=ct^T(v_j\otimes v_i)$ 
so that we get $t\circ c=c\circ t^T$. Now the condition $t\circ c=q c\circ t$ together with $t\circ c =c\circ t^T$ imply 
$c\circ t^T=q c\circ t$ and hence $t^T=q t$. Consider the matrix $M:=(p_{ij})$. Then $M^T=qM$ and hence $M=(M^T)^T=qM^T=q^2M$. This forces $M=0$ or $q=\pm1$. Thus we get that the matrix $(p_{ij})$ is either  zero or orthogonal ($q=1$)
  or skew-symmetric ($q=-1$). 
  Note that 
  \begin{gather*}
      (\chi * \Rr)(T_i^k\otimes T_j^l)
      = \chi(T_i^m\otimes T_j^n)\Rr(T_m^k\otimes T_n^l)
      =t_{ij}^{nm}c_{nm}^{kl}\\
      (c\circ t^T)(v_j\otimes v_i)=t_{ij}^{nm}c(v_n\otimes v_m)=t_{ij}^{nm}c_{nm}^{kl}v_k\otimes v_l
  \end{gather*}
  so apparently it could be that $\chi * \Rr=\Rr * \chi$.
 \end{invisible} 
\end{example}

We are prepared to prove the main theorem of this section, namely an FRT-construction of infinitesimal $\Rr$-forms.

\begin{theorem}[Infinitesimal FRT]\label{thm:curvedFRT}
Given an infinitesimally braided vector space $(V,c,t)$, 
there is a unique infinitesimal $\Rr$-form $\chi$ on the
coquasitriangular bialgebra $(A(c,t),\Rr)$ such that
$\chi(T_i^k\otimes T_j^\ell)=t_{ij}^{k\ell}$.
\end{theorem}
\begin{proof}
By Lemma~\ref{Lem:FRT1} the coquasitriangular bialgebra $(A(c),\Rr)$
descends to a coquasitriangular bialgebra $(A(c,t),\Rr)$ since
\eqref{eq:t1} and \eqref{eq:t2} hold. According to Lemma~\ref{Lem:FRT2}
the map $\chi\colon A(c)\otimes A(c)\to\Bbbk$ is well-defined since
\eqref{eq:t3} is satisfied in addition.
We show that $\chi$ descends to a map $A(c,t)\otimes A(c,t)\to\Bbbk$,
i.e., $\chi(A(c)\otimes I(t))=0=\chi(I(t)\otimes A(c))$,
if and only if \eqref{eq:t5} and \eqref{eq:t6} hold.
As before, it is necessary and sufficient to prove the vanishing on
generators of $A(c)$ and $I(t)$. This can be done in complete analogy
to Lemma~\ref{Lem:FRT2} by replacing 
$C_{ij}^{k\ell}=c_{ij}^{mn}T_m^kT_n^\ell-T_i^mT_j^nc_{mn}^{k\ell}$ with
$D_{ij}^{k\ell}=t_{ij}^{mn}T_m^kT_n^\ell-T_i^mT_j^nt_{mn}^{k\ell}$.
Explicitly, 
\begin{align*}
    \chi(T_i^j\otimes D_{k\ell}^{mn})
    &=t_{k\ell}^{pq}\chi(T_i^j\otimes T_p^mT_q^n)
    -\chi(T_i^j\otimes T_k^pT_\ell^q)t_{pq}^{mn}\\
    &\overset{\eqref{eq:X4}}{=}t_{k\ell}^{pq}t_{ip}^{jm}\delta_q^n
    +t_{k\ell}^{pq}\overline{c}_{ip}^{m_1j_1}t_{j_1q}^{j_2n}c_{m_1j_2}^{jm}
    -t_{ik}^{jp}\delta_\ell^qt_{pq}^{mn}
    -\overline{c}_{ik}^{p_1j_1}t_{j_1\ell}^{j_2q}c_{p_1j_2}^{jp}t_{pq}^{mn}
    \\
    &{=}t_{k\ell}^{pn}t_{ip}^{jm}
    +t_{k\ell}^{pq}\overline{c}_{ip}^{m_1j_1}t_{j_1q}^{j_2n}c_{m_1j_2}^{jm}
    -t_{ik}^{jp}t_{p\ell}^{mn}
    -\overline{c}_{ik}^{p_1j_1}t_{j_1\ell}^{j_2q}c_{p_1j_2}^{jp}t_{pq}^{mn}
\end{align*}
vanishes if and only if the sum of 
\begin{align*}
    v_{ik\ell}
    \xmapsto{t_{23}}t_{k\ell}^{pn}v_{ipn}
    \xmapsto{t_{12}}t_{k\ell}^{pn}t_{ip}^{jm}v_{jmn}
\end{align*}
and
\begin{align*}
    v_{ik\ell}
    &\xmapsto{t_{23}}t_{k\ell}^{pq}v_{ipq}
    \xmapsto{c^{-1}_{12}}t_{k\ell}^{pq}\overline{c}_{ip}^{m_1j_1}
    v_{m_1j_1q}
    \xmapsto{t_{23}}t_{k\ell}^{pq}\overline{c}_{ip}^{m_1j_1}t_{j_1q}^{j_2n}v_{m_1j_2n}
    \xmapsto{c_{12}}t_{k\ell}^{pq}\overline{c}_{ip}^{m_1j_1}t_{j_1q}^{j_2n}c_{m_1j_2}^{jm}v_{jmn}
\end{align*}
equals the sum of
\begin{align*}
    v_{ik\ell}
    \xmapsto{t_{12}}t_{ik}^{jp}v_{jp\ell}
    \xmapsto{t_{23}}t_{ik}^{jp}t_{p\ell}^{mn}v_{jmn}
\end{align*}
and
\begin{align*}
    v_{ik\ell}
    &\xmapsto{c^{-1}_{12}}\overline{c}_{ik}^{p_1j_1}v_{p_1j_1\ell}
    \xmapsto{t_{23}}\overline{c}_{ik}^{p_1j_1}t_{j_1\ell}^{j_2q}
    v_{p_1j_2q}
    \xmapsto{c_{12}}\overline{c}_{ik}^{p_1j_1}t_{j_1\ell}^{j_2q}c_{p_1j_2}^{jp}v_{jpq}
    \xmapsto{t_{23}}\overline{c}_{ik}^{p_1j_1}t_{j_1\ell}^{j_2q}c_{p_1j_2}^{jp}t_{pq}^{mn}v_{jmn}.
\end{align*}
Thus, $\chi(A(c)\otimes I(t))=0$ if and only if \eqref{eq:t5} holds
and similarly $\chi(I(t)\otimes A(c))=0$ if and only if \eqref{eq:t6}
is satisfied.

It remains to prove that $\chi\colon A(c,t)\otimes A(c,t)\to\Bbbk$
is an infinitesimal $\Rr$-form for the coquasitriangular bialgebra
$(A(c,t),\Rr)$. Namely, we have to check 
\eqref{eq:CC1},\eqref{eq:CC2} and \eqref{eq:CC3}.
The last two are satisfied on generators by definition, see
\eqref{eq:X2} and \eqref{eq:X3}. They hold on arbitrary elements of
$A(c,t)\otimes A(c,t)$ by the inductive definition of $\chi$, see also Proposition~\ref{prop:bitriang}.
Thus, it is left to prove \eqref{eq:CC1} (and again, this is sufficient
on generators). Now, by the definition of $I(t)$ we have
\begin{align*}
    \chi(T_i^k\otimes T_j^\ell)T_k^mT_\ell^n
    =t_{ij}^{k\ell}T_k^mT_\ell^n
    =T_i^kT_j^\ell t_{k\ell}^{mn}
    =T_i^kT_j^\ell\chi(T_k^m\otimes T_\ell^n),
\end{align*}
which concludes the proof of the theorem.
\begin{invisible}
    We include here a proof (at list one one side) of the fact that it suffices to check \eqref{eq:CC1}, i.e., the equality $\chi \left(
a_{1}\otimes b_{1}\right) a_{2}b_{2}=a_{1}b_{1}\chi \left( a_{2}\otimes
b_{2}\right) $, on generators.
Assume that the above formula holds for $a$ of homogeneous degree $%
\left\vert ab\right\vert =\left\vert a\right\vert +\left\vert b\right\vert
\leq n-1$ (e.g. $\left\vert T_{i}^{j}T_{s}^{t}\right\vert =2$).
Now, let $a$ and $b$ such that $\left\vert a\right\vert ,\left\vert
b\right\vert ,\left\vert c\right\vert \geq 1$ and $\left\vert abc\right\vert
=\left\vert a\right\vert +\left\vert b\right\vert +\left\vert c\right\vert
=n.$ Then

\begin{eqnarray*}
\chi \left( a_{1}b_{1}\otimes c_{1}\right) (a_{2}b_{2})c_{2} &=&\varepsilon
(a_{1})\chi (b_{1}\otimes c_{1})a_{2}b_{2}c_{2}+\chi (a_{1}\otimes
b_{1}\triangleright c_{1})a_{2}b_{2}c_{2} \\
&=&a\chi (b_{1}\otimes c_{1})b_{2}c_{2}+\chi (a_{1}\otimes
b_{1}\triangleright c_{1})a_{2}b_{2}c_{2} \\
&\overset{\left\vert b_{1}+c_{1}\right\vert \leq n-1}{=}&ab_{1}c_{1}\chi
(b_{2}\otimes c_{2})+\chi (a_{1}\otimes b_{1}\triangleright
c_{1})a_{2}b_{2}c_{2} \\
&\overset{(\ast )}{=}&a_{1}b_{1}c_{1}\varepsilon (a_{2})\chi (b_{2}\otimes
c_{2})+a_{1}b_{1}c_{1}\chi (a_{2}\otimes b_{2}\triangleright c_{2}) \\
&=&(a_{1}b_{1})c_{1}\chi \left( a_{2}b_{2}\otimes c_{2}\right)
\end{eqnarray*}%
where $(\ast )$ follows once we have checked that $\chi (a_{1}\otimes
b_{1}\triangleright c_{1})a_{2}b_{2}c_{2}=a_{1}b_{1}c_{1}\chi (a_{2}\otimes
b_{2}\triangleright c_{2}).$ Indeed we have
\begin{eqnarray*}
&&\chi (a_{1}\otimes b_{1}\triangleright c_{1})a_{2}b_{2}c_{2} \\
&=&\chi (a_{1}\otimes \mathcal{R}^{-1}(b_{1}\otimes c_{1})c_{2}\mathcal{R}%
(b_{2}\otimes c_{3}))a_{2}b_{3}c_{4} \\
&=&\mathcal{R}^{-1}(b_{1}\otimes c_{1})\chi (a_{1}\otimes c_{2})a_{2}%
\mathcal{R}(b_{2}\otimes c_{3})b_{3}c_{4} \\
&\overset{\eqref{eq:ct1}}{=}&\mathcal{R}^{-1}(b_{1}\otimes c_{1})\chi
(a_{1}\otimes c_{2})a_{2}c_{3}b_{2}\mathcal{R}\left( b_{3}\otimes
c_{4}\right)  \\
&\overset{\left\vert a_{2}+c_{2}\right\vert \leq n-1}{=}&\mathcal{R}%
^{-1}(b_{1}\otimes c_{1})a_{1}c_{2}\chi \left( a_{2}\otimes c_{3}\right)
b_{2}\mathcal{R}\left( b_{3}\otimes c_{4}\right)  \\
&=&a_{1}\mathcal{R}^{-1}(b_{1}\otimes c_{1})c_{2}b_{2}\chi \left(
a_{2}\otimes c_{3}\right) \mathcal{R}\left( b_{3}\otimes c_{4}\right)  \\
&\overset{\eqref{eq:ct1}}{=}&a_{1}b_{1}c_{1}\mathcal{R}^{-1}\left( b_{2}\otimes
c_{2}\right) \chi \left( a_{2}\otimes c_{3}\right) \mathcal{R}\left(
b_{3}\otimes c_{4}\right)  \\
&=&a_{1}b_{1}c_{1}\chi (a_{2}\otimes b_{2}\triangleright c_{2}).
\end{eqnarray*}
\end{invisible}
\end{proof}
We show that also the converse of Theorem~\ref{thm:curvedFRT} is true,
thus giving a classification of infinitesimal braidings on FRT
bialgebras.
\begin{proposition}\label{prop:ibvsCartier}
Let $(H,\Rr,\chi)$ be a pre-Cartier coquasitriangular bialgebra
and $V$ a left $H$-comodule. Then, the $\Bbbk$-linear maps
$c,t\in\mathrm{End}_\Bbbk(V\otimes V)$ defined on $v,v'\in V$ by
\begin{align}
    c(v\otimes v')=\Rr(v'_{-1}\otimes v_{-1})v'_0\otimes v_0,
    \qquad\qquad
    t(v\otimes v')=\chi(v_{-1}\otimes v'_{-1})v_0\otimes v'_0
\end{align}
yield an infinitesimally braided vector space $(V,c,t)$.
%satisfy $c_{12}c_{23}c_{12}=c_{23}c_{12}c_{23}$, \eqref{eq:t1},\eqref{eq:t2},\eqref{eq:t3},%\eqref{eq:t4},
%\eqref{eq:t5} and \eqref{eq:t6}.
\end{proposition}
\begin{proof}
Since $\Rr$ is a coquasitriangular structure, the induced braiding
$\sigma^\Rr_{V,V}=c$ on the left $H$-comodule $V$  satisfies the
braid equation \eqref{eq:braid}.
Let $v,v',v''\in V$.
Then,
\begin{align*}
    c_{12}(c_{23}(t_{12}(v\otimes v'\otimes v'')))
    &=\chi(v_{-1}\otimes v'_{-1})c_{12}(c_{23}(v_0\otimes v'_0\otimes v''))\\
    &=\chi(v_{-1}\otimes v'_{-2})\Rr(v''_{-1}\otimes v'_{-1})c_{12}(v_0\otimes v''_0\otimes v'_0)\\
    &=\chi(v_{-2}\otimes v'_{-2})\Rr(v''_{-2}\otimes v'_{-1})\Rr(v''_{-1}\otimes v_{-1})v''_0\otimes v_0\otimes v'_0\\
    &\overset{\eqref{eq:ct2}}{=}\chi(v_{-2}\otimes v'_{-2})\Rr(v''_{-1}\otimes v_{-1}v'_{-1})v''_0\otimes v_0\otimes v'_0\\
    &\overset{\eqref{eq:CC1}}{=}\Rr(v''_{-1}\otimes v_{-2}v'_{-2})\chi(v_{-1}\otimes v'_{-1})v''_0\otimes v_0\otimes v'_0\\
    &\overset{\eqref{eq:ct2}}{=}\Rr(v''_{-2}\otimes v'_{-2})\Rr(v''_{-1}\otimes v_{-2})\chi(v_{-1}\otimes v'_{-1})v''_0\otimes v_0\otimes v'_0\\
    &=\Rr(v''_{-2}\otimes v'_{-1})\Rr(v''_{-1}\otimes v_{-1})t_{23}(v''_0\otimes v_0\otimes v'_0)\\
    &=\Rr(v''_{-1}\otimes v'_{-1})t_{23}(c_{12}(v\otimes v''_0\otimes v'_0))
    =t_{23}(c_{12}(c_{23}(v\otimes v'\otimes v''))),
\end{align*}
i.e., \eqref{eq:t1} is satisfied.
Similarly, one shows that \eqref{eq:t2} holds by using \eqref{eq:ct3}
and \eqref{eq:CC1}.
Note that the inverse of $c$ is given by $c^{-1}(v\otimes v')
=\Rr^{-1}(v_{-1}\otimes v'_{-1})v'_0\otimes v_0$.
Moreover, one shows that \eqref{eq:t3} holds by means of the equality $\mathcal{R}_{23}^{-1}\ast \chi _{13}\ast \mathcal{R}_{23}= \mathcal{R}_{12}^{-1}\ast \chi
_{13}\ast \mathcal{R}_{12}$ that we know is true by Theorem \ref{thm:ccQYB}.

\begin{invisible} OLD PROOF.
Next, we show that \eqref{eq:t3'} and \eqref{eq:t4'} are satisfied.
Then, together with \eqref{eq:t1} and \eqref{eq:t2} this implies that
\eqref{eq:t3}  hold.
We have
\begin{align*}
    &(t_{12}+c_{12}t_{23}c^{-1}_{12})c_{23}(v\otimes v'\otimes v'')
    =\Rr(v''_{-1}\otimes v'_{-1})(t_{12}+c_{12}t_{23}c^{-1}_{12})(v\otimes v''_0\otimes v'_0)\\
    &=\Rr(v''_{-2}\otimes v'_{-1})\chi(v_{-1}\otimes v''_{-1})v_0\otimes v''_0\otimes v'_0\\
    &\qquad+\Rr(v''_{-2}\otimes v'_{-1})\Rr^{-1}(v_{-1}\otimes v''_{-1})c_{12}(t_{23}(v''_0\otimes v_0\otimes v'_0))\\
    &=\Rr(v''_{-2}\otimes v'_{-1})\chi(v_{-1}\otimes v''_{-1})v_0\otimes v''_0\otimes v'_0\\
    &\qquad+\Rr(v''_{-2}\otimes v'_{-2})\Rr^{-1}(v_{-2}\otimes v''_{-1})\chi(v_{-1}\otimes v'_{-1})c_{12}(v''_0\otimes v_0\otimes v'_0)\\
    &=\Rr(v''_{-2}\otimes v'_{-1})\chi(v_{-1}\otimes v''_{-1})v_0\otimes v''_0\otimes v'_0\\
    &\qquad+\Rr(v''_{-3}\otimes v'_{-2})\Rr^{-1}(v_{-3}\otimes v''_{-2})\chi(v_{-2}\otimes v'_{-1})\Rr(v_{-1}\otimes v''_{-1})v_0\otimes v''_0\otimes v'_0\\
    &=\Rr(v''_{-2}\otimes v'_{-2})(\chi_{12}
    +\Rr^{-1}_{12}*\chi_{13}*\Rr_{12})(v_{-1}\otimes v''_{-1}\otimes v'_{-1})v_0\otimes v''_0\otimes v'_0\\
    &\overset{\eqref{eq:CC2}}{=}\Rr(v''_{-2}\otimes v'_{-2})\chi(v_{-1}\otimes v''_{-1}v'_{-1})v_0\otimes v''_0\otimes v'_0\\
    &\overset{\eqref{eq:ct1}}{=}\chi(v_{-1}\otimes v'_{-2}v''_{-2})\Rr(v''_{-1}\otimes v'_{-1})v_0\otimes v''_0\otimes v'_0\\
    &=\chi(v_{-1}\otimes v'_{-1}v''_{-1})c_{23}(v_0\otimes v'_0\otimes v''_0)\\
    &\overset{\eqref{eq:CC2}}{=}(\chi_{12}
    +\Rr^{-1}_{12}*\chi_{13}*\Rr_{12})(v_{-1}\otimes v'_{-1}\otimes v''_{-1})c_{23}(v_0\otimes v'_0\otimes v''_0)\\
    &=c_{23}t_{12}(v\otimes v'\otimes v'')+ c_{23}((\Rr^{-1}(v_{-3}\otimes v'_{-2})\chi(v_{-2}\otimes v''_{-1})\Rr(v_{-1}\otimes v'_{-1})\mathrm{Id}_{V^{\otimes 3}})(v_0\otimes v'_0\otimes v''_0))\\
    &=c_{23}(t_{12}+c_{12}t_{23}c^{-1}_{12})(v\otimes v'\otimes v''),
\end{align*}
which proves that \eqref{eq:t3'} holds. Similarly \eqref{eq:t4'}
is shown.\end{invisible}
Finally, equations \eqref{eq:t5} and \eqref{eq:t6} are proven
using \eqref{eq:CC1}.
\end{proof}

We noticed that, given a braided vector space $(V,c)$, then $(V,c,\lambda\cdot\mathrm{Id}_{V\ot V})$ is an infinitesimally braided vector spaces. This leads to the following
proposition.
\begin{proposition}[Canonical infinitesimal braiding on $(A(c),\Rr)$]\label{Prop:xFRT}
For any braided vector space $(V,c)$  the coquasitriangular bialgebra
$(A(c),\Rr)$ admits a $1$-parameter family of infinitesimal $\Rr$-forms $\chi_\lambda
\colon A(c)\otimes A(c)\to\Bbbk$ given explicitly on generators by
\begin{equation}\label{xFRT}
    \chi_\lambda(T_{i_1}^{j_1}\cdots T_{i_m}^{j_m}
    \otimes T_{k_1}^{\ell_1}\cdots T_{k_n}^{\ell_n})
    =\lambda mn\varepsilon(T_{i_1}^{j_1}\cdots T_{i_m}^{j_m}
    T_{k_1}^{\ell_1}\cdots T_{k_n}^{\ell_n})
\end{equation}
for all $\lambda\in\Bbbk$. Moreover, $(A(c),\Rr,\chi_\lambda)$ is Cartier.
\end{proposition}
\begin{proof}
As noted in Example \ref{ex:canT}, given an invertible solution
$c\in\mathrm{End}_\Bbbk(V\otimes V)$ of the braid equation 
and choosing $t=\lambda\cdot\mathrm{Id}\in\mathrm{End}_\Bbbk(V\otimes V)$ the tuple
$(c,t)$ satisfies all assumptions of Theorem~\ref{thm:curvedFRT}.
In this case $I(t)=0$, since $t_{ij}^{k\ell}=\lambda\delta_i^k\delta_j^\ell$
and consequently $A(c,t)=A(c)$. Explicitly,
$\chi_\lambda\colon A(c)\otimes A(c)\to\Bbbk$ is determined on
generators by 
$\chi_\lambda(T_i^k\otimes T_j^\ell)=t_{ij}^{k\ell}=\lambda\delta_i^k\delta_j^\ell=\lambda\varepsilon(T_i^kT_j^\ell)
$
and extended via \eqref{eq:CC2} and \eqref{eq:CC3}.
\begin{invisible}
\begin{align*}
    \chi_\lambda(T_i^j\otimes T_k^mT_\ell^n)
    &=\chi_\lambda(T_i^j\otimes T_k^m)\varepsilon(T_\ell^n)
    +(\Rr^{-1}_{12}*(\chi_\lambda)_{13}*\Rr_{12})(T_i^j\otimes T_k^m\otimes T_\ell^n)\\
    &=\lambda\varepsilon(T_i^jT_k^mT_\ell^n)
    +\lambda(\Rr^{-1}_{12}*\varepsilon*\Rr_{12})(T_i^j\otimes T_k^m\otimes T_\ell^n)
    =2\lambda\varepsilon(T_i^jT_k^mT_\ell^n)
\end{align*}
and
\begin{align*}
    \chi_\lambda(T_i^kT_j^\ell\otimes T_m^n)
    &=\varepsilon(T_i^k)\chi_\lambda(T_j^\ell\otimes T_m^n)
    +(\Rr^{-1}_{23}*(\chi_\lambda)_{13}*\Rr_{23})(T_i^k\otimes T_j^\ell\otimes T_m^n)\\
    &=\lambda\varepsilon(T_i^kT_j^\ell T_m^n)
    +\lambda(\Rr^{-1}_{23}*\varepsilon*\Rr_{23})(T_i^k\otimes T_j^\ell\otimes T_m^n)
    =2\lambda\varepsilon(T_i^kT_j^\ell T_m^n).
\end{align*}\end{invisible}
An easy induction reveals the formula \eqref{xFRT}.

We include also a direct proof that $\chi_\lambda$ is an infinitesimal $\Rr$-form. By the same ideas of Remark \ref{rmk:hateps}, the algebra $A(c)$ inherits the graduation from the free algebra $F$ and hence we obtain a derivation $\partial:A(c)\to\Bbbk$ that on $\xi$ homogeneous of degree $|\xi|$ is defined by $\partial(\xi)=|\xi|\varepsilon(\xi)$. Note that $\xi(T_{i_1}^{j_1}\cdots T_{i_m}^{j_m})=m\varepsilon(T_{i_1}^{j_1}\cdots T_{i_m}^{j_m})$ so that $\chi_\lambda=\lambda\partial\otimes\partial$ is a symmetric biderivation. Since we further have $\partial(\xi_1)\xi_2=\xi_1 \partial(\xi_2)$ for every $\xi\in A(c)$, we get that $\partial$ is central whence so is $\chi_\lambda$. Thus, by Example \ref{exa:chisymco}, we conclude that $(A(c),\Rr,\chi_\lambda)$ is Cartier.
\end{proof}
%%%%%%%%%%%%%%%%%%%%%%%%%%%%%%%%%%%%%%%%%%%
\begin{comment}
\begin{proposition}
Let $(V,c,t)$ be the infinitesimally braided vector space
of Example~\ref{ex:diag}. Namely, for some $q_{ij}\in\Bbbk\setminus\{0\}$ and $p_{ij}\in\Bbbk$ we have
$$
c(v_i\otimes v_j)=q_{ij}v_j\otimes v_i,\qquad\qquad
t(v_i\otimes v_j)=p_{ij}v_i\otimes v_j.
$$
Then 

\end{proposition}
\end{comment}
%%%%%%%%%%%%%%%%%%%%%%%%%%%%%%%%%%%%%%%%%%%
\begin{example}
Following \cite[Section VIII.7]{Kassel} for a $2$-dimensional $\mathbb{C}$-vector space with basis $v_1,v_2\in V$ and for a non-zero scalar $q\in\mathbb{C}$ we consider the invertible solution $c\in\mathrm{End}_\mathbb{C}(V\otimes V)$ of the braid equation determined by $c(v_1\otimes v_1)=q^{\frac{1}{2}}v_1\otimes v_1$, $c(v_2\otimes v_2)=q^{\frac{1}{2}}v_2\otimes v_2$, $c(v_1\otimes v_2)=q^{-\frac{1}{2}}v_2\otimes v_1$ and $c(v_2\otimes v_1)=q^{-\frac{1}{2}}v_1\otimes v_2+q^{-\frac{1}{2}}(q-q^{-1})v_2\otimes v_1$. Using Proposition~\ref{Prop:xFRT} this gives a Cartier coquasitriangular bialgebra $(A(c),\Rr',\chi_\lambda)$ for any $\lambda\in\mathbb{C}$. According to \cite[Proposition VIII.7.1 and Corollary VIII.7.2]{Kassel} $(A(c),\Rr')\cong(\mathrm{M}_q(2),\Rr)$ are isomorphic as coquasitriangular bialgebras. The induced infinitesimal $\Rr$-form on $(\mathrm{M}_q(2),\Rr)$ is the one we already discussed in Proposition~\ref{prop:Mq2chi}.
\end{example}

\subsection{Tannaka-Krein reconstruction}\label{Sec:Tannaka}

Denote by $\mathsf{Vec}_{\mathsf{f}}$ the category of finite-dimensional $%
\Bbbk$-vector spaces. Given a category $\Cc$ and a functor $\omega:%
\Cc\to \mathsf{Vec}_{\mathsf{f}}$, then Tannaka-Krein reconstruction
allows to build a coalgebra $C$ such that $\omega$ factors through the category $%
\mathcal{M}^C_{\mathsf{f}}$ of finite-dimensional right $C$-comodules
\begin{align}\label{diag:tannaka}
\begin{split}\xymatrixrowsep{.7cm}
\xymatrix{\Cc\ar[rr]^-{\Omega}\ar[rd]_-{\omega}&&\Mm^C_{\mathsf
f}\ar[ld]^-F\\&\Vec_{\mathsf f} }  
\end{split}
\end{align}
where $F$ denotes the forgetful functor, see e.g. \cite[Corollary 2.1.9]%
{Schauenburg92}. We will refer to $C$ as the reconstructed coalgebra. In
favorable cases $\Cc$ comes out to be even equivalent to $\mathcal{M}%
^C_{\mathsf{f}}$.

If $\Cc$ is further a monoidal category and $\omega:\Cc\to
\mathsf{Vec}_\mathsf{f}$ is a strong monoidal functor, then $C$ becomes a bialgebra, eventually with antipode in case $\Cc$ is rigid, see
e.g. \cite[Section 2]{Ulbrich-onHopf}. Moreover, if $\Cc$ is braided, then $C$ inherits a coquasitriangular structure $\Rr:C\otimes C\to\Bbbk$, which results to be cotriangular in case the category is symmetric, see e.g. \cite[page 475]{Majid-book}.

In this section, we are dealing with conditions on $(\Cc,\omega )$
guaranteeing that the reconstructed coalgebra is a (pre-)Cartier (quasi)triangular bialgebra
possibly with antipode. Since we are not interested in finding the most
general possible results, but only eager to find applications of the notions
we have introduced so far, it may be that some of our assumptions can be relaxed.

\begin{theorem}
Let $\left( \Cc,\sigma \right) $ be a pre-additive braided monoidal
category, $\omega :\Cc\rightarrow \mathsf{Vec}_{\mathsf{f}}$ be an additive strong monoidal functor and $B$ be the reconstructed coalgebra
with its coquasitriangular bialgebra structure $\Rr:B\otimes B\rightarrow
\Bbbk .$ If there is a natural transformation $t_{X,Y}:X\otimes Y\to X\otimes Y$, then there is a $\Bbbk$-linear map $\chi
:B\otimes B\rightarrow \Bbbk $ that verifies \eqref{eq:CC1}. Moreover:

\begin{enumerate}[$i)$]
\item if $t$ fulfills \eqref{qC-I}, then $\chi $ verifies %
\eqref{eq:CC2};

\item if $t$ fulfills \eqref{qC-II}, then $\chi $ verifies %
\eqref{eq:CC3};

\item if $t$ fulfills \eqref{cart}, then $\chi $ verifies \eqref{eq:CC4}.
\end{enumerate}
As a consequence
\begin{itemize}
  \item If $\left( \Cc,\sigma \right) $ is braided (pre-)Cartier, then $(B,\Rr)$ is a (pre-)Cartier coquasitriangular bialgebra;
  \item If $\left( \Cc,\sigma \right) $ is symmetric (pre-)Cartier, then $(B,\Rr)$ is a (pre-)Cartier cotriangular bialgebra.
\end{itemize}
\end{theorem}

\begin{proof}
Since $\Cc$ is a monoidal category and $(\omega,\phi^0,\phi^2) :\Cc%
\rightarrow \mathsf{Vec}_{\mathsf{f}}$ is a strong monoidal functor, we can construct a bialgebra $B$  following \cite[%
Section 2]{Ulbrich-onHopf}. Indeed, the
functor $\mathsf{Vec}\rightarrow \mathsf{Set},M\mapsto \mathrm{Nat}%
(\omega ,\omega \otimes M)$ is representable. Choose a representing vector
space $B$ and natural isomorphisms $\theta _{M}:\mathrm{Hom}_{\Bbbk
}(B,M)\rightarrow \mathrm{Nat}(\omega ,\omega \otimes M).$ Set $\alpha
:=\theta _{B}(\mathrm{Id}_{B}):\omega \rightarrow \omega \otimes B$, set $%
d:=(\alpha \otimes B)\alpha :\omega \rightarrow \omega \otimes B\otimes B$
and let $e:\omega \rightarrow \omega \otimes \Bbbk $ be the canonical
isomorphism. Then the coalgebra structure $(B,\Delta ,\varepsilon )$ is
given by $\theta _{B\otimes B}(\Delta )=d$ and $\theta _{\Bbbk }(\varepsilon
)=e$. In this way we obtain the functor $\Omega :\Cc\rightarrow \mathcal{M}%
_{\mathsf{f}}^{B},X\mapsto (\omega (X),\alpha _{X})$, such that $F\circ
\Omega =\omega $, as in \eqref{diag:tannaka}. For the algebra structure, there is an isomorphism $\theta
_{M}^{2}:\mathrm{Hom}_{\Bbbk }(B\otimes B,M)\rightarrow \mathrm{Nat}(\omega
\otimes \omega ,\omega \otimes \omega \otimes M)$ where, for every $%
f:B\otimes B\rightarrow M$, the natural transformation $\theta
_{M}^{2}(f):\omega \otimes \omega \rightarrow \omega \otimes \omega \otimes M
$ is defined on components by setting $\theta _{M}^{2}(f)_{X,Y}:\omega
(X)\otimes \omega (Y)\rightarrow \omega (X)\otimes \omega (Y)\otimes
M,x\otimes y\mapsto x_{0}\otimes y_{0}\otimes f(x_{1}\otimes y_{1}).$ Here
we used the notation $\alpha _{X}(x):=x_{0}\otimes x_{1}$. Denote by $\mu
_{X,Y}$ the following composition
\begin{equation*}
\xymatrixcolsep{1.8cm}\xymatrix{\omega(X)\otimes\omega(Y)\ar[r]^{%
\phi^2_{X,Y}}& \omega(X\otimes Y)\ar[r]^{\alpha_{X\otimes Y}}&
\omega(X\otimes Y)\otimes B\ar[r]^{(\phi^2_{X,Y}\otimes B)^{-1}}&
\omega(X)\otimes\omega(Y)\otimes B}.
\end{equation*}%
This defines a natural transformation $\mu :\omega \otimes \omega
\rightarrow \omega \otimes \omega \otimes B$ and the multiplication $%
m:B\otimes B\rightarrow B$ is given by $\theta _{B}^{2}(m)=\mu $. The unit $%
u:\Bbbk \rightarrow B$ is defined as the composition
\begin{equation*}
\xymatrixcolsep{1.8cm}\xymatrix{\Bbbk\ar[r]^{\phi^0}&
\omega(\unit)\ar[r]^{\alpha_\unit}& \omega(\unit)\otimes
B\ar[r]^{(\phi^0\otimes B)^{-1}}& \Bbbk\otimes B\ar[r]^{l_B}& B }.
\end{equation*}

Suppose further that $\Cc$ has a braiding $\sigma $. Now, following
\cite[page 475]{Majid-book}, we can endow $B$ with a coquasitriangular
structure $\Rr:B\otimes B\rightarrow \Bbbk $ defined as the unique
map such that $\theta _{\Bbbk }^{2}(\Rr)$ on components is uniquely determined by the commutativity of the following diagram
\begin{equation*}
\xymatrixcolsep{1.2cm}\xymatrix{\omega(X)\otimes\omega(Y)
\ar[d]_{\theta _{\Bbbk }^{2}(\Rr)_{X,Y}}
\ar[r]^{\phi^2_{X,Y}}& \omega(X\otimes Y)\ar[r]^{\omega(\sigma_{X, Y})}&
\omega(Y\otimes X)\ar[r]^{(\phi^2_{Y,X})^{-1}}&
\omega(Y)\otimes\omega(X)\ar[d]^{\tau_{\omega(Y),\omega(X)}}\\
\omega(X)\otimes\omega(Y)\otimes\Bbbk\ar[rrr]^-{r_{\omega(X)\otimes\omega(Y)}}&&& \omega(X)\otimes\omega(Y).  }
\end{equation*}%
Although the functor here lands in the finite-dimensional vector spaces, while in
Majid's approach it lands in vector spaces, still one gets that \eqref{eq:ct1}, \eqref{eq:ct2}, \eqref{eq:ct3} hold true. Note that $\Rr$ is convolution invertible with convolution inverse $%
\Rr^{-1}:B\otimes B\rightarrow \Bbbk $ which is the unique map such
that $\theta _{\Bbbk }^{2}(\Rr^{-1})$ on components is $\theta
_{\Bbbk }^{2}(\Rr^{-1})_{X,Y}=r_{\omega (X)\otimes \omega
(Y)}^{-1}\left( \phi _{X,Y}^{2}\right) ^{-1}\omega \left( \sigma
_{X,Y}^{-1}\right) \phi _{Y,X}^{2}\tau _{\omega \left( X\right) ,\omega
\left( Y\right) }.$ 

Moreover, $B$ becomes in fact cotriangular in case $\sigma $ is a
symmetry, i.e., $\sigma _{X,Y}^{-1}=\sigma _{Y,X}.$ 

So far everything is well-known. Let us concern now the
infinitesimal braiding.

Suppose further that $\Cc$ has a natural transformation $t$. Then we
can define $\chi :B\otimes B\rightarrow \Bbbk $ to be the unique map such
that $\theta _{\Bbbk }^{2}(\chi )$ on components is the composition
\begin{equation*}
\xymatrixcolsep{1.3cm}\xymatrix{\omega(X)\otimes\omega(Y)\ar[r]^{%
\phi^2_{X,Y}}& \omega(X\otimes Y)\ar[r]^{\omega(t_{X, Y})}& \omega(X\otimes
Y)\ar[r]^{(\phi^2_{X,Y})^{-1}}& \omega(X)\otimes\omega(Y)\ar[r]^{r_{\omega
(X)\otimes \omega (Y)}^{-1}}& \omega (X)\otimes \omega (Y)\otimes \Bbbk}
\end{equation*}

One can check that condition \eqref{eq:CC1} is equivalent to the equality
\begin{equation*}
(r_{\omega (X)\otimes \omega (Y)}\otimes B)(\theta _{\Bbbk }^{2}(\chi
)_{X,Y}\otimes B)\theta _{B}^{2}(m)_{X,Y}=\theta _{B}^{2}(m)_{X,Y}r_{\omega
(X)\otimes \omega (Y)}\theta _{\Bbbk }^{2}(\chi )_{X,Y}
\end{equation*}%
which, in turn, holds in view of the following computation:
\begin{align*}
&(r_{\omega (X)\otimes \omega (Y)}\otimes B)(\theta _{\Bbbk }^{2}(\chi
)_{X,Y}\otimes B)\theta _{B}^{2}(m)_{X,Y} \\
&=(r_{\omega (X)\otimes \omega (Y)}\otimes B)(r_{\omega (X)\otimes \omega
(Y)}^{-1}\otimes B)(\left( \phi _{X,Y}^{2}\right) ^{-1}\otimes B)(\omega
\left( t_{X,Y}\right) \otimes B)(\phi _{X,Y}^{2}\otimes B)\left( \phi
_{X,Y}^{2}\otimes B\right) ^{-1}\alpha _{X\otimes Y}\phi _{X,Y}^{2} \\
&=(\left( \phi _{X,Y}^{2}\right) ^{-1}\otimes B)(\omega \left(
t_{X,Y}\right) \otimes B)\alpha _{X\otimes Y}\phi _{X,Y}^{2} \\
&=(\left( \phi _{X,Y}^{2}\right) ^{-1}\otimes B)\alpha _{X\otimes Y}\omega
\left( t_{X,Y}\right) \phi _{X,Y}^{2} \\
&=\left( \phi _{X,Y}^{2}\otimes B\right) ^{-1}\alpha _{X\otimes Y}\phi
_{X,Y}^{2}\left( \phi _{X,Y}^{2}\right) ^{-1}\omega \left( t_{X,Y}\right)
\phi _{X,Y}^{2} \\
&=\left( \phi _{X,Y}^{2}\otimes B\right) ^{-1}\alpha _{X\otimes Y}\phi
_{X,Y}^{2}r_{\omega (X)\otimes \omega (Y)}r_{\omega (X)\otimes \omega
(Y)}^{-1}\left( \phi _{X,Y}^{2}\right) ^{-1}\omega \left( t_{X,Y}\right)
\phi _{X,Y}^{2} \\
&=\theta _{B}^{2}(m)_{X,Y}r_{\omega (X)\otimes \omega (Y)}\theta _{\Bbbk
}^{2}(\chi )_{X,Y}.
\end{align*}

$i)$ It comes out that \eqref{eq:CC2} is equivalent to 
\begin{eqnarray*}
&&\left( \omega \left( X\right) \otimes \left( \phi _{Y,Z}^{2}\right)
^{-1}\otimes \Bbbk \right) \theta _{\Bbbk }^{2}(\chi )_{X,Y\otimes Z}\left(
\omega \left( X\right) \otimes \phi _{Y,Z}^{2}\right)  \\
&=&r_{\omega (X)\otimes \omega (Y)\otimes \omega (Z)}^{-1}\left( r_{\omega
(X)\otimes \omega (Y)}\otimes \omega (Z)\right) \left( \theta _{\Bbbk
}^{2}(\chi )_{X,Y}\otimes \omega (Z)\right) + \\
&&+\left[
\begin{array}{c}
r_{\omega (X)\otimes \omega (Y)\otimes \omega (Z)}^{-1}\left( r_{\omega
\left( X\right) \otimes \omega \left( Y\right) }\otimes \omega \left(
Z\right) \right) \left( \theta _{\Bbbk }^{2}(\Rr^{-1})_{X,Y}\otimes
\omega \left( Z\right) \right) \left( \tau _{\omega \left( Y\right) ,\omega
\left( X\right) }\otimes \omega \left( Z\right) \right)  \\
r_{\omega \left( Y\right) \otimes \omega \left( X\right) \otimes \omega
\left( Z\right) }\left( \omega \left( Y\right) \otimes \theta _{\Bbbk
}^{2}(\chi )_{X,Z}\right) \left( \tau _{\omega \left( X\right) ,\omega
\left( Y\right) }\otimes \omega \left( Z\right) \right) \\
\left( r_{\omega
\left( X\right) \otimes \omega \left( Y\right) }\otimes \omega \left(
Z\right) \right) \left( \theta _{\Bbbk }^{2}(\Rr)_{X,Y}\otimes
\omega \left( Z\right) \right) %
\end{array}%
\right] .
\end{eqnarray*}%
If we compose by $r_{\omega (X)\otimes \omega (Y)\otimes \omega (Z)}$ this
equality becomes%
\begin{eqnarray*}
&&\left( \omega \left( X\right) \otimes \left( \phi _{Y,Z}^{2}\right)
^{-1}\right) r_{\omega (X)\otimes \omega (Y\otimes Z)}\theta _{\Bbbk
}^{2}(\chi )_{X,Y\otimes Z}\left( \omega \left( X\right) \otimes \phi
_{Y,Z}^{2}\right)  \\
&=&\left( r_{\omega (X)\otimes \omega (Y)}\otimes \omega (Z)\right) \left(
\theta _{\Bbbk }^{2}(\chi )_{X,Y}\otimes \omega (Z)\right) + \\
&&+\left[
\begin{array}{c}
\left( r_{\omega \left( X\right) \otimes \omega \left( Y\right) }\otimes
\omega \left( Z\right) \right) \left( \theta _{\Bbbk }^{2}(\Rr%
^{-1})_{X,Y}\otimes \omega \left( Z\right) \right) \left( \tau _{\omega
\left( Y\right) ,\omega \left( X\right) }\otimes \omega \left( Z\right)
\right)  \\
r_{\omega \left( Y\right) \otimes \omega \left( X\right) \otimes \omega
\left( Z\right) }\left( \omega \left( Y\right) \otimes \theta _{\Bbbk
}^{2}(\chi )_{X,Z}\right) \left( \tau _{\omega \left( X\right) ,\omega
\left( Y\right) }\otimes \omega \left( Z\right) \right) \\
\left( r_{\omega
\left( X\right) \otimes \omega \left( Y\right) }\otimes \omega \left(
Z\right) \right) \left( \theta _{\Bbbk }^{2}(\Rr)_{X,Y}\otimes
\omega \left( Z\right) \right)%
\end{array}%
\right] .
\end{eqnarray*}%
Let us rewrite the three terms in the order in which they appear%
\begin{eqnarray*}
&&\left( \omega \left( X\right) \otimes \left( \phi _{Y,Z}^{2}\right)
^{-1}\right) r_{\omega (X)\otimes \omega (Y\otimes Z)}\theta _{\Bbbk
}^{2}(\chi )_{X,Y\otimes Z}\left( \omega \left( X\right) \otimes \phi
_{Y,Z}^{2}\right)  \\
&=&\left( \omega \left( X\right) \otimes \left( \phi _{Y,Z}^{2}\right)
^{-1}\right) r_{\omega (X)\otimes \omega (Y\otimes Z)}r_{\omega (X)\otimes
\omega (Y\otimes Z)}^{-1}\left( \phi _{X,Y\otimes Z}^{2}\right) ^{-1}\omega
\left( t_{X,Y\otimes Z}\right) \phi _{X,Y\otimes Z}^{2}\left( \omega \left(
X\right) \otimes \phi _{Y,Z}^{2}\right)  \\
&=&\left( \omega \left( X\right) \otimes \left( \phi _{Y,Z}^{2}\right)
^{-1}\right) \left( \phi _{X,Y\otimes Z}^{2}\right) ^{-1}\omega \left(
t_{X,Y\otimes Z}\right) \phi _{X,Y\otimes Z}^{2}\left( \omega \left(
X\right) \otimes \phi _{Y,Z}^{2}\right)  \\
&=&\left[ \phi _{X,Y\otimes Z}^{2}\left( \omega \left( X\right) \otimes \phi
_{Y,Z}^{2}\right) \right] ^{-1}\omega \left( t_{X,Y\otimes Z}\right) \left[
\phi _{X,Y\otimes Z}^{2}\left( \omega \left( X\right) \otimes \phi
_{Y,Z}^{2}\right) \right]
\end{eqnarray*}%
\begin{eqnarray*}
&&\left( r_{\omega (X)\otimes \omega (Y)}\otimes \omega (Z)\right) \left(
\theta _{\Bbbk }^{2}(\chi )_{X,Y}\otimes \omega (Z)\right)  \\
&=&\left( r_{\omega (X)\otimes \omega (Y)}\otimes \omega (Z)\right) \left(
r_{\omega (X)\otimes \omega (Y)}^{-1}\otimes \omega (Z)\right) \left( \left(
\phi _{X,Y}^{2}\right) ^{-1}\otimes \omega (Z)\right) \left( \omega \left(
t_{X,Y}\right) \otimes \omega (Z)\right) \left( \phi _{X,Y}^{2}\otimes
\omega (Z)\right)  \\
&=&\left( \left( \phi _{X,Y}^{2}\right) ^{-1}\otimes \omega (Z)\right)
\left( \phi _{X\otimes Y,Z}^{2}\right) ^{-1}\omega \left( t_{X,Y}\otimes
Z\right) \phi _{X\otimes Y,Z}^{2}\left( \phi _{X,Y}^{2}\otimes \omega
(Z)\right)  \\
&=&\left[ \phi _{X,Y\otimes Z}^{2}\left( \omega \left( X\right) \otimes \phi
_{Y,Z}^{2}\right) \right]^{-1} \omega \left( t_{X,Y}\otimes Z\right) \left[ \phi
_{X,Y\otimes Z}^{2}\left( \omega \left( X\right) \otimes \phi
_{Y,Z}^{2}\right) \right]
\end{eqnarray*}%
and%
\begin{eqnarray*}
&&\left[
\begin{array}{c}
\left( r_{\omega \left( X\right) \otimes \omega \left( Y\right) }\otimes
\omega \left( Z\right) \right) \left( \theta _{\Bbbk }^{2}(\Rr%
^{-1})_{X,Y}\otimes \omega \left( Z\right) \right) \left( \tau _{\omega
\left( Y\right) ,\omega \left( X\right) }\otimes \omega \left( Z\right)
\right)  \\
r_{\omega \left( Y\right) \otimes \omega \left( X\right) \otimes \omega
\left( Z\right) }\left( \omega \left( Y\right) \otimes \theta _{\Bbbk
}^{2}(\chi )_{X,Z}\right) \left( \tau _{\omega \left( X\right) ,\omega
\left( Y\right) }\otimes \omega \left( Z\right) \right) \\
\left( r_{\omega
\left( X\right) \otimes \omega \left( Y\right) }\otimes \omega \left(
Z\right) \right) \left( \theta _{\Bbbk }^{2}(\Rr)_{X,Y}\otimes
\omega \left( Z\right) \right)
\end{array}%
\right]  \\
&=&\left[
\begin{array}{c}
\left( r_{\omega \left( X\right) \otimes \omega \left( Y\right) }\otimes
\omega \left( Z\right) \right)  
\left( r_{\omega (X)\otimes \omega (Y)}^{-1}\otimes \omega \left( Z\right)
\right) \left( \left( \phi _{X,Y}^{2}\right) ^{-1}\otimes \omega \left(
Z\right) \right) \left( \omega \left( \sigma _{X,Y}^{-1}\right) \otimes
\omega \left( Z\right) \right) \\
\left( \phi _{Y,X}^{2}\otimes \omega \left(
Z\right) \right) \left( \tau _{\omega \left( X\right) ,\omega \left(
Y\right) }\otimes \omega \left( Z\right) \right)  
\left( \tau _{\omega \left( Y\right) ,\omega \left( X\right) }\otimes \omega
\left( Z\right) \right)  \\
r_{\omega \left( Y\right) \otimes \omega \left( X\right) \otimes \omega
\left( Z\right) }\left( \omega \left( Y\right) \otimes r_{\omega (X)\otimes
\omega (Z)}^{-1}\right) \left( \omega \left( Y\right) \otimes \left( \phi
_{X,Z}^{2}\right) ^{-1}\right)\\
 \left( \omega \left( Y\right) \otimes \omega
\left( t_{X,Z}\right) \right) \left( \omega \left( Y\right) \otimes \phi
_{X,Z}^{2}\right)  
\left( \tau _{\omega \left( X\right) ,\omega \left( Y\right) }\otimes \omega
\left( Z\right) \right) \left( r_{\omega \left( X\right) \otimes \omega
\left( Y\right) }\otimes \omega \left( Z\right) \right)  \\
\left( r_{\omega (X)\otimes \omega (Y)}^{-1}\otimes \omega \left( Z\right)
\right) \left( \tau _{\omega \left( Y\right) ,\omega \left( X\right)
}\otimes \omega \left( Z\right) \right) \left( \left( \phi _{Y,X}^{2}\right)
^{-1}\otimes \omega \left( Z\right) \right) \\\left( \omega \left( \sigma
_{X,Y}\right) \otimes \omega \left( Z\right) \right) \left( \phi
_{X,Y}^{2}\otimes \omega \left( Z\right) \right)
\end{array}%
\right]  \\
&=&\left[
\begin{array}{c}
\left( \left( \phi _{X,Y}^{2}\right) ^{-1}\otimes \omega \left( Z\right)
\right) \left( \omega \left( \sigma _{X,Y}^{-1}\right) \otimes \omega \left(
Z\right) \right) \left( \phi _{Y,X}^{2}\otimes \omega \left( Z\right)
\right)  \\
\left( \omega \left( Y\right) \otimes \left( \phi _{X,Z}^{2}\right)
^{-1}\right) \left( \omega \left( Y\right) \otimes \omega \left(
t_{X,Z}\right) \right) \left( \omega \left( Y\right) \otimes \phi
_{X,Z}^{2}\right)  \\
\left( \left( \phi _{Y,X}^{2}\right) ^{-1}\otimes \omega \left( Z\right)
\right) \left( \omega \left( \sigma _{X,Y}\right) \otimes \omega \left(
Z\right) \right) \left( \phi _{X,Y}^{2}\otimes \omega \left( Z\right)
\right)
\end{array}%
\right]  \\
&=&\left[
\begin{array}{c}
\left( \left( \phi _{X,Y}^{2}\right) ^{-1}\otimes \omega \left( Z\right)
\right) \left( \phi _{X\otimes Y,Z}^{2}\right) ^{-1}\omega \left( \sigma
_{X,Y}^{-1}\otimes Z\right) \phi _{Y\otimes X,Z}^{2}\left( \phi
_{Y,X}^{2}\otimes \omega \left( Z\right) \right)  \\
\left( \omega \left( Y\right) \otimes \left( \phi _{X,Z}^{2}\right)
^{-1}\right) \left( \phi _{Y,X\otimes Z}^{2}\right) ^{-1}\omega \left(
Y\otimes t_{X,Z}\right) \phi _{Y,X\otimes Z}^{2}\left( \omega \left(
Y\right) \otimes \phi _{X,Z}^{2}\right)  \\
\left( \left( \phi _{Y,X}^{2}\right) ^{-1}\otimes \omega \left( Z\right)
\right) \left( \phi _{Y\otimes X,Z}^{2}\right) ^{-1}\omega \left( \sigma
_{X,Y}\otimes Z\right) \phi _{X\otimes Y,Z}^{2}\left( \phi _{X,Y}^{2}\otimes
\omega \left( Z\right) \right)
\end{array}%
\right]  \\
&=&\left[
\begin{array}{c}
\left[ \phi _{X,Y\otimes Z}^{2}\left( \omega \left( X\right) \otimes \phi
_{Y,Z}^{2}\right) \right] ^{-1}\omega \left( \sigma _{X,Y}^{-1}\otimes
Z\right) \phi _{Y\otimes X,Z}^{2}\left( \phi _{Y,X}^{2}\otimes \omega \left(
Z\right) \right)  \\
\left[ \phi _{Y\otimes X,Z}^{2}\left( \phi _{Y,X}^{2}\otimes \omega \left(
Z\right) \right) \right] ^{-1}\omega \left( Y\otimes t_{X,Z}\right) \left[
\phi _{Y\otimes X,Z}^{2}\left( \phi _{Y,X}^{2}\otimes \omega \left( Z\right)
\right) \right]  \\
\left[ \phi _{Y\otimes X,Z}^{2}\left( \phi _{Y,X}^{2}\otimes \omega \left(
Z\right) \right) \right] ^{-1}\omega \left( \sigma _{X,Y}\otimes Z\right) %
\left[ \phi _{X,Y\otimes Z}^{2}\left( \omega \left( X\right) \otimes \phi
_{Y,Z}^{2}\right) \right]
\end{array}%
\right]  \\
&=&\left[ \phi _{X,Y\otimes Z}^{2}\left( \omega \left( X\right) \otimes \phi
_{Y,Z}^{2}\right) \right] ^{-1}\omega \left( \sigma _{X,Y}^{-1}\otimes
Z\right) \omega \left( Y\otimes t_{X,Z}\right)  \omega \left( \sigma
_{X,Y} \otimes  Z \right) \left[ \phi
_{X,Y\otimes Z}^{2}\left( \omega \left( X\right) \otimes \phi
_{Y,Z}^{2}\right) \right]
\end{eqnarray*}%
so that the equality we have to check becomes%
\begin{eqnarray*}
&&\left[ \phi _{X,Y\otimes Z}^{2}\left( \omega \left( X\right) \otimes \phi
_{Y,Z}^{2}\right) \right] ^{-1}\omega \left( t_{X,Y\otimes Z}\right) \left[
\phi _{X,Y\otimes Z}^{2}\left( \omega \left( X\right) \otimes \phi
_{Y,Z}^{2}\right) \right]  \\
&=&\left[ \phi _{X,Y\otimes Z}^{2}\left( \omega \left( X\right) \otimes \phi
_{Y,Z}^{2}\right) \right]^{-1} \omega \left( t_{X,Y}\otimes Z\right) \left[ \phi
_{X,Y\otimes Z}^{2}\left( \omega \left( X\right) \otimes \phi
_{Y,Z}^{2}\right) \right] + \\
&&+\left[ \phi _{X,Y\otimes Z}^{2}\left( \omega \left( X\right) \otimes \phi
_{Y,Z}^{2}\right) \right] ^{-1}\omega \left( \sigma _{X,Y}^{-1}\otimes
Z\right) \omega \left( Y\otimes t_{X,Z}\right) \omega \left( \sigma
_{X,Y}\otimes Z\right) \left[ \phi _{X,Y\otimes Z}^{2}\left( \omega \left(
X\right) \otimes \phi _{Y,Z}^{2}\right) \right]
\end{eqnarray*}%
that is%
\begin{equation*}
\omega \left( t_{X,Y\otimes Z}\right) =\omega \left( t_{X,Y}\otimes
Z\right) +\omega \left[ \left( \sigma _{X,Y}^{-1}\otimes Z\right) \left(
Y\otimes t_{X,Z}\right) \left( \sigma _{X,Y}\otimes Z\right) \right] .
\end{equation*}%
Since $\omega \,$\ is additive, this equality is true by \eqref{qC-I}.

$ii)$ Similarly, \eqref{eq:CC3} amounts to the equality 
\begin{eqnarray*}
&&\left( \left( \phi _{X,Y}^{2}\right) ^{-1}\otimes \omega \left( Z\right)
\otimes \Bbbk \right) \theta _{\Bbbk }^{2}(\chi )_{X\otimes Y,Z}\left( \phi
_{X,Y}^{2}\otimes \omega \left( Z\right) \right)  \\
&=&\omega \left( X\right) \otimes \theta _{\Bbbk }^{2}(\chi )_{Y,Z}+ \\
&&+\left[
\begin{array}{c}
\left( \omega \left( X\right) \otimes \theta _{\Bbbk }^{2}(\Rr%
^{-1})_{Y,Z}\right) \left( \omega \left( X\right) \otimes \tau _{\omega
\left( Z\right) ,\omega \left( Y\right) }\right) \left( r_{\omega \left(
X\right) \otimes \omega \left( Z\right) }\otimes \omega \left( Y\right)
\right)  \\
\left( \theta _{\Bbbk }^{2}(\chi )_{X,Z}\otimes \omega \left( Y\right)
\right) \left( \omega \left( X\right) \otimes \tau _{\omega \left( Y\right)
,\omega \left( Z\right) }\right) r_{\omega \left( X\right) \otimes \omega
\left( Y\right) \otimes \omega \left( Z\right) }\left( \omega \left(
X\right) \otimes \theta _{\Bbbk }^{2}(\Rr)_{Y,Z}\right)
\end{array}%
\right] .
\end{eqnarray*}%
which follows from \eqref{qC-II} and the fact that $\omega $ is additive.

$iii)$ One checks that \eqref{eq:CC4} is equivalent to
\begin{equation*}
\theta _{\Bbbk }^{2}(\Rr)_{X,Y}r_{\omega (X)\otimes \omega
(Y)}\theta _{\Bbbk }^{2}(\chi )_{X,Y}=\left( \tau _{\omega \left( Y\right)
,\omega \left( X\right) }\otimes \Bbbk \right) \theta _{\Bbbk }^{2}(\chi
)_{Y,X}\tau _{\omega \left( X\right) ,\omega \left( Y\right) }r_{\omega
(X)\otimes \omega (Y)}\theta _{\Bbbk }^{2}(\Rr)_{X,Y}.
\end{equation*}%
By starting from the right-hand side and by using \eqref{cart}, we compute%
\begin{eqnarray*}
&&\left( \tau _{\omega \left( Y\right) ,\omega \left( X\right) }\otimes
\Bbbk \right) \theta _{\Bbbk }^{2}(\chi )_{Y,X}\tau _{\omega \left( X\right)
,\omega \left( Y\right) }r_{\omega (X)\otimes \omega (Y)}\theta _{\Bbbk
}^{2}(\Rr)_{X,Y} \\
&=&\left[
\begin{array}{c}
\left( \tau _{\omega \left( Y\right) ,\omega \left( X\right) }\otimes
\Bbbk \right) r_{\omega (Y)\otimes \omega (X)}^{-1}\left( \phi
_{Y,X}^{2}\right) ^{-1}\omega \left( t_{Y,X}\right) \phi _{Y,X}^{2}\\
\tau_{\omega \left( X\right) ,\omega \left( Y\right) }r_{\omega (X)\otimes
\omega (Y)}r_{\omega (X)\otimes \omega (Y)}^{-1}\tau _{\omega \left(
Y\right) ,\omega \left( X\right) }\left( \phi _{Y,X}^{2}\right) ^{-1}\omega
\left( \sigma _{X,Y}\right) \phi _{X,Y}^{2} 
\end{array}\right]
\\
&=&\left( \tau _{\omega \left( Y\right) ,\omega \left( X\right) }\otimes
\Bbbk \right) r_{\omega (Y)\otimes \omega (X)}^{-1}\left( \phi
_{Y,X}^{2}\right) ^{-1}\omega \left( t_{Y,X}\right) \omega \left( \sigma
_{X,Y}\right) \phi _{X,Y}^{2} \\
&\overset{\eqref{cart}  }=&\left( \tau _{\omega \left( Y\right) ,\omega \left( X\right) }\otimes
\Bbbk \right) r_{\omega (Y)\otimes \omega (X)}^{-1}\left( \phi
_{Y,X}^{2}\right) ^{-1}\omega \left( \sigma _{X,Y}\right) \omega \left(
t_{X,Y}\right) \phi _{X,Y}^{2} \\
&=&r_{\omega (X)\otimes \omega (Y)}^{-1}\tau _{\omega \left( Y\right)
,\omega \left( X\right) }\left( \phi _{Y,X}^{2}\right) ^{-1}\omega \left(
\sigma _{X,Y}\right) \omega \left( t_{X,Y}\right) \phi _{X,Y}^{2} \\
&=&r_{\omega (X)\otimes \omega (Y)}^{-1}\tau _{\omega \left( Y\right)
,\omega \left( X\right) }\left( \phi _{Y,X}^{2}\right) ^{-1}\omega \left(
\sigma _{X,Y}\right) \phi _{X,Y}^{2}r_{\omega (X)\otimes \omega
(Y)}r_{\omega (X)\otimes \omega (Y)}^{-1}\left( \phi _{X,Y}^{2}\right)
^{-1}\omega \left( t_{X,Y}\right) \phi _{X,Y}^{2} \\
&=&\theta _{\Bbbk }^{2}(\Rr)_{X,Y}r_{\omega (X)\otimes \omega
(Y)}\theta _{\Bbbk }^{2}(\chi )_{X,Y}
\end{eqnarray*} which completes the proof.
\end{proof}

\subsection*{Acknowledgements}
This paper was written while AA, LB and AS were members of the
``National Group for Algebraic and Geometric Structures and their
Applications'' (GNSAGA-INdAM). The work of TW was partially supported by INFN Sezione di Torino while all authors were partially supported by ``Ministero dell’Università e della Ricerca'' within the National Research Project PRIN 2017 ``\emph{Categories, Algebras: Ring-Theoretical and Homological Approaches} (CARTHA)''.

The authors would like to thank P. Aschieri, D. Ferri and R. Fioresi for valuable comments on the topic. Moreover, they express their gratitude to the referee for meaningful suggestions.

\end{document}